\documentclass[12pt, a4paper, reqno]{amsart}

\usepackage[T1]{fontenc} 

\usepackage{graphics}
\DeclareGraphicsExtensions{.eps}

\usepackage{epstopdf}
\usepackage{ulem} 

\usepackage{a4wide}
\usepackage[margin=2cm]{geometry}

\makeatletter
\newcommand{\addresseshere}{%
  \enddoc@text\let\enddoc@text\relax
}
\makeatother

\makeatletter
\newcommand{\eqlab}[1]{\leavevmode\hfill\refstepcounter{equation}\label{#1}\textup{\tagform@{\theequation}}}
\makeatother

\usepackage{amsmath, amsthm}
\allowdisplaybreaks

\makeatletter
\@ifpackageloaded{amssymb}{}{\usepackage[bitstream-charter]{mathdesign}}
\makeatother

\usepackage{graphicx}
\usepackage{caption}
\usepackage{color}
\usepackage{enumitem}

\usepackage[pdftex, bookmarksopen, pagebackref,
            pdftitle={Numerical methods for the deterministic second moment equation of parabolic stochastic PDEs},
            pdfauthor={K. Kirchner},
            colorlinks, linkcolor=black, urlcolor=black, citecolor=black,
]{hyperref}

\usepackage[boxslash]{stmaryrd}
\SetSymbolFont{stmry}{bold}{U}{stmry}{m}{n}

\usepackage{multirow}

\newcommand{\T}{\mathsf{T}}

\newcommand{\GL}{\text{GL}} 
\newcommand{\GR}{\text{GR}^{\leftarrow}} 

\newcommand{\IP}{\mathbb{P}}
\newcommand{\IR}{\mathbb{R}}
\newcommand{\IN}{\mathbb{N}}

\newcommand{\norm}[2]{     \| #1       \|_{ #2 }}

\newcommand{\normiii}[2]{\vert\kern-0.25ex\vert\kern-0.25ex\vert #1 \vert\kern-0.25ex \vert\kern-0.25ex\vert_{ #2 }}
\newcommand{\Normiii}[2]{\left\vert\kern-0.25ex\left\vert\kern-0.25ex\left\vert #1 \right\vert\kern-0.25ex\right\vert\kern-0.25ex\right\vert_{ #2 }}
\newcommand{\scalar}[2]{     ( #1       )_{ #2 }}

\newcommand{\duality}[2]{     \langle #1       \rangle_{ #2 }}

\newcommand{\bbE}{\mathbb{E}}
\newcommand{\bbM}{\mathbb{M}}
\newcommand{\bbP}{\mathbb{P}}
\newcommand{\cA}{\mathcal{A}}
\newcommand{\cB}{\mathcal{B}}
\newcommand{\cD}{\mathcal{D}}
\newcommand{\cF}{\mathcal{F}}
\newcommand{\cG}{\mathcal{G}}

\newcommand{\cL}{\mathcal{L}}
\newcommand{\cN}{\mathcal{N}}

\newcommand{\cS}{\mathcal{S}}
\newcommand{\cT}{\mathcal{T}}
\newcommand{\cU}{\mathcal{U}}

\newcommand{\cX}{\mathcal{X}}
\newcommand{\cY}{\mathcal{Y}}
\newcommand{\rd}{\,\mathrm{d}}

\newcommand{\Dmat}{\mathbf{D}}
\newcommand{\Gmat}{\mathbf{G}}
\newcommand{\Lmat}{\mathbf{L}}
\newcommand{\Umat}{\mathbf{U}}

\newcommand{\Wmat}{\mathbf{W}}
\newcommand{\Thetamat}{\boldsymbol{\Theta}}
\newcommand{\Lambdamat}{\boldsymbol{\Lambda}}
\newcommand{\tr}{\operatorname{tr}}

\newcommand{\Cov}{\operatorname{Cov}}
\newcommand{\Var}{\operatorname{Var}}
\newcommand{\E}[1]{{\mathbb{E}\!}\left[ #1 \right]}
\newcommand{\e}[1]{{\mathbb{E}}[ #1 ]}

\newcommand{\AND}{\quad\text{and}\quad}
\newcommand{\TEXT}[1]{\quad\text{#1}\quad}

\let\nativecdot\cdot
\usepackage{marginnote}
\newcommand{\CDOT}{\,\nativecdot\,}
\renewcommand{\cdot}{{\nativecdot\smash{\marginnote{\text{\color{red}CDOT}}}}}

\newcommand{\IntervalCO}[1]{[#1)}
\newcommand{\boxrule}{\reflectbox{$\boxslash$}}

\let\nativedelta\delta
\let\nativeDelta\Delta
\usepackage{marginnote}
\renewcommand{\delta}{{\nativedelta\smash{\marginnote{\text{\color{red}$\nativedelta$ illegal}}}}}
\renewcommand{\Delta}{{\nativeDelta\smash{\marginnote{\text{\color{red}$\nativeDelta$ illegal}}}}}
\newcommand{\Dirac}{\text{\dh}}
\newcommand{\diag}{\nativedelta}
\newcommand{\DIAG}{\nativeDelta}
\newcommand{\krondelta}{\nativedelta}

\newcommand{\from}{\colon}

\newtheorem{lemma}{Lemma}[section]
\newtheorem{proposition}[lemma]{Proposition}
\newtheorem{theorem}[lemma]{Theorem}
\newtheorem{corollary}[lemma]{Corollary}
\newtheorem{remark}[lemma]{Remark}
\theoremstyle{remark}
\theoremstyle{definition}

\newtheorem{example}[lemma]{Example}

\definecolor{somegree}{rgb}{0,.6,0}
\definecolor{someblue}{rgb}{0,0,.8}
\definecolor{toverify}{rgb}{1,0,0}
\definecolor{lila}{RGB}{128,0,255}

\def\[#1\]{%
  \begin{align}#1\end{align}%
}

\begin{document}

\title%
[Numerics for the second moment equation of SPDEs]%
{Numerical methods for the deterministic second moment equation of parabolic stochastic PDEs}

\author{Kristin Kirchner}
\date{\today}

\address[K.~Kirchner]{Department of Mathematical Sciences, Chalmers University of Technology and University of Gothenburg, SE-412 96 Gothenburg, Sweden}
\email{kristin.kirchner@chalmers.se}

\begin{abstract}
	Numerical methods for stochastic partial differential equations
	typically
	estimate moments of the solution
	from sampled paths.
	Instead,
	we shall directly target
	the deterministic equations
	satisfied by the first and second moments, as well as the covariance.

	In the first part, we focus on stochastic ordinary differential equations.
	For the canonical examples
	with additive noise (Ornstein--Uhlenbeck process)
	or
	multiplicative noise (geometric Brownian motion)
	we derive these deterministic equations
	in variational form
	and
	discuss their well-posedness in detail.
	Notably,
    the second moment equation
	in the multiplicative case
    is naturally posed on
	projective--injective tensor product spaces
	as trial--test spaces.
	We construct Petrov--Galerkin discretizations
	based on tensor product piecewise polynomials
	and analyze their stability and convergence
	in these natural norms.

	In the second part,
	we proceed with
	parabolic stochastic partial differential equations
	with affine multiplicative noise.
	We prove well-posedness of the deterministic variational problem
	for the second moment,
	improving an earlier result.
	We then propose conforming space-time Petrov--Galerkin discretizations,
    which we show to be stable and quasi-optimal.

    In both parts, the outcomes are illustrated by numerical examples.
\end{abstract}




\maketitle

\section{Introduction}
\label{s:0}

\subsection{Introduction}

Ordinary and partial differential equations
are pervasive
in
financial, biological, engineering and social sciences,
to name a few.
Often,
randomness
is introduced
in order to model
uncertainties
in the coefficients,
in the geometry of the physical domain,
in the boundary or initial conditions,
or in the sources (right-hand sides).
In this work we aim at the latter scenario,
specifically
ordinary or partial differential evolution equations
driven by additive or multiplicative noise.
The random solution
is then a stochastic process
with values in a certain state space.
If the noise is a Wiener process,
the solution paths are continuous in time.
When the state space is
of finite dimension ($\leq 3$, say),
it may be possible
to
approximate numerically
the temporal evolution of
the probability density function
of
the stochastic process.
For stochastic PDEs, this is in general computationally too expensive.
One therefore estimates
the mean and possibly the covariance
of the solution process,
also given by
its first two statistical moments.

To estimate moments
of the random solution
one can resort to sampling methods
such
as Monte Carlo (MC).
For every sample path,
viz.~a realization of the random input,
a deterministic
ordinary or partial
differential evolution equation
is solved.
The vanilla MC
exhibits the notorious convergence rate $1/2$
in the number of samples.
On the upside,
sampling methods are usually trivial to parallelize
across samples.
Recent developments
include multilevel MC
\cite{BarthLangSchwab2013, Charrier2013, Cliffe2011, Giles2008, Giles2015, TeckentrupScheichlGilesUllmann2013},
quasi-MC \cite{DickKuoSloan2013, Graham2011, GrahamKuoNicholsScheichlSchwabSloan2015, KuoSchwabSloan2012},
and combinations thereof~\cite{Giles2009, KuoSchwabSloan2015}.
More
on solving random and parametric equations
can be found in
\cite{BachmayrSchneiderUschmajew2016, BennerGugercinWillcox2015, CohenDeVore2015, GrasedyckKressnerTobler2013, SchwabGittelson2011}.

For the covariance
of
the solution to a parabolic stochastic PDE
driven by
additive Wiener noise,
an alternative to sampling
was proposed
in \cite{LangLarssonSchwab2013}.
It is based
on the insight
that the second moment
solves
a well-posed
linear
deterministic space-time variational problem
on
Hilbert tensor products of Bochner spaces.
The main promise
of space-time variational formulations
is
in potential computing time and memory savings
through
space-time compressive schemes,
e.g., using
adaptive wavelet methods
\cite{Stevenson2009}
or
low-rank tensor approximations
\cite{BachmayrSchneiderUschmajew2016, GrasedyckKressnerTobler2013, Hackbusch2014}.
In principle,
it is straightforward
to construct numerical methods
for the formulation from \cite{LangLarssonSchwab2013}
by tensorizing
existing
discretizations
of
deterministic parabolic evolution equations (space-time or not),
the main practical issue
being
the high dimensionality of the resulting equations.

The space-time variational formulation from \cite{LangLarssonSchwab2013}
was extended in \cite{KirchnerLangLarsson2016}
to include
multiplicative L\'evy noise.
This
required
a more careful analysis
because
firstly,
an extra term
in
the space-time variational formulation
constrains it
to
non-Hilbert
tensor product spaces
for the trial and test spaces;
secondly,
the well-posedness
is
self-evident
only as long as
the volatility of the multiplicative noise
is sufficiently small.
Consequently,
contrary to the additive case,
a dedicated design and analysis of numerical schemes is required for
stochastic PDEs with multiplicative noise.
To fully explain and address
these issues,
in this work
we first focus
on canonical examples
of stochastic ODEs
driven by
additive
or
multiplicative
Wiener noise.
To facilitate the transition and the comparison
to parabolic stochastic PDEs,
our estimates
are explicit and sharp in
the relevant parameters.
We then proceed with parabolic
stochastic PDEs driven by
multiplicative L\'evy noise as in \cite{KirchnerLangLarsson2016}.
The transition from convolutions of real-valued functions
to semigroup theory on tensor product spaces
allows us to prove well-posedness
of the deterministic second moment equation
also in the vector-valued situation
even beyond the smallness assumption
on the multiplicative noise term
made in~\cite[Eq.~(5.5)]{KirchnerLangLarsson2016}.

This article is structured as follows.
In \S\ref{s:sode}
we introduce the model stochastic ODEs
and the necessary definitions,
derive the deterministic equations
for
the first and second moments
and
discuss their well-posedness.
In \S\ref{s:dis}
we
present conforming
Petrov--Galerkin discretizations
of these equations
and
discuss their stability,
concluding with a numerical example.
In \S\ref{s:pde}
we generalize the
results of \S\S\ref{s:sode}--\ref{s:dis}
to stochastic PDEs
with affine multiplicative noise
and, again, verify
these
by numerical experiments.
The outcomes of this work
are summarized in \S\ref{s:end}.

\subsection{Notation}\label{s:notation}

We briefly comment on notation.
If $X$ is a Banach space then $S(X)$ denotes its unit sphere
and $X'$ its dual, i.e, all linear
continuous mappings from $X$ to $\IR$.
We write $s \wedge t := \min\{ s, t \}$.
The symbol $\Dirac$ ($\Dirac_s$) denotes the Dirac measure (at $s$).
The closure of an interval $J$ is $\bar{J}$.
We mark equations which hold almost everywhere
or $\bbP$-almost surely with a.e.\ and $\bbP$-a.s., respectively.
The space of bounded linear operators $X \to Y$
is denoted by $\cL(X; Y)$;
those on $X$ by $\cL(X)$.

Depending on the context,
the symbol $\otimes$ denotes
the tensor product of two functions or operators,
the algebraic tensor product
of function spaces,
or the Kronecker product of matrices.

If $H$ is a Hilbert space
then
the Hilbert tensor product space
$
H_2 := H \otimes_2 H
$
is
obtained as the closure of
the algebraic tensor product $H \otimes H$
under
the norm $\norm{\CDOT}{2}$ induced by
the ``tensorized'' inner product
$
\scalar{ a \otimes b, c \otimes d }{2}
:=
\scalar{ a, c }{H}
\scalar{ b, d }{H}
$.

A function $w \in L_2(J \times J)$
is called
symmetric positive semi-definite (SPSD)
if
\[
	\label{e:JSPSD}
	w(s, t) = w(t, s)
	\quad
	\text{a.e.\ in } J\times J
	\TEXT{and}
	\int_J \int_J
	w(s, t)
	\varphi(s)
	\varphi(t)
	\rd s
	\rd t
	\geq 0
	\quad
	\forall
	\varphi \in L_2(J)
	.
\]
More generally,
if $H$ is a Hilbert space (we have $H = L_2(J)$ in \eqref{e:JSPSD}, cf.~\eqref{e:L2}),
then
an element $w$ of the Hilbert tensor product $H \otimes_2 H$
is symmetric and positive semi-definite,
abbreviated as $H$-SPSD,
if
\[
	\label{e:HSPSD}
	\scalar{ w, \varphi \otimes \tilde\varphi }{2}
	=
	\scalar{ w, \tilde\varphi \otimes \varphi }{2}
	\TEXT{and}
	\scalar{ w, \varphi \otimes \varphi }{2} \geq 0
	\qquad
	\forall
	\varphi, \tilde\varphi \in H
	.
\]
It is called symmetric if the equality in~\eqref{e:HSPSD} holds,
and antisymmetric if
$
	\scalar{ w, \varphi \otimes \tilde\varphi }{2}
	=
	-\scalar{ w, \tilde\varphi \otimes \varphi }{2}
	.
$
A functional $\ell$ defined on some closure of $H \otimes H$
is
called
symmetric positive semi-definite (SPSD)
if
\[
	\label{e:psd}
	\ell( \psi \otimes \tilde\psi )
	=
	\ell( \tilde\psi \otimes \psi )
	\TEXT{and}
	\ell( \psi \otimes \psi ) \geq 0
	\qquad
	\forall
	\psi, \tilde\psi \in H
	.
\]
It is called
antisymmetric if
$\ell( \psi \otimes \tilde\psi ) = -\ell( \tilde\psi \otimes \psi )$.
If \eqref{e:psd} holds only on a subset $\psi, \tilde\psi \in V \subset H$,
we say that $\ell$ is SPSD on $V\otimes V$ for short.

\pagebreak

\section{Derivation of the deterministic moment equations}
\label{s:sode}

\subsection{Model stochastic ODEs}
\label{s:sode:model}

Let $T > 0$, set $J := (0, T)$.
The first part of this article focusses on
the model real-valued stochastic ODEs
with {additive} noise
\[
	\label{e:sode-add}
	\rd X(t) + \lambda X(t) \rd t = \mu \rd W(t),
	\quad t \in \bar{J},
	\TEXT{with}
	X(0) = X_0
	,
\]
or with {multiplicative} noise
\[
	\label{e:sode-mul}
	\rd X(t) + \lambda X(t) \rd t = \rho X(t) \rd W(t),
	\quad t \in \bar{J},
	\TEXT{with}
	X(0) = X_0
	.
\]
Here,
\begin{itemize}
	\item
		$\lambda > 0$ is a fixed positive number
		that models the action of an elliptic operator,
	\item
		$W$
		is a
		real-valued
		Wiener process defined
		on a
		complete
		probability space $(\Omega,\cA,\IP)$,
	\item
		$\mu,\rho \geq 0$
		are parameters
		specifying the volatility
		of the noise,
	\item
		the initial value
		$X_0 \in L_2(\Omega)$
		is
		a random variable, independent of the Wiener process,
		with known first and second moments
		(but not necessarily with a known distribution).
\end{itemize}
We call $\cF_t$ the $\sigma$-algebra
generated by the Wiener process $\{ W(s) : 0 \leq s \leq t \}$
and the initial value $X_0$,
and $\cF$ the resulting filtration.
The expectation operator is denoted by $\bbE$.
We refer to~\cite{Kloeden1992, Oksendal2013}
for basic notions of stochastic
integration and It\^o calculus.

A real-valued stochastic process $X$ is said to be
a (continuous strong) solution 
of
the stochastic differential equation
``$\rd X + \lambda X = \sigma(X) \rd W$ on $\bar{J}$ with $X(0) = X_0$''
if
\textbf{a)}
$X$ is progressively measurable with respect to $\cF$,
\textbf{b)}
the expectation of
$\norm{\lambda X}{L_1(J)} + \norm{\sigma(X)}{L_2(J)}^2$
is finite,
\textbf{c)}
the integral equation
\begin{equation*}
	X(t) =  X_0 - \lambda \int_0^t X(s) \rd s + \int_0^t \sigma(X(s)) \rd W(s)
	\quad
	\forall
	t \in \bar{J}
\end{equation*}
holds ($\IP\text{-a.s.}$),
and
\textbf{d)}
$t \mapsto X(t)$ is continuous ($\IP\text{-a.s.}$).
By standard theory (\cite[Thm.~4.5.3]{Kloeden1992} or \cite[Thm.~5.2.1]{Oksendal2013})
a Lipschitz condition on $\sigma$ implies
existence and uniqueness of such a solution.
Moreover, it has finite second moments.
For future reference,
we state here the integral equations for \eqref{e:sode-add}--\eqref{e:sode-mul}:
\begin{align}
	\label{e:X-add}
	X(t) & =
	X_0 - \lambda \int_0^t X(s) \rd s + \mu \int_0^t \rd W(s)
	     &
	\forall t \in \bar{J}
	\quad
	(\IP\text{-a.s.})
	,
	\\
	\label{e:X-mul}
	X(t) & =
	X_0 - \lambda \int_0^t X(s) \rd s + \rho \int_0^t X(s) \rd W(s)
	     &
	\forall t \in \bar{J}
	\quad
	(\IP\text{-a.s.})
	.
\end{align}
The solution processes
and
their first/second moments
are known explicitly (e.g.~\cite[\S4.4]{Kloeden1992}):
\begin{subequations}
	\begin{center}
		\label{e:X-table}
		\begin{tabular}{ll|c|c}
			&
			&
			Additive \eqref{e:sode-add}/\eqref{e:X-add}
			&
			Multiplicative \eqref{e:sode-mul}/\eqref{e:X-mul}
			\\
			&
			&
			\small
			Ornstein--Uhlenbeck process
			&
			\small
			Geometric Brownian motion
			\\
			\hline
			\eqlab{e:X-table-Xt}
			&
			$X(t)$
			&
			$e^{-\lambda t} X_0 + \mu \int_0^t e^{-\lambda(t-s)} \rd W(s)$
			&
			$X_0 e^{-(\lambda + \rho^2/2) t + \rho W(t)}$
			\\
			\hline
			\eqlab{e:X-table-EXt}
			&
			$\e{ X(t) }$
			&
			$e^{-\lambda t} \e{X_0}$
			&
			$e^{-\lambda t} \e{X_0}$
			\\
			\hline
			\eqlab{e:X-table-EXsXt}
			&
			$\e{ X(s) X(t) }$
			&
			$
				e^{-\lambda(t+s)}
				\e{ X_0^2 }
				+
				\tfrac{\mu^2}{2\lambda}
				( e^{-\lambda|t - s|} - e^{-\lambda(t+s)} )
			$
			&
			$e^{-\lambda(t+s) + \rho^2 (s\wedge t)} \e{ X_0^2 }$
			\\
			\hline
			\eqlab{e:X-table-EX2}
			&
			$\e{ \norm{ X }{L_2(J)}^2 }$
			&
			$
				\tfrac{1 - e^{-2 \lambda T}}{2 \lambda}
				\e{X_0^2}
				+
				\tfrac{\mu^2}{4 \lambda^2}
				(
				e^{-2 \lambda T} + 2 \lambda T - 1
				)
			$
			&
			$
				\tfrac{
				e^{(\rho^2 - 2\lambda)T} - 1
				}{
				\rho^2 - 2\lambda
				}
				\e{X_0^2}
			$
		\end{tabular}
	\end{center}
\end{subequations}

The square integrability
\eqref{e:X-table-EX2}
in conjunction with Fubini's theorem
will be used to interchange the order of integration over $J$ and $\Omega$
without further mention.
Square integrability
also implies
the useful martingale property
(\cite[Thm.~3.2.5]{Kloeden1992} or
\cite[Cor.~3.2.6 \& Def.~3.1.4]{Oksendal2013})
\[
	\label{e:mul-M}
	\E{ \int_0^t X(r) \rd W(r) \Bigl| \cF_s }
	=
	\int_0^s X(r) \rd W(r),
	\quad
	0 \leq s \leq t
	.
\]
Choosing $s = 0$ shows
that the stochastic integral
$\int_0^t X(r) \rd W(r)$
has expectation zero.
If
$Y_1$ and $Y_2$
are
two square integrable processes
adapted to
$\cF$,
the It\^o isometry
(\cite[Thm.~3.2.3]{Kloeden1992} or \cite[Cor.~3.1.7]{Oksendal2013})
along with~\eqref{e:mul-M}
and the polarization identity
yield the equality
\[
	\label{e:itoiso}
	%
	\E{
	\int_0^s Y_1(r) \rd W(r)
	\int_0^t Y_2(r) \rd W(r)
	}
	=
	\int_0^{s \wedge t} \E{ Y_1(r) Y_2(r) } \rd r
	.
\]
These are the main tools in the derivation of \eqref{e:X-table}.
We will write $X \otimes X$ for
the real-valued stochastic process
$(s, t) \mapsto X(s) X(t)$ on $(\Omega, \cA, \bbP)$
indexed by the parameter space $J \times J$.

Our first aim will be to derive deterministic
equations for
the first and the second moments
\begin{equation*}
	%
	m(t) :=
	\e{ X(t) }
	\TEXT{and}
	M(s, t) :=
	\e{ X(s) X(t) }
	,
	\quad
	s, t \in J
	,
\end{equation*}
as well as for the covariance function
\[
	\label{e:CovX}
	\Cov(X) := \e{ (X - m) \otimes (X - m) }
	=
	M - (m \otimes m)
\]
of the stochastic process $X$.

In showing
well-posedness of the deterministic equations,
the notions \eqref{e:JSPSD}--\eqref{e:psd}
of positive semi-definiteness
will be important.
Indeed,
the second moment
and
the covariance of a real-valued stochastic process
are SPSD.
Importantly, the
SPSD functions form a cone, so that sums (and integrals) thereof remain SPSD.

%

\subsection{Deterministic first moment equations}
\label{s:sode:det}

We first introduce the spaces
\begin{equation*}
	%
	E := L_2(J)
	\TEXT{and}
	F := H^1_{0,\{T\}}(J)
	,
\end{equation*}
where the latter denotes the closed subspace of
the Sobolev space $H^1(J)$
of functions vanishing at $t = T$.
Thanks to the embedding $F \hookrightarrow C^0(\bar{J})$,
elements of $F$ will be identified
by their continuous representative.
These spaces are equipped with the $\lambda$-dependent norms
\[
	\label{e:EF-norms}
	\norm{w}{E}^2
	:=
	\lambda
	\norm{w}{L_2(J)}^2
	\TEXT{and}
	\norm{ v }{F}^2
	:=
	\lambda^{-1}
	\norm{v'}{L_2(J)}^2
	+
	\lambda
	\norm{v}{L_2(J)}^2
	+
	|v(0)|^2
	,
\]
and the obvious corresponding inner products
$\scalar{\CDOT, \CDOT}{E}$
and
$\scalar{\CDOT, \CDOT}{F}$.
The norm on $F$
is motivated by the fact that
\[
	\label{e:Fnorm-eq}
	\norm{v}{F}^2 = \lambda^{-1} \norm{-v' + \lambda v}{L_2(J)}^2
	\quad
	\forall v \in F
	.
\]

\begin{lemma}
	\label{lem:FC0}
	Let $v \in F$.
	Then
	\[
		\label{e:vt2}
		|v(t)| \leq \tfrac{1}{\sqrt{2}} \norm{v}{F}
		\quad \forall t \in \bar{J}
		.
	\]
\end{lemma}

\begin{proof}
	Suppose that the supremum of $|v(t)|$ is attained at some $0 \leq t \leq T$.
	Integrating $(v^2)' = 2 v v'$ over $(0,t)$,
	applying the Cauchy--Schwarz and the Young inequalities
	leads to the estimate
	$|v(t)|^2 \leq \lambda^{-1} \norm{v'}{}^2 + \lambda \norm{v}{}^2 + |v(0)|^2$
	in terms of the $L_2(0,t)$ norms.
	In a similar way, observing that $v(T) = 0$, we obtain
	$|v(t)|^2 \leq \lambda^{-1} \norm{v'}{}^2 + \lambda \norm{v}{}^2$
	in terms of the $L_2(t,T)$ norms.
	Adding the two inequalities gives \eqref{e:vt2}.
\end{proof}

The inequality \eqref{e:vt2} is sharp
as
the function
$\psi(t) := \sinh(\lambda(T - t)) / \sinh(\lambda T)$ attests:
\[
	\label{e:v0}
	1
	=
	\psi(0)
	=
	\sup_{t \in \bar{J}} |\psi(t)|
	\TEXT{and}
	\norm{\psi}{F} =
	\sqrt{\coth(\lambda T) + 1}
	\to
	\sqrt{2}
	\TEXT{as}
	\lambda T \to \infty
	.
\]

The deterministic moment equations
will be expressed
in terms of
the continuous bilinear form
\[
	\label{e:b}
	b \from E \times F \to \IR,
	\quad
	b(w, v) :=
	\int_J
	w(t)
	(-v'(t) + \lambda v(t))
	\rd t
	.
\]
We employ the same notation
for the induced bounded linear operator
\begin{equation*}
	b \from E \to F',
	\quad
	\duality{ b w, v }{}
	:=
	b(w, v)
	,
\end{equation*}
and use whichever is more convenient,
as should be evident from the context.
The operator $b$ arises in the weak formulation
of the ordinary differential equation $u' + \lambda u = f$.
With the definition of the norms \eqref{e:EF-norms},
it is an isometric isomorphism,
\[
	\label{e:b-iso}
	\norm{ b w }{F'} = \norm{w}{E}
	\quad
	\forall w \in E
	.
\]
Indeed,
$\norm{b w}{F'} \leq \norm{w}{E}$ is obvious from
\eqref{e:EF-norms}--\eqref{e:Fnorm-eq}.
To verify $\norm{b w}{F'} \geq \norm{w}{E}$,
let $w \in E$ be arbitrary.
Taking $v$ as the solution to the ODE
$-v' + \lambda v = \lambda w$ with $v(T) = 0$,
it follows
using \eqref{e:EF-norms}--\eqref{e:Fnorm-eq}
that
$\duality{ b w, v }{} = \norm{w}{E}^2 = \norm{v}{F}^2$.
Therefore,
$\duality{ b w, v }{} = \norm{w}{E} \norm{v}{F}$,
and in particular $\norm{b w}{F'} \geq \norm{w}{E}$.
This shows the isometry property.
By a similar argument,
$\sup_w \duality{ b w, v }{} \neq 0$ for all nonzero $v \in F$.
By \cite[Thm.~2.1]{Babuska1971}, $b$ is an isomorphism.

If a functional $\ell \in F'$ can be expressed as
$\ell(v) = \int_J g v$
for some $g \in L_1(J)$,
then
$u = b^{-1} \ell$
enjoys the representation
\[
	\label{e:ib}
	u(t) = (b^{-1} \ell)(t) = \int_0^t e^{-\lambda(t - s)} g(s) \rd s
	.
\]
Despite this integral representation,
$b^{-1}$ is not a compact operator
(it is an isomorphism).
Applying the expectation operator to
\eqref{e:X-add}--\eqref{e:X-mul}
shows
that
the first moment $m$ of the solution
satisfies the integral equation
\begin{equation*}
	m(t) = \e{ X_0 } - \lambda \int_0^t m(s) \rd{s}
	.
\end{equation*}
Testing this equation with
the derivative of an arbitrary
$v \in F$
and integrating by parts in time
shows that
the first moment of \eqref{e:sode-add}--\eqref{e:sode-mul}
solves the
deterministic variational problem
\begin{equation}
	\label{e:meaneqn}
	%
		\text{Find}
		\quad
		m \in E
		\TEXT{s.t.}
		b(m, v) = \e{ X_0 } v(0)
		\quad
		\forall v \in F.
\end{equation}

\subsection{Second moment equations: additive noise}
\label{s:sode:add}

For the deterministic equations for the second moment and the covariance
we need
the Hilbert tensor product spaces
\[
	\label{e:E2F2}
	E_2
	:=
	E \otimes_2 E
	\TEXT{and}
	F_2
	:=
	F \otimes_2 F
	,
\]
with $\norm{\CDOT}{2}$ denoting the norms on both spaces.
We further write
$\norm{\CDOT}{-2}$ for the norm of the dual
space $F_2'$ of $F_2$.
We recall the canonical isometry
(see \cite[Thm.~II.10]{ReedSimon1980} or \cite[Thm.~12.6.1]{Aubin2000})
\[
	\label{e:L2}
	E_2 =
	L_2(J) \otimes_2 L_2(J)
	\cong L_2(J \times J)
	.
\]

By virtue of square integrability \eqref{e:X-table-EX2},
the second moment $M$ is an element of $E_2$.
We define the bilinear form
\begin{equation*}
	B
	\from
	E_2 \times F_2
	\to
	\IR,
	\quad
	B := b \otimes b
	,
\end{equation*}
or explicitly as
\[
	\label{e:B}
	B(w, v)
	:=
	\int_J \int_J w(s, t)
	(-\partial_s + \lambda) (-\partial_t + \lambda) v(s, t) \rd s \rd t
	.
\]
More precisely,
$B$
is the unique continuous extension
of
$b \otimes b$
by bilinearity
from the algebraic tensor products to $E_2 \times F_2$.
Boundedness and injectivity of
the operator $B \from E_2 \to F_2'$ induced by the bilinear form $B$
follow readily from the corresponding properties of $b$,
so that the operator
$B$ is an isometry
and
its inverse is
the due continuous extension of
$b^{-1} \otimes b^{-1}$.
A representation of the inverse analogous to \eqref{e:ib} also holds.
For example,
the integral kernel of
the functional
$\ell(v) := v(0)$ is $\Dirac_0 \otimes \Dirac_0$,
which gives
$(B^{-1} \ell)(t, t') = e^{-\lambda (t + t')}$.

Recall the definitions of SPSD-ness from \eqref{e:JSPSD}--\eqref{e:psd}.

\begin{lemma}
	\label{l:psd}
	The function $U := B^{-1} \ell \in E_2$ is SPSD
	if and only if the functional $\ell \in F_2'$ is.
\end{lemma}

\begin{proof}
	Identifying $\varphi \in L_2(J)$ with $\psi \in F$
	via
	$\scalar{w, \varphi}{L_2(J)} = b(w, \psi)$ for all $w \in E$,
	we observe that
	$
	\scalar{U, \varphi \otimes \tilde\varphi}{L_2(J \times J)}
	=
	B(U, \psi \otimes \tilde\psi)
	=
	\ell(\psi \otimes \tilde\psi)
	$.
	Thus $U$ is SPSD \emph{iff} $\ell$ is.
\end{proof}

Finally, we introduce the \emph{bounded} linear functional
\[
	\label{e:d}
	\diag \from F_2 \to \IR
	,
	\quad
	\diag(v) := \int_J v(t,t) \rd t
	.
\]
As in \cite[Lem.~4.1]{LangLarssonSchwab2013},
one could use
\cite[Lem.~5.1]{TodorDiss}
to show boundedness of $\diag$.
We give here an elementary and quantitative argument.
Writing $\diag(v)$ as the integral of
$\Dirac(s - s') v(s, s')$ over $J \times J$
and exploiting the representation \eqref{e:ib} of $b^{-1}$
we find
$
(B^{-1} \diag)(t, t') =
(
e^{-\lambda|t - t'|}
-
e^{-\lambda(t + t')}
)
/
(2 \lambda)
$.
Since $B$ is an isometry,
the operator norm of $\diag$ is
\begin{equation*}
	%
	\norm{\diag}{-2}
	=
	\lambda \norm{B^{-1} \diag}{L_2(J \times J)}
	=
	\tfrac{1}{4\lambda}
	(
	{
	4 \lambda T - 5
	+
	(8\lambda T + 4)
	e^{-2\lambda T}
	+
	e^{-4\lambda T}
	}
	)^{1/2}.
\end{equation*}
%
%
In particular,
this yields the asymptotics
$\norm{\diag}{-2} \sim T^2\lambda / \sqrt{6}$ for small $\lambda$
and
$\norm{\diag}{-2} \sim \sqrt{T / (4 \lambda) }$ for large $\lambda$.
%
%
%
%
In addition, the uniform bound $\norm{\diag}{-2} \leq \tfrac12 T$
holds,
see Remark \ref{r:diag}.

We are now ready to
state
the deterministic equation for the second moment
(derived for stochastic PDEs in \cite{LangLarssonSchwab2013}).

\begin{proposition}
	The second moment $M = \e{ X \otimes X }$
	of
	the solution $X$
	to the stochastic ODE~\eqref{e:sode-add}
	with additive noise
	solves the deterministic variational problem
	\[
		\label{e:U-add}
		%
		\text{Find}
		\quad
		M \in E_2
		\TEXT{s.t.}
		B(M, v)
		=
		\e{ X_0^2 } v(0)
		+
		\mu^2 \diag(v)
		\quad
		\forall v \in F_2
		.
	\]
\end{proposition}

\begin{proof}
	Inserting the solution \eqref{e:X-add}
	in the first argument of
	$b(X, v) = \int_J \{ - X v' + \lambda X v \}$
	and integrating it by parts one finds
	\begin{equation*}
		b(X, v)
		=
		X_0 v(0)
		- \mu \int_J W(t) v'(t) \rd t
		=
		X_0 v(0)
		+
		\mu \int_J v(t) \rd W(t)
		\quad
		\forall v \in F
        \quad
		(\IP\text{-a.s.})
		,
	\end{equation*}
	where the stochastic
	integration by parts formula \cite[Thm.~4.1.5]{Oksendal2013}
	was used
	in the second equality.
	Employing this in
	$B(M, v_1 \otimes v_2) = \e{ b(X, v_1) b(X, v_2) }$
	with \eqref{e:itoiso} for the $\mu^2$ term
	leads to the desired conclusion.
\end{proof}

From the equations for the first and second moments,
an equation for the covariance function
$\Cov(X) \in E_2$ follows:
\begin{equation*}
	%
	B(\Cov(X), v)
	=
	\Cov(X_0) v(0)
	+
	\mu^2 \diag(v)
	\quad
	\forall v \in F_2
	.
\end{equation*}
The proof is straightforward and is therefore omitted.

\subsection{Second moment equations: multiplicative noise}
\label{s:sode:mul}

%

Before proceeding with the second moment
equation for the case of multiplicative noise,
we formulate a lemma
which repeats the derivation of the first moment equation
\eqref{e:meaneqn}
without taking the expectation first.

\begin{lemma}
	\label{l:cBX}
	Let $X$ be the solution~\eqref{e:X-mul}
	to the stochastic ODE \eqref{e:sode-mul}.
	Then
	\[
		\label{e:l:cBX}
		b(X,v)
		=
		X_0 v(0)
		-
		\rho \int_J \left( \int_0^t X(r) \rd W(r) \right) v'(t) \rd t
		\quad
		\forall
		v \in F
		\quad
		(\IP\text{-a.s.})
		.
	\]
\end{lemma}

\begin{proof}
	Let $v \in F$.
	We employ the definition~\eqref{e:X-mul}
	of the solution
	in the first term
	of $b(X, v)$
	and
	integration by parts on the first two summands of the integrand
	to obtain
	(observing that the terms at $t = T$ vanish due to $v(T) = 0$)
	\begin{align*}
		\int_J X(t) v'(t) \rd t
		  & =
		\int_J
		\left(
		X_0 - \int_0^t \lambda X(r) \rd r + \int_0^t \rho X(r) \rd W(r)
		\right)
		v'(t)
		\rd t
		\\
		  & =
		- X_0 v(0)
		+
		\lambda \int_J X(t) v(t) \rd t
		+
		\rho
		\int_J \left( \int_0^t X(r) \rd W(r) \right) v'(t) \rd t
		\quad
		(\IP\text{-a.s.})
		.
	\end{align*}
	Inserting this expression in the definition~\eqref{e:b}
	of $b(X, v)$
	yields the claimed formula.
\end{proof}

The next ingredient in
the second moment equation
for the case of multiplicative noise,
which appears
due to the integral term in \eqref{e:l:cBX},
is
the bilinear form
\[
	\label{e:def:DIAG}
	\DIAG(w, v) :=
	\int_J w(t,t) v(t,t) \rd t
	,
	\quad
	w \in E \otimes E,
	\quad
	v \in F \otimes F
	,
\]
referred to as the trace product.
Again, we use the same symbol for the induced operator, where convenient.
Here,
$\otimes$ denotes
the algebraic tensor product.
The expression \eqref{e:def:DIAG}
is meaningful
because functions in $F \subset H^1(J)$
are bounded.
As we will see in Lemma \ref{l:Delta},
this bilinear form extends
continuously
to a form
\[
	\label{e:D}
	\DIAG
	\from E_\pi \times F_\epsilon
	\to \IR
\]
on
the projective and the injective tensor product spaces
\[
	\label{e:EpiFeps}
	E_\pi := E \otimes_\pi E
	\TEXT{and}
	F_\epsilon := F \otimes_\epsilon F
	.
\]
These spaces are defined as
the closure of the algebraic tensor product
under
the projective norm
\[
	\label{e:pi}
	\norm{w}{\pi}
	  & :=
	\inf
	\left\{
	\textstyle
	\sum_i \norm{w_i^1}{E} \norm{w_i^2}{E}
	:
	w = \sum_i w_i^1 \otimes w_i^2
	\right\},
\]
and the injective norm
\[
	\label{e:eps}
	\norm{v}{\epsilon}
	  & :=
	\sup
	\left\{
	|(g_1 \otimes g_2)(v)|
	:
	g_1, g_2 \in S(F')
	\right\}
	,
\]
respectively.
Note that, initially, these norms are
defined
on the algebraic tensor product space.
In particular, the sums in~\eqref{e:pi} are finite
and the action of $g_1 \otimes g_2$ in~\eqref{e:eps}
is well-defined.
The spaces in~\eqref{e:EpiFeps} are separable Banach spaces.
They are reflexive
if and only if their dimension
is finite
\cite[Thm.~4.21]{Ryan2002}.
By~\cite[Prop.~6.1(a)]{Ryan2002},
these tensor norms satisfy
\[
	\label{e:simple}
	\norm{w_1 \otimes w_2}{\pi} = \norm{w_1}{E} \norm{w_2}{E}
	\quad \TEXT{and} \quad
	\norm{v_1 \otimes v_2}{\epsilon} = \norm{v_1}{F} \norm{v_2}{F}
	,
\]
as well as
\[
	\label{e:order}
	\norm{\CDOT}{2} \leq \norm{\CDOT}{\pi}
	\TEXT{on}
	E \otimes E
	{\quad}\TEXT{and}{\quad}
	\norm{\CDOT}{\epsilon} \leq \norm{\CDOT}{2}
	\TEXT{on}
	F \otimes F
	.
\]
We write $\norm{\CDOT}{-\epsilon}$ for the norm
of the continuous dual $F_\epsilon' := (F_\epsilon)'$.

\begin{example}
	\label{x:RN}
	%
	Consider $V := \IR^N$ with the Euclidean norm.
	Elements $A \in V \otimes V$ can be identified with $N \times N$ real matrices.
	Let $\sigma(A)$ denote the singular values of $A$.
	The projective, the Hilbert, and the injective norms
	on $V \otimes V$
	are
	the nuclear norm
	$\norm{A}{\pi} = \sum_{s \in \sigma(A)} s$,
	the Frobenius norm
	$\norm{A}{2} = (\sum_{s \in \sigma(A)} s^2)^{1/2}$,
	and
	the operator norm
	$\norm{A}{\epsilon} = \max \sigma(A)$,
	respectively.
	They are also known as the Schatten $p$-norms
	with $p = 1$, $2$, and $\infty$.
	Evidently,
	$\norm{\CDOT}{\pi} \geq \norm{\CDOT}{2} \geq \norm{\CDOT}{\epsilon}$.
\end{example}

If a function $w \in E_2$ is SPSD \eqref{e:JSPSD},
the operator
$S_w \from E \to E$
defined by
$S_w \varphi := \int_J w(s, \CDOT) \varphi(s) \rd s$
is self-adjoint and positive semi-definite.
%
%
%
Let
$
\{ s_n \}_n
\subset
\IntervalCO{0, \infty}
$
denote its eigenvalues.
If their sum is finite then
the operator is trace-class
and
$\norm{w}{\pi} = \sum_n s_n$,
see \cite[Thm.~9.1.38 and comments]{Palmer2001}.
We note that the correspondence
between symmetric positive semi-definite kernels, covariances, and
trace-class operators was
already observed
in~\cite{Mercer1909},
\cite[Thm.~XI.37.1.A]{Loeve1978}
and~\cite[Thm.~A.8 and p.~363]{PeszatZabczyk2007}
and extended to case of Hilbert space valued
kernels in~\cite[Thm.~2.3 and Cor.~2.4]{SchwabTodor2006}.
For our purposes, the following specialization will be particularly useful.

\begin{lemma}
	\label{l:2andpinorm}
	If $w \in E_\pi$ is SPSD
	then
	$\norm{w}{\pi} = \lambda \diag(w)$
	with
	$\diag$ from \eqref{e:d}.
\end{lemma}

\begin{proof}
	Let $\{ e_n \}_n$ be an orthonormal basis of $E$ consisting
	of eigenvectors of $S_w$ with the eigenvalues $\{s_n\}_n$.
	By symmetry,
	$w = \sum_n s_n e_n \otimes e_n$.
	Since $\lambda \diag(e_n \otimes e_n) = 1$,
	we have
	$
	\lambda \diag(w) = \sum_n s_n = \norm{w}{\pi}
	$.
\end{proof}

An arbitrary $w \in E_\pi$
can be decomposed
(via the corresponding integral operator)
as $w = w^+ - w^- + w^a$
with
SPSD $w^\pm \in E_\pi$
and
an
antisymmetric $w^a \in E_\pi$.
This decomposition is stable in the sense that
\[
	\label{e:wpm}
	\norm{w^a}{\pi}
	\leq \norm{w}{\pi}
	\TEXT{and}
	\norm{w^+ - w^-}{\pi} = \norm{w^+}{\pi} + \norm{w^-}{\pi}
	\leq \norm{w}{\pi}
	.
\]

The tensor product spaces $E_\pi$ and $F_\epsilon$
seem necessary
because the trace product
$\DIAG$ is \emph{not} continuous
on the Hilbert tensor product spaces
$E_2 \times F_2$
as the following example illustrates.
\begin{example}
	\label{x:E2F2counter}
	To simplify the notation, suppose $T = 1$,
	so that $J = (0,1)$.
	Define $v \in F_2$
	by
	$v(s, t):=(1 - s)(1 - t)$ for $s, t \in J$.
	Consider
	the sequence $u_1, u_2, \ldots$ of indicator functions
	\begin{equation*}
		u_n(s, t) :=
		\chi_{A_n}(s, t),
		\TEXT{where}
		A_n :=
		\left(0, \tfrac{1}{n} \right)^2
		\cup
		\left(\tfrac{1}{n}, \tfrac{2}{n} \right)^2
		\cup
		\cdots
		\cup
		\left(\tfrac{n-1}{n}, 1 \right)^2
		\subset
		J \times J
		.
	\end{equation*}
	In view of the canonical isometry \eqref{e:L2},
	this sequence
	is a null sequence in $E_2$.
	However,
	$\DIAG(u_n, v) = \int_J u_n(t,t) v(t,t) \rd t = \tfrac{1}{3}$
	for all $n \geq 1$.
	Therefore, $\DIAG(\CDOT, v)$ is not continuous on $E_2$.

	The example additionally shows that
	$\DIAG$ is not continuous on $E_\epsilon \times F_\pi$ either,
	since by \eqref{e:simple}--\eqref{e:order}
	we have
	$\norm{v}{\pi} = \norm{v}{2}$,
	while
	$\norm{u_n}{\epsilon} \leq \norm{u_n}{2} \to 0$
	as $n\to\infty$.

	By contrast, $\{ u_n \}_{n \geq 1}$ is not a null sequence in $E_\pi$.
	Indeed, Lemma~\ref{l:2andpinorm} gives $\norm{u_n}{\pi} = \lambda$
	for all $n \geq 1$.
\end{example}

\begin{lemma}
	\label{l:Delta}
	The trace product
	$\DIAG$ in \eqref{e:D} is continuous on $E_\pi \times F_\epsilon$
	with $\norm{\DIAG}{} \leq 1/(2\lambda)$.
\end{lemma}
\begin{proof}
	By density
	it suffices to bound
	$\DIAG(w, v)$
	for arbitrary
	$w \in E \otimes E$ and
	$v \in F \otimes F$.
	By \cite[Thm.~2.4]{Schatten1950}
	we may assume that $w = w^1 \otimes w^2$.
	We note first that
	the point evaluation functionals $\Dirac_t : v \mapsto v(t)$ have norm $1/\sqrt{2}$ on $F$ by
	\eqref{e:vt2}.
	Therefore,
	if $v = \sum_j v_j^1 \otimes v_j^2$ then
	\[
		\label{e:vtt}
		\textstyle
		| v(s, t) |
		=
		|
		\sum_j \Dirac_s(v_j^1) \Dirac_t(v_j^2)
		|
		\leq
		\sup
		\{
		\tfrac12
		|
		\sum_j g_1(v_j^1) g_2(v_j^2)
		|
		:
		g_1, g_2 \in S(F')
		\}
		=
		\tfrac12
		\norm{v}{\epsilon}
	\]
	and the continuity of $\DIAG$ follows:
	\begin{equation*}
		\textstyle
		|\DIAG(w, v)|
		=
		\left|
		\int_J w(t,t) v(t,t) \rd t
		\right|
		\leq
		\tfrac12 \norm{v}{\epsilon} \int_J |w(t,t)| \rd t
		\leq
		\tfrac{1}{2 \lambda} \norm{v}{\epsilon} \norm{w}{\pi}
		,
	\end{equation*}
	where the integral Cauchy--Schwarz inequality on $w(t,t) = w^1(t) w^2(t)$ was used in the last step,
	together with
	the fact that $\lambda \norm{w^1}{L_2(J)} \norm{w^2}{L_2(J)} = \norm{w^1}{E} \norm{w^2}{E} = \norm{w^1 \otimes w^2}{\pi}$.
\end{proof}

We note that the bound $\norm{\DIAG}{} \leq 1 / (2\lambda)$
is sharp:
For $\eta > 0$
take
$w = \varphi \otimes \varphi$ with $\varphi := \chi_{(0, \eta)} / \sqrt{\eta}$
and
$v = \psi \otimes \psi$ with $\psi(t) := \sinh(\lambda(T - t)) / \sinh(\lambda T)$
as in
\eqref{e:v0}.
Then $\lim_{\eta \to 0} \DIAG(w, v) = 1$
and
$\lim_{T \to \infty} \norm{v}{\epsilon} \norm{w}{\pi} = 2\lambda$,
and the bound
is tight when applying both limits.

\begin{remark}
	\label{r:diag}
	Consider the functional
	$\diag$ from \eqref{e:d}.
	Since $\diag = \DIAG (1 \otimes 1)$ and $\norm{1 \otimes 1}{\pi} = \lambda T$,
	we have $\norm{\diag \from F_\epsilon \to \IR}{-\epsilon} \leq T/2$.
	In view of $\norm{\CDOT}{\epsilon} \leq \norm{\CDOT}{2}$ from \eqref{e:order},
	we find
	$\norm{\diag \from F_2 \to \IR}{-2} \leq T/2$.
	Finally, $\norm{\diag \from E_\pi \to \IR}{-\pi} = 1 / \lambda$
	by
	the integral Cauchy--Schwarz inequality
	and
	Lemma~\ref{l:2andpinorm}.
\end{remark}

A crucial observation is that
the second moment $M$ lies
not only in the Hilbert tensor product space $E_2$
but in the smaller projective tensor product space,
$M \in E_\pi$.
This follows by passing the norm under
the expectation
$
\norm{ \e{ X \otimes X } }{\pi}
\leq
\e{ \norm{ X \otimes X }{\pi} }
$,
then using \eqref{e:simple} and
the square integrability~\eqref{e:X-table-EX2} of $X$.

We recall here
from \cite[Thm.~2.5 and Thm.~5.13]{Schatten1950}
the fact that
\begin{equation*}
	%
	F_\epsilon' =
	(F \otimes_\epsilon F)' \cong F' \otimes_\pi F'
	\quad
	\text{isometrically}
	,
\end{equation*}
(whereas
the space $(F')_\epsilon$
is isometric to a proper subspace of $(F_\pi)'$,
see \cite[pp.~23/46]{Ryan2002}).
A corollary of this representation
is that
\[
	\label{e:bob-epi}
	b \otimes b
	\from
	E_\pi \to F_\epsilon'
	\quad
	\text{defines an isometric isomorphism}
	,
\]
because $b \otimes b$
extends to
an isometric isomorphism
from $E \otimes_\pi E$ onto $F' \otimes_\pi F'$.
We call it also $B$.
This isometry property \eqref{e:bob-epi},
Lemma~\ref{l:psd}
and
Lemma~\ref{l:2andpinorm}
produce the useful identity
\[
	\label{e:eps2del}
	\norm{ \ell }{-\epsilon}
	=
	\norm{ B^{-1} \ell }{\pi}
	=
	\lambda
	\diag( B^{-1} \ell )
\]
for any $\ell \in F_\epsilon'$ which is SPSD \eqref{e:psd}.
Here and below,
Lemma~\ref{l:psd} applies to functionals in $F_\epsilon'$
mutatis mutandis.
Using the decomposition from \eqref{e:wpm}
we can decompose
any $\ell = \ell^+ - \ell^- + \ell^a$
into SPSD
and antisymmetric
parts
with
\[
	\label{e:lpm}
	\norm{\ell^a}{-\epsilon} \leq \norm{\ell}{-\epsilon}
	\TEXT{and}
	\norm{\ell^+ - \ell^-}{-\epsilon}
	=
	\norm{\ell^+}{-\epsilon} + \norm{\ell^-}{-\epsilon}
	\leq
	\norm{\ell}{-\epsilon}
	.
\]

Now we are in position to introduce
the bilinear form
\[
	\label{e:cB}
	\cB
	\from E_\pi \times F_\epsilon \to \IR
	,
	\quad
	\cB
	:=
	B
	-
	\rho^2
	\DIAG
	,
\]
or more explicitly,
\begin{equation*}
	\cB(w, v)
	=
	\int_J \int_J w(s, t)
	(-\partial_s + \lambda) (-\partial_t + \lambda) v(s, t)
	\rd s \rd t
	-
	\rho^2 \int_J w(t,t) v(t,t) \rd t
	.
\end{equation*}

The reason for this definition
is the following result from
\cite[Thm.~4.2]{KirchnerLangLarsson2016}
derived there for stochastic {PDEs}.
The simplified proof is given here for completeness.

\begin{proposition}
	\label{p:U-mul}
	The second moment $M = \e{ X \otimes X }$
	of
	the solution $X$
	to the stochastic ODE~\eqref{e:sode-mul} with multiplicative noise
	solves the deterministic variational problem
	\[
		\label{e:U-mul}
		%
		\text{Find}
		\quad
		M \in E_\pi
		\quad
		\text{s.t.}
		\quad
		\cB(M, v) = \e{ X_0^2 } v(0)
		\quad
		\forall v \in F_\epsilon
		.
	\]
\end{proposition}

\begin{proof}
	It suffices to verify the claim for $v$ of the form $v = v_1 \otimes v_2$
	with $v_1, v_2 \in F$.
	The more general statement
	follows by
	linearity and continuity of both sides
	in $v \in F_\epsilon$.
	We first observe with
	Fubini's theorem on $\Omega \times J$
	that
	$ 
	B(M, v_1 \otimes v_2)
	=
	B(\e{ X \otimes X }, v_1 \otimes v_2)
	=
	\e{ b(X, v_1) b(X, v_2) }
	.
	$ 
	Next, we insert the expression \eqref{e:l:cBX}
	for both $b(X, v_j)$
	and expand the product.
	The cross-terms vanish
	because
	the terms of the form
	$ 
	X_0 \int_0^t X(r) \rd W(r)
	$ 
	vanish in expectation;
	this is seen by
	conditioning this term on $\cF_0$
	and
	employing the martingale property \eqref{e:mul-M}.
	With the identity
	\eqref{e:itoiso}
	and
	$\e{ X(r)^2 } = M(r,r)$
	we arrive at
	\begin{equation*}
		B(M, v_1 \otimes v_2)
		=
		\e{ X_0^2 } v(0)
		+
		\rho^2
		\int_J \int_J
		v_1'(s) v_2'(t)
		\int_0^{s\wedge t} M(r,r) \rd r
		\rd s \rd t
		.
	\end{equation*}
	It remains to verify that $\rho^2 \DIAG(M, v)$
	coincides with the last term on the right-hand side.
	Let us distinguish
	the two cases
	$s = s \wedge t$ and $t = s \wedge t$
	and
	write that triple integral as
	\[
		\label{e:1short}
		\int_J v_1'(s) \int_s^T v_2'(t) \rd t \int_0^{s} M(r,r) \rd r \rd s
		+
		\int_J v_2'(t) \int_t^T v_1'(s) \rd s \int_0^{t} M(r,r) \rd r \rd t
		.
	\]
	Evaluating the $\rd t$ integral in the first summand
	and the $\rd s$ integral in the second summand,
	we see that
	$
	(\eqref{e:1short} - \DIAG(M, v))
	=
	\int_J \tfrac{d}{dt} \{ - v_1(t) v_2(t) \int_0^t M(r,r) \rd r \} \rd t
	=
	0
	$.
	Hence, $\eqref{e:1short} = \DIAG(M, v)$.
	This completes the proof.
\end{proof}

Using the equations for the first and second moments
we obtain an equation for the covariance function $\Cov(X) \in E_\pi$ from \eqref{e:CovX}:
\[
	\label{e:C-mul}
	\cB(\Cov(X), v) = \Cov(X_0) v(0) + \rho^2 \DIAG(m \otimes m, v)
	\quad
	\forall
	v \in F_\epsilon
	.
\]

Identity \eqref{e:eps2del} yields
$
\norm{ v \mapsto v(0) }{-\epsilon} =
\norm{\Dirac_0 \otimes \Dirac_0}{-\epsilon} =
\tfrac12 (1 - e^{-2 \lambda T})
$
for the functional appearing
on the right-hand side of \eqref{e:U-mul} and \eqref{e:C-mul}.
Similarly,
$
\norm{\DIAG (m \otimes m)}{-\epsilon}
=
\tfrac{1}{2}
\int_J (1 - e^{-2\lambda(T-t)}) |m(t)|^2 \rd t
\leq
\tfrac{1}{2\lambda} \norm{m}{E}^2
$,
in agreement with
Lemma \ref{l:Delta}.

We emphasize that
it is not possible to replace
in the present case of multiplicative noise
the pair of trial and test spaces
$E_\pi \times F_\epsilon$
by either pair
$E_2 \times F_2$ or
$E_\epsilon \times F_\pi$,
because
by Example~\ref{x:E2F2counter}
the operator $\DIAG$ is not continuous there.
We note, however,
that
in the case of additive noise (\S\ref{s:sode:add})
the pair $E_\pi \times F_\epsilon$
could be used
instead of $E_2 \times F_2$.
Then
$
\norm{\diag}{-\epsilon} =
\lambda \diag(B^{-1} \diag) =
\tfrac{1}{4 \lambda}
( e^{-2 T \lambda} - 1 + 2 T \lambda )
$
with the asymptotics
$\tfrac12 T^2 \lambda$ (small $\lambda$)
and $\tfrac12 T$ (large $\lambda$).

In order to discuss the well-posedness of
the variational problem \eqref{e:U-mul},
given
a functional $\ell \in F_\epsilon'$,
we consider
the more general problem:
\[
	\label{e:U-mul-L}
	\text{Find}
	\quad
	U \in E_\pi
	\quad\text{s.t.}\quad
	\cB(U, v) = \ell(v)
	\quad
	\forall v \in F_\epsilon
	.
\]

Owing to
$
\norm{B w}{-\epsilon}
=
\norm{w}{\pi}
$
and
$\norm{\DIAG}{} \leq 1 / (2\lambda)$
we have
$\norm{ \cB w }{-\epsilon} \geq (1 - \rho^2 / (2\lambda)) \norm{w}{\pi}$.
Thus,
injectivity of $\cB$ holds under
the condition
$
\rho^2 < 2 \lambda
$
of small ``volatility''.
A similar condition was imposed in \cite[Thm.~5.5]{KirchnerLangLarsson2016}.
This is exactly
the threshold for
the second moment
\eqref{e:X-table-EXsXt}
to diverge as $s = t \to \infty$,
but it stays nevertheless finite for all finite $s = t$.
We discuss here what happens in
the variational formulation \eqref{e:U-mul-L}
for larger volatilities $\rho$,
and
summarize in Theorem \ref{t:stab-mul} below.

Since $B$ is an isomorphism, problem \eqref{e:U-mul-L}
is equivalent to
$U = \rho^2 B^{-1} \DIAG U + B^{-1} \ell$.
Using the representation of $\DIAG(U, v)$ as the double integral
of $\Dirac(s - s') U(s, s') v(s, s')$,
and the integral representation of $B^{-1}$ through \eqref{e:ib},
we obtain the integral equation
\[
	\label{e:uii}
	U(t, t') =
	\rho^2
	\int_0^{t \wedge t'}
	e^{-\lambda(t + t' - 2 s)}
	U(s, s)
	\rd s
	+
	(B^{-1} \ell)(t, t')
	.
\]
Defining
$f(t) := (B^{-1} \DIAG U)(t,t) = \int_0^t e^{-2 \lambda (t - s)} U(s, s) \rd s$
and
$g(t) := (B^{-1} \ell)(t, t)$
we find from \eqref{e:uii}
the ODE
$f'(t) + 2 \lambda f(t) = \rho^2 f(t) + g(t)$
with
the initial condition $f(0) = 0$.
The solution is
\[
	\label{e:Ft}
	f(t)
	=
	(B^{-1} \DIAG U)(t,t)
	=
	\int_0^t e^{-(2 \lambda - \rho^2)(t - r)} g(r) \rd r
	.
\]
Inserting
\[
	\label{e:ugrel}
	U(s,s)
	= \rho^2 f(s) + g(s)
	= \rho^2 \int_0^s e^{-(2\lambda-\rho^2)(s-r)} g(r) \rd r + g(s)
\]
under the integral of \eqref{e:uii}
provides a unique candidate for $U$.
Moreover, $U \in E_2$.
We now estimate $\norm{U}{\pi}$ in terms of the norm of $\ell$.

Clearly,
not all functionals $\ell$
lead to solutions
that are potential second moments.
Let us therefore assume first
that $\ell$
is SPSD.
Then $B^{-1} \ell$ is SPSD
by Lemma \ref{l:psd}.
In particular, $f \geq 0$ and $g \geq 0$.
Thus the functional
$
v \mapsto \DIAG (U, v)
=
\int_J (\rho^2 f(t) + g(t)) \, v(t,t) \rd t
$
is SPSD.
Now
$U = \rho^2 B^{-1} \DIAG U  + B^{-1} \ell$
is the sum of two SPSD functions (Lemma \ref{l:psd})
and
is therefore SPSD.
Under these assumptions,
Lemma~\ref{l:2andpinorm} gives
\[
	\label{e:upiell}
	\norm{U}{\pi}
	=
	\lambda \diag(U)
	=
	\rho^2 \lambda \diag(B^{-1} \DIAG U)
	+
	\lambda \diag(B^{-1} \ell)
	.
\]
For the first term
on the right-hand side of \eqref{e:upiell}
we employ \eqref{e:Ft} as follows:
\[
	\label{e:dBDu}
	\diag(B^{-1} \DIAG U)
	=
	\int_J
	g(r)
	\int_r^T
	e^{-(2 \lambda - \rho^2)(s - r)}
	\rd s
	\rd r
	\leq
	\diag(B^{-1} \ell)
	\tfrac{
	e^{(\rho^2 - 2\lambda) T } - 1
	}{
	\rho^2 - 2\lambda
	}
	,
\]
where we have
exchanged the order of integration
in the first step,
evaluated the inner integral
and
used $g \geq 0$
with
$\norm{g}{L_1(J)} = \diag(B^{-1} \ell)$
in the last step.
The fraction evaluates to $T$
in the limit $\rho^2 = 2\lambda$.
Combining \eqref{e:upiell}--\eqref{e:dBDu}
and \eqref{e:eps2del},
we arrive at the following theorem.

\begin{theorem}
	\label{t:stab-mul}
	Suppose that $\ell \in F_\epsilon'$ is
	SPSD.
	Then,
	for any $\rho \geq 0$ and $\lambda > 0$,
	the variational problem \eqref{e:U-mul-L}
	has a unique solution $U \in E_\pi$.
	This solution is SPSD
	and admits the bound
	\[
		\label{e:stab-mul}
		\norm{U}{\pi}
		\leq
		C
		\norm{\ell}{-\epsilon}
		\TEXT{with}
		C :=
		\tfrac{
		\rho^2 e^{(\rho^2 - 2\lambda) T } - 2 \lambda
		}{
		\rho^2 - 2 \lambda
		}
		,
	\]
	where $C = \rho^2 T + 1$ for $\rho^2 = 2\lambda$.
\end{theorem}

The bound in \eqref{e:stab-mul}
is sharp:
for $\eta > 0$ and
$\ell := \eta^{-1} B (\chi_{(0,\eta)} \otimes \chi_{(0,\eta)})$
we have $g = \eta^{-1} \chi_{(0,\eta)}$ in~\eqref{e:dBDu},
and
the inequality
in \eqref{e:dBDu}
approaches an equality
as $\eta \searrow 0$.

For a general functional $\ell \in F_\epsilon'$,
we decompose $\ell = \ell^+ - \ell^- + \ell^a$
as in \eqref{e:lpm}.
The corresponding solutions
$U^\pm := \cB^{-1} \ell^\pm$ and $U^a := \cB^{-1} \ell^a = B^{-1} \ell^a$
(noting that $\DIAG U^a = 0$ by antisymmetry)
satisfy
the bounds
$\norm{U^\pm}{\pi} \leq C \norm{\ell^\pm}{-\epsilon}$
and
$\norm{U^a}{\pi} = \norm{\ell^a}{-\epsilon}$.
By linearity, $U := U^+ - U^- + U^a$
is the solution to \eqref{e:U-mul-L},
and the estimate
$ 
\norm{U}{\pi}
\leq
C(\norm{\ell^+}{-\epsilon} + \norm{\ell^-}{-\epsilon}) + \norm{\ell^a}{-\epsilon}
\leq
(C+1) \norm{\ell}{-\epsilon}
$ 
follows
by triangle inequality in the first step
and
by \eqref{e:lpm} in the last step.

In contrast to Lemma~\ref{l:psd},
the solution $U$ to~\eqref{e:U-mul-L}
may be SPSD
even though the right-hand side $\ell$ is not.
Indeed,
for
any $(w, v) \in E \times F$
with $\DIAG(w \otimes w, v \otimes v) = \int_J |w(t) v(t)|^2 \rd t \neq 0$,
the expression
$\cB(w \otimes w, v \otimes v) = |b(w, v)|^2 - \rho^2 \DIAG(w \otimes w, v \otimes v)$
is negative for sufficiently large $\rho$.

The variational formulations~\eqref{e:U-mul},~\eqref{e:C-mul}
for the second moment and the covariance function
are of the form \eqref{e:U-mul-L}
for the functionals
$
\ell :=
\e{X_0^2} (\Dirac_0 \otimes \Dirac_0)
$
and
$
\ell :=
\Cov(X_0) (\Dirac_0 \otimes \Dirac_0)
+
\rho^2 \DIAG(m \otimes m)
$.

The proof of the above theorem
highlights the special status
of the diagonal $t \mapsto U(t,t)$.
First, it
is uniquely defined
as the solution of
an integral equation.
Second, it determines
all other
off-diagonal values of $U$.
Finally,
the projective norm \eqref{e:upiell} only ``looks'' at the diagonal
(when
$U$ is SPSD).
These insights will guide
\textbf{a)} the development of
the numerical methods in \S\ref{s:dis} and
\textbf{b)} the proof of
well-posedness of the deterministic second moment
equation also for the vector-valued case in \S\ref{s:pde}.

\section{Conforming discretizations of the deterministic equations}
\label{s:dis}

\subsection{Orientation}
\label{s:dis:0}

In \S\ref{s:sode} we have derived
deterministic variational formulations
for
the first and second moments of
the stochastic processes \eqref{e:X-add} and \eqref{e:X-mul}.
In particular,
the first moment satisfies
a known ``weak'' variational formulation
of an ODE.
To our knowledge,
\cite{BabuskaJanik1989,BabuskaJanik1990}
were the first
to discuss
the numerical analysis
of
conforming finite element discretizations
of
a space-time variational formulation
for
linear parabolic PDEs.
The problem was first reduced
to the underlying family of ODEs
parameterized by the spectral parameter $\lambda$.
With the notation
from \S\ref{s:sode:det}
for
the bilinear form $b$
and
the spaces $E$ and $F$,
the solution $u$ to such an ODE
is characterized by
a well-posed variational problem
of the above form \eqref{e:meaneqn},
with a general right-hand side $\ell$.
The temporal discretization analyzed in~\cite{BabuskaJanik1990}
was of the conforming type,
employing discontinuous piecewise polynomials
as the discrete trial space for $u$
and
continuous piecewise polynomials of one degree higher
as the discrete test space for $v$.
The analysis in essence revealed
that the discretization
is \emph{not} uniformly stable
(in the Petrov--Galerkin sense, as discussed below)
in the choice of
the discretization parameters
such as the polynomial degree and
the location of the temporal nodes
\cite[Thm.~2.2.1]{BabuskaJanik1990}.

The same question of stability of
was taken up in \cite{AndreevSchweitzer2014}
for
a ``strong'' space-time variational formulation
of linear parabolic PDEs
and
for the two classes of discretizations,
of Gauss--Legendre (e.g.,~Crank--Nicolson, CN) or Gauss--Radau (e.g.,~implicit Euler, iE) type.
It was confirmed that both types
are in general only
\emph{conditionally}
space-time stable,
but
the Gauss--Radau type
can be made \emph{unconditionally} stable
under mild restrictions on the temporal mesh.
We will first revisit
the simplest representative of each group
adapted to
the present variational formulation.
The adaptation consists in
switching the roles of
the discrete trial and test spaces
and by reversing the temporal direction,
the latter
due to the integration by parts
that was used in the derivation of
the variational formulation \eqref{e:meaneqn}.
The resulting adjoint discretizations
will therefore be denoted by
$\textrm{CN}^{\star}$ and $\textrm{iE}^{\star}$,
respectively.
The $\textrm{CN}^{\star}$ discretization
is thus a special case of
the discretizations
analyzed in~\cite{BabuskaJanik1990}.

In summary,
in \S\ref{s:dis:1}
we will discuss
two conforming discretizations
for
the deterministic first moment equation \eqref{e:meaneqn}:
$\textrm{CN}^{\star}$ which is only conditionally stable
(depending on the spectral parameter $\lambda$)
and
$\textrm{iE}^{\star}$ which is stable under a mild condition
on the temporal mesh
(comparable size of neighboring temporal elements).
Both employ discontinuous trial spaces
but $\textrm{iE}^{\star}$ requires
additional discussion due to the somewhat unusual shape functions,
whereby the discrete trial spaces are not nested and therefore do not generate a dense subspace in the usual sense.
The situation
transfers with no surprises
to
the second moment equations with additive noise~\eqref{e:U-add}
by tensorizing the discrete trial/test spaces.
The case of multiplicative noise~\eqref{e:U-mul}, however,
presents a significant twist
due to:
\begin{enumerate}
	\item
		the presence of the $\DIAG$ term in
		the definition~\eqref{e:cB} of the bilinear form $\cB$.
		We will see that
		$\textrm{CN}^{\star}$ interacts naturally with the $\DIAG$ operator
		while
		$\textrm{iE}^{\star}$ requires a modification to
		restore the expected convergence order.
	\item
		the non-Hilbertian nature of the trial and test spaces
		in \eqref{e:U-mul}.
\end{enumerate}

We will then provide a common framework
for both discretizations,
generalizing to arbitrary polynomial degrees.
This will allow us to use
the unconditionally stable Gauss--Radau discretization family
without resorting to the modification
of the lowest-order $\textrm{iE}^{\star}$ discretization
because
the discrete trial spaces with higher polynomial degree
do generate a dense subspace.

Since the
trial and test spaces in \eqref{e:U-mul} are not Hilbert spaces,
we briefly state results on
Petrov--Galerkin discretizations
of variational problems
on normed spaces
in \S\ref{s:dis:PG}.
In \S\ref{s:dis:1:o}
we construct discretizations
on tensor product spaces
and
comment on their stability.
These are applied
to
the variational problem~\eqref{e:U-add}
for the second moment in the additive case
in \S\ref{s:dis:2-add}.

In the multiplicative case
we obtained
existence and stability
of the exact solution
for arbitrary $\rho \geq 0$
in Theorem \ref{t:stab-mul},
even
beyond the trivial range $0 \leq \rho^2 < 2 \lambda$.
The situation is similar in the discrete setting,
where this trivial range is reduced
by the discrete inf-sup constant $\gamma_k$
to
$0 \leq \rho^2 < 2 \lambda \gamma_k^2$.
In \S\ref{s:dis:2-mul}
we will therefore investigate,
for the low order $\textrm{CN}^{\star}$ and $\textrm{iE}^{\star}$ schemes
and some of their variants,
whether stability holds
for all $\rho \geq 0$.
The behavior of the high order discretizations
beyond the trivial stability range
remains an open question.

%

\subsection{First moment discretization}
\label{s:dis:1}

We are using the notation from \S\ref{s:sode:det}.
Let us consider
the general formulation of \eqref{e:meaneqn}
as
the variational problem
\[
	\label{e:uode-vf}
	\text{Find}
	\quad
	u \in E
	\quad\text{s.t.}\quad
	b(u,v) = \ell(v)
	\quad
	\forall v \in F
\]
with some bounded linear functional $\ell \in F'$.
Recall that the spaces $E$ and $F$ carry
the $\lambda$-dependent norms \eqref{e:EF-norms}
that
render $b : E \to F'$ an isometric isomorphism.
This variational problem is formally obtained
by testing the real-valued ODE
\[
	\label{e:uode}
	u'(t) + \lambda u(t) = f
	\quad\text{on}\quad
	J = (0, T),
	\quad
	u(0) = g,
\]
with a test function $v$,
integrating over $J$,
moving
the derivative from $u'$ to $v$
via integration by parts
and
then replacing
the exposed $u(0)$ by the given initial datum $g$.
The corresponding
right-hand side then reads as
$\ell(v) := \int_J \duality{f, v}{} \rd t + \duality{ g, v(0) }{}$.
We write $\duality{\CDOT, \CDOT}{}$
for the simple multiplication
to emphasize
the structure of the problem
and
to facilitate the transition to vector-valued ODEs.

For the discretization
of the variational problem~\eqref{e:uode-vf}
we need to define subspaces
\begin{equation*}
	%
	E^k \subset E
	\TEXT{and}
	F^k \subset F
\end{equation*}
of the same (nontrivial) finite dimension.
We then consider the discrete variational problem
\[
	\label{e:uk}
	\text{Find}
	\quad
	u^k \in E^k
	\TEXT{s.t.}
	b(u^k, v)
	=
	\ell(v)
	\quad
	\forall v \in F^k
	.
\]
The well-posedness of
this discrete problem
is quantified by the discrete inf-sup constant
\[
	\label{e:ga-EF}
	\gamma_k
	:=
	\inf_{w \in S(E^k)}
	\sup_{v \in S(F^k)}
	b(w, v)
	>
	0,
\]
since
the norm of
the discrete data-to-solution mapping $\ell|_{F^k} \mapsto u_k$
equals $1 / \gamma_k$.
Moreover,
the quasi-optimality estimate
\[
	\label{e:u-qo}
	\norm{ u - u^k }{E}
	\leq
	(\norm{b}{} / \gamma_k)
	\inf_{w \in E^k} \norm{ u - w }{E}
\]
holds \cite[Thm.~2]{XuZikatanov2003},
where in fact $\norm{b}{} = 1$ by \eqref{e:b-iso}.
We call a family
$\{ E^k \times F^k \}_{k > 0}$,
of discretization pairs
uniformly stable
if $\inf_{k > 0} \gamma_k > 0$.
To construct $E^k \times F^k$
we introduce a temporal mesh
\[
	\label{e:cT}
	\cT
	:=
	\{
	0 =: t_0
	< t_1
	< \ldots
	< t_N := T
	\}
\]
subdividing $J = (0, T)$
into $N$ temporal elements.
Below,
the dependence on $\cT$
is implicit in the notation.
We write
\begin{equation*}
	J_n := (t_{n-1}, t_n)
	\TEXT{and}
	k_n := |t_n - t_{n-1}|,
	\quad
	n = 1, \ldots, N.
\end{equation*}

As announced above, we first discuss
the simplest representatives of the
Gauss--Legendre and Gauss--Radau
discretizations in \S\ref{s:dis:1:CN*}--\S\ref{s:dis:1:iE*},
which are the $\text{CN}^\star$ and
the $\text{iE}^\star$ schemes.
For both methods, the discrete test space
$F^k \subset F$
is defined as
the spline space of
continuous piecewise affine functions $v$
with respect to the temporal mesh $\cT$
such that $v(T) = 0$.
A common framework is the subject of \S\ref{s:dis:1:both}.

\subsubsection{The $\text{CN}^\star$ discretization}
\label{s:dis:1:CN*}

For the discrete trial space~$E^k \subset E$,
the space of piecewise constant functions
with respect to~$\cT$
seems a natural choice.
We call this discretization
$\textrm{CN}^\star$
in reference to
the reversal of the roles of the trial and test spaces
compared to
the usual Crank--Nicolson time-stepping scheme.
Unfortunately,
if we keep the temporal mesh~$\cT$ fixed,
the discrete inf-sup constant~\eqref{e:ga-EF}
of the couple~$E^k \times F^k$
depends on the spectral parameter~$\lambda$,
see Figure~\ref{f:CN-iE}.
This was already observed in
\cite[Eqn.~(2.3.10)]{BabuskaJanik1990}.
It can be shown
along the lines of~\cite{AndreevSchweitzer2014}
that
$\gamma_k \gtrsim (1 + \min\{ \sqrt{\lambda T}, \textrm{CFL} \} )^{-1}$,
where $\textrm{CFL} := \max_n k_n \lambda$ is
the parabolic CFL number.
The three-phase behavior of
the $\text{CN}^{\star}$ scheme in
Figure~\ref{f:CN-iE}
can be intuitively understood as follows:
Consider
$b(w, v) = \int_J (-v' + \lambda v) w$
from \eqref{e:ga-EF}.
For any $w \in E^k$ we can find a $v \in F^k$ such that $-v' = w$,
so that
at sufficiently low spectral numbers $\lambda$,
the estimate $\gamma_k \geq 1 - \epsilon$ is evident.
For large $\lambda$, the function $-v' + \lambda v$
is, up to negligible jumps, a piecewise linear continuous one.
Such functions approximate a general piecewise constant $w$
poorly, see
\cite[Eqn.~(2.3.10)]{BabuskaJanik1990}.

This behavior renders the method
less useful for parabolic PDEs
because following a spatial semi-discretization,
a low parabolic CFL number has to be maintained
for uniform stability.
\begin{figure}[tpb]
	\centering
	\includegraphics[width=0.5\textwidth]{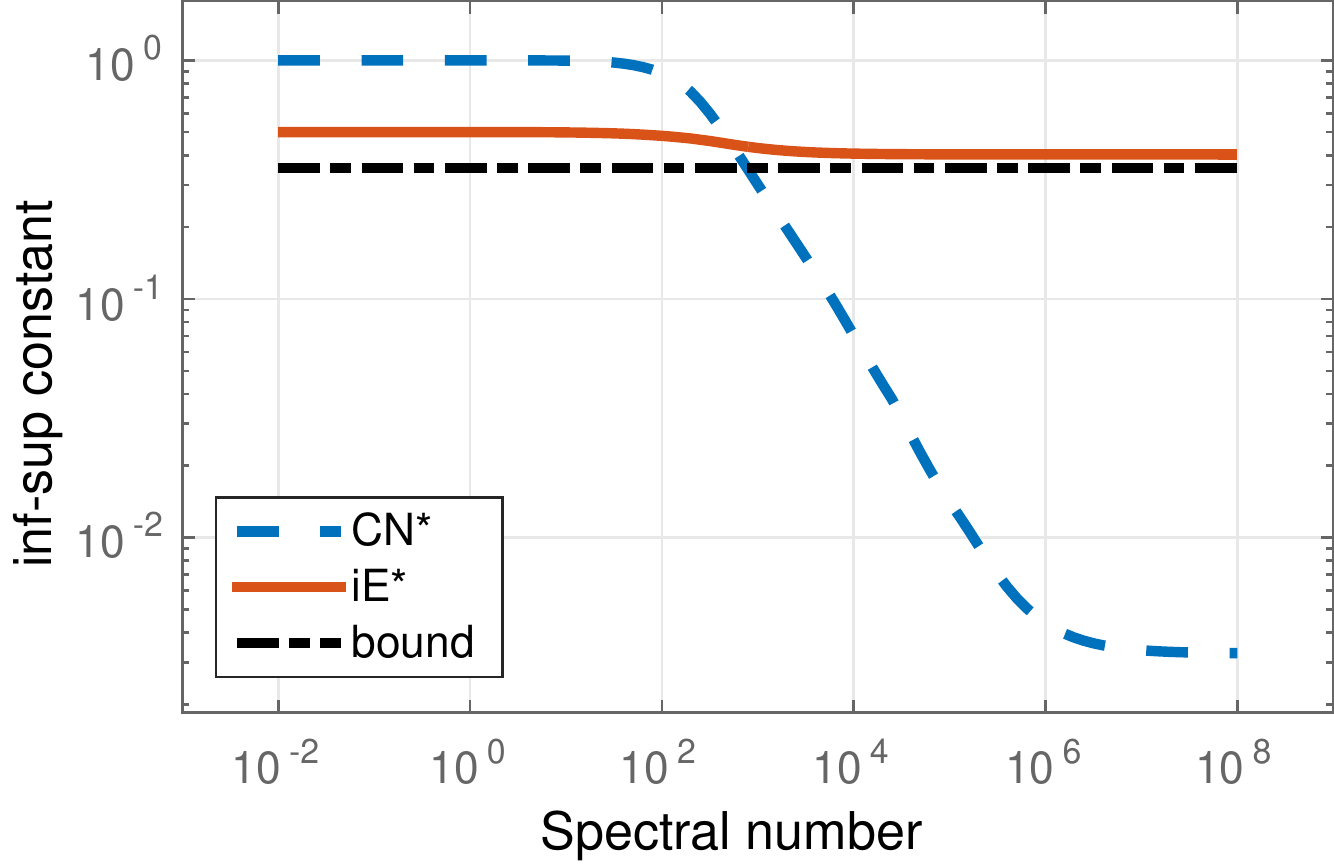}
	\caption{%
		The inf-sup constant \eqref{e:ga-EF} for
		the $\textrm{CN}^*$
		and
		the $\textrm{iE}^*$
		discretizations
		on the same ``random'' temporal mesh of the interval $(0, 1)$
		with 210 nodes
		and
		backward successive temporal element ratio
		$\sigma \leq 3$ in \eqref{e:sigma}.
		The bound shown is the estimate from \eqref{e:p:iE*}.
	}
	\label{f:CN-iE}
\end{figure}
%
%

\subsubsection{The $\text{iE}^\star$ discretization}
\label{s:dis:1:iE*}

To obtain stability under only mild restrictions
we
recur to an observation
of \cite[\S3.4]{AndreevSchweitzer2014};
for the sake of a self-contained exposition
and
sharp results
we confine
the discussion first
to the lowest order case.
We take $E^k$
as the space of functions $w \in L_2(J)$
for which each
$w|_{J_n}$
is a dilated translate
of
the shape function $\phi : s \mapsto (4 - 6 s)$
from the reference temporal element $(0, 1)$ to
the temporal element $J_n = (t_{n-1}, t_n)$.
We refer to this combination of $E^k \times F^k$ as
$\textrm{iE}^\star$ (adjoint implicit Euler).
The motivation for this definition is as follows.
Consider
the adjoint (backward) ODE
\[
	\label{e:adj-ODE}
	-v' + \lambda v = f,
	\quad
	v(T) = 0,
\]
with a given $f$ that
for the sake of argument
is
piecewise affine with respect to $\cT$.
Define the approximate continuous piecewise affine solution $v \in F^k$
(hence, $v(T) = 0$)
through the implicit Euler time-stepping scheme
\emph{backward in time}:
\[
	\label{e:iE*}
	-
	\tfrac{1}{k_n}
	( v(t_{n}) - v(t_{n-1}) )
	+
	\lambda
	v(t_{n-1})
	=
	f(t_{n-1}^+)
	,
	\quad
	n = N,\ldots,1
	,
\]
where
$t_{n-1}^+$ denotes the limit from above.
We shall use the obvious abbreviations $v_n$ and $f_{n-1}^+$
when referring to \eqref{e:iE*}.
The definition of the discrete trial space $E^k$
implies
that the time-step condition~\eqref{e:iE*}
is equivalent to
the variational requirement
\[
	\label{e:iE*-vf}
	\int_{J_n}
	\duality{
	w,
	-v' + \lambda v - f
	}{}
	\rd t
	=
	0
	\quad
	\forall w \in E^k
	\quad
	\forall n = N,\ldots,1
	.
\]
The equivalence
is due to
the identity
$
\int_0^1 \phi(s) (a s + b) \rd s
=
b
$
for all real $a$ and $b$,
which implies that the integral in
\eqref{e:iE*-vf}
is a multiple of
$(-v' + \lambda v - f)(t_{n-1}^+)$.

The role of the adjoint ODE \eqref{e:adj-ODE}
is elucidated in the proof of the following proposition
that bounds
the inf-sup constant \eqref{e:ga-EF}
for the $\textrm{iE}^\star$ discretization.
The result is formulated in terms of
the backward successive temporal element ratio
\[
	\label{e:sigma}
	\sigma
	:=
	\max_{n = 1, \ldots, N-1} k_n / k_{n+1}
	.
\]

\begin{proposition}
	\label{p:iE*}
	The inf-sup condition \eqref{e:ga-EF}
	holds
	for the $\textrm{iE}^\star$ discretization
    with
	\[
		\label{e:p:iE*}
		\gamma_k \geq
		\gamma_\sigma
		:=
		1 / \sqrt{2(1 + \max\{1, \sigma\})}
		,
	\]
	uniformly in $\lambda > 0$.
\end{proposition}

Thus,
in order
to obtain uniform stability of the $\textrm{iE}^\star$ discretization
it suffices to ensure that
the backward successive temporal element
ratio~\eqref{e:sigma}
stays bounded.
This is verified numerically in Figure \ref{f:CN-iE}.
We generated an initial temporal mesh for $T = 1$
with $129$ nodes
by distributing
the inner nodes in interval $(0, 1)$ uniformly at random.
New nodes were inserted
by subdividing large temporal elements into two equal ones
until $\sigma \leq 3$,
leading to a temporal mesh with 210 nodes.
On this new temporal mesh,
we observe that
the inf-sup constant of
the $\textrm{iE}^{\star}$ discretization
is controlled as in~\eqref{e:p:iE*},
while that of $\textrm{CN}^{\star}$
depends strongly
on the spectral parameter $\lambda$, as already
explained in \S\ref{s:dis:1:CN*}.

\begin{proof}[Proof of Proposition \ref{p:iE*}]
	Let $w \in E^k$ be arbitrary nonzero.
	We will find a discrete $v \in F^k$
	such that
	$ 
	b(w, v)
	\geq
	\gamma_\sigma
	\norm{ w }{E}
	\norm{ v }{F}
	.
	$ 
	To this end, consider the adjoint ODE \eqref{e:adj-ODE}
	with $f := \lambda w$.
	If we took $v$ as the exact solution
	we would obtain
	$
	b(w, v)
	=
	\norm{w}{E}^2
	=
	\lambda^{-1} \norm{-v' + \lambda v}{L_2(J)}^2
	=
	\norm{v}{F}^2
	$.
	However,
	the exact solution is not necessarily an element of
	the discrete test space $F^k$,
	so we
	take
	$v \in F^k$ according to
	the implicit Euler scheme~\eqref{e:iE*}
	instead.
	By the equivalence
	of~\eqref{e:iE*}--\eqref{e:iE*-vf}
	we see that
	$ 
	b(w, v)
	=
	\int_J \duality{ w, -v' + \lambda v }{} \rd t
	=
	\int_J \duality{ w, \lambda w }{} \rd t
	=
	\norm{w}{E}^2
	$ 
	still holds.

	To conclude, it is enough to establish
	$\norm{w}{E} \geq \gamma_\sigma \norm{v}{F}$.
	To that end, we square \eqref{e:iE*} 
	with $f := \lambda w$
	on both sides
	and rearrange to obtain
	\[
		\label{e:p:iE*:n}
		\lambda^{-1} k_n^{-1} |v_n - v_{n-1}|^2
		+
		\lambda k_n |v_{n-1}|^2
		+
		|v_n - v_{n-1}|^2
		+
		|v_{n-1}|^2 - |v_n|^2
		=
		\lambda k_n |w_{n-1}^+|^2
		.
	\]
	Let
	$i_k v$ the denote the piecewise constant function with
	$i_k v(t_{n-1}^+) = v(t_{n-1})$
	for all $n = 1, \ldots, N$.
	The dual of $E$ is identified via the (unweighted) $L_2(J)$ inner product.
	We introduce
	the mesh-dependent norm
	\[
		\label{e:iii-F1}
		\normiii{v}{F}^2
		& :=
		\textstyle
		\norm{v'}{E'}^2
		+
		\norm{i_k v}{E}^2
		+
		|v(0)|^2
		+
		\sum_{n = 1}^N
		|v_n - v_{n-1}|^2
	\]
	and sum up~\eqref{e:p:iE*:n} over $n$.
	This yields the equality
	$\norm{w}{E} = \normiii{\tfrac12 v}{F}$,
	since
	$\int_0^1 |\phi(s)|^2 \rd s = 4 = \tfrac14 |\phi(0)|^2$.
	With $\sigma$ from~\eqref{e:sigma}
	we obtain the estimate
	(the last term is omitted for $n = N$)
	\[
	\label{e:p:iE*:In}
		\norm{v}{L_2(J_n)}^2
		\leq
		\tfrac12 k_n (|v_{n-1}|^2 + |v_n|^2)
		\leq
		\tfrac12 \norm{i_k v}{L_2(J_n)}^2
		+
		\tfrac12 \sigma \norm{i_k v}{L_2(J_{n+1})}^2
		.
	\]
	Summation over $n$ yields
	$
	\norm{v}{F}^2
	\leq
	2 (1 + \max\{1, \sigma\})
	\normiii{\tfrac12 v}{F}^2
	$.
	In concatenation,
	$ 
	\norm{w}{E}
	=
	\normiii{\tfrac12 v}{F}
	\geq
	\gamma_\sigma
	\norm{v}{F}
	,
	$ 
	as anticipated.
\end{proof}

The choice of the shape function
$\phi : s \mapsto (4 - 6 s)$ in
the trial space $E^k$
defining
the $\textrm{iE}^{\star}$ discretization
leads to uniform stability
as discussed above.
In view of the quasi-optimality estimate \eqref{e:u-qo}
we need to address
the approximation properties of this trial space $E^k$.
Unfortunately,
we do not have nestedness $E^k \subset E^{k+1}$.
Moreover, no matter how fine the temporal mesh,
$E^k$ does not approximate
the constant function.
To be precise,
let $P_d$ denote the $L_2$-orthonormal Legendre polynomial
(normalized to $P_d(1) = \sqrt{1 + 2d}$)
of degree $d\geq 0$ on the reference interval $(0,1)$.
For real $a, b$,
set $u := a P_0 + b P_1 + r$, where $r$ is
$E$-orthogonal to $P_0$ and $P_1$.
The $E$-orthogonal projection of $u$ onto
the span of the shape function
$\phi = P_0 - \sqrt{3} P_1$
is
$w := c \phi$ with $c = \tfrac14 ( a - \sqrt{3} b )$.
The error
$\norm{ u - w }{E}^2 = \lambda \tfrac14 |\sqrt{3} a + b|^2 + \norm{r}{E}^2$
may be large,
for example, if $u$ is constant.

\subsubsection{Common framework}
\label{s:dis:1:both}

On the $n$-th element
of the temporal mesh $\cT$ in~\eqref{e:cT},
let
{$
	\cN_n \subset
	\IntervalCO{t_{n-1}, t_n}
$}
be a set of $p \geq 1$ collocation nodes
(we choose the same $p$ for all $n$ for simplicity).
The compound element-wise
interpolation operator
based on these collocation nodes~$\cN_n$
is denoted by~$i_k$.
As the discrete test space $F^k \subset F$, we take the subspace
of continuous piecewise polynomials of degree $p$ with respect to~$\cT$. 
%
%
We introduce $i_k^\star \from i_k F^k \to F^k$ by
$\scalar{ i_k \CDOT, \CDOT }{L_2(J)} = \scalar{ \CDOT, i_k^\star \CDOT }{L_2(J)}$
on $F^k \times i_k F^k$.
The discrete trial space is then defined as $E^k := i_k^\star i_k F^k$.
Note that the dimensions $\dim E^k = \dim F^k$ match.

We are interested in two types of nodes:
Gauss--Legendre nodes
and
(left) Gauss--Radau nodes,
to which we refer as $\GL_p$ and $\GR_p$, respectively.
All temporal elements host the same type of nodes.
The lowest-order examples are
$\cN_n = \{ \tfrac12(t_{n-1} + t_n) \}$ for $\GL_1$
and
$\cN_n = \{ t_{n-1} \}$ for $\GR_1$,
corresponding to
the $\text{CN}^\star$ and $\text{iE}^\star$ schemes.
The shape functions on the reference element $(0, 1)$
for the space $E^k = i_k^\star i_k F^k$
are
(cf.~\cite[\S2.3]{AndreevSchweitzer2014})
\begin{enumerate}
	\item
		the Legendre polynomials $P_0, \ldots, P_{p-1}$
		for $\GL_p$,
		and
	\item
		the Legendre polynomials
		$P_0, \ldots, P_{p-2}$ together with $P_{p-1} - \tfrac{P_p(1)}{P_{p-1}(1)} P_p$
		for $\GR_p$.
\end{enumerate}
In particular,
for $p \geq 2$,
the $\GR_p$ family contains the piecewise constant functions,
which means that any function in $E$
can be approximated to arbitrary accuracy upon mesh refinement.

Define the mesh-dependent norm
$\normiii{\CDOT}{F}$ by
\begin{equation*}
	%
	\normiii{v}{F}^2
	  :=
	\norm{v'}{E'}^2
	+
	\norm{i_k v}{E}^2
	+
	|v(0)|^2
	+
	\begin{cases}
	0
	  & \text{for $\GL_p$},
	\\
	\sum_{n = 1}^N
	[v - i_k v]_{\to n}^2
	  & \text{for $\GR_p$}
	,
	\end{cases}
\end{equation*}
where $[f]_{\to n}$ denotes
$\lim_{t \to t_n^-} f(t)$.
This is the generalization of \eqref{e:iii-F1}.

Following~\cite[Proof of Thm.~3.3]{AndreevSchweitzer2014},
we can now show:
\begin{lemma}
	\label{l:b-stab}
	For any $v \in F^k$
	there exists a nonzero $w \in E^k = i_k^{\star} i_k F^k$
	such that
	\[
		\label{e:b-stab}
		b(w, v)
		\geq
		\norm{(i_k^\star)^{-1} w}{E}
		\normiii{v}{F}
		.
	\]
\end{lemma}

\begin{proof}
	The space $i_k F^k \subset E$ carries the norm of $E$.
	Let $v \in F^k$.
	We first show that
	$\norm{ \Gamma v }{E} = \normiii{ v }{F}$,
	where
	$\Gamma : F^k \to i_k F^k$ is defined by
	\begin{equation*}
		\scalar{ \Gamma v, \tilde{w} }{E}
		=
		b(i_k^\star \tilde{w}, v)
		\quad
		\forall
		(v, \tilde{w}) \in F^k \times i_k F^k
		.
	\end{equation*}
	To this end, we expand
	$
	\norm{ \Gamma v }{E}^2 =
	\norm{ \Gamma v - i_k v }{E}^2
	+
	2 \scalar{ \Gamma v, i_k v }{E}
	-
	\norm{ i_k v }{E}^2
	$.
	For the first term we have
	\begin{equation*}
		\norm{ \Gamma v - i_k v }{E}
		=
		\sup_{ \tilde{w} \in S(i_k F^k) }
		\scalar{ \Gamma v - i_k v, \tilde{w} }{E}
		=
		\sup_{ \tilde{w} \in S(i_k F^k) }
		\left\{
		b(i_k^\star \tilde{w}, v)
		-
		\scalar{ i_k v, \tilde{w} }{E}
		\right\}
		=
		\norm{ v' }{E'}
		.
	\end{equation*}
	For the second term,
	we use the definition of $\Gamma$,
	followed by \cite[Lem.~3.1]{AndreevSchweitzer2014}:
	\begin{equation*}
		\scalar{ \Gamma v, i_k v }{E}
		=
		\norm{ i_k v }{E}^2
		-
		\scalar{ i_k v, v' }{L_2(J)}
		=
		\norm{ i_k v }{E}^2
		+
		\tfrac12 |v(0)|^2
		+
		\begin{cases}
		0
		  &
		(\GL_p)
		,
		\\
		\tfrac12 \sum_{n=1}^N [v - i_k v]_{\to n}^2
		  & (\GR_p)
		.
		\end{cases}
	\end{equation*}
	Hence, $\norm{ \Gamma v }{E} = \normiii{ v }{F}$.
	Now take $\tilde{w} := \Gamma v$.
	Then
	$
	b(i_k^\star \tilde{w}, v)
	=
	\scalar{ \Gamma v, \tilde{w} }{E}
	=
	\norm{ \Gamma v }{E}^2
	=
	\norm{ \tilde{w} }{E}
	\normiii{ v }{F}
	.
	$
	The claim \eqref{e:b-stab} follows
	for $w := i_k^\star \tilde{w}$.
\end{proof}

In order to convert \eqref{e:b-stab}
to a statement with the original norms,
we need to compare these norms.
First,
it can be shown as in \cite[\S3.2.2]{AndreevSchweitzer2014}
that
$ 
\norm{w}{E}
\leq
\norm{i_k^\star}{}
\norm{(i_k^\star)^{-1} w}{E}
\leq
2
\norm{(i_k^\star)^{-1} w}{E}
.
$ 

Second, we need to quantify $\norm{v}{F} \lesssim \normiii{v}{F}$.
For
the Gauss--Radau family $\GR_p$
we can, for example, use
the estimate
(akin to~\eqref{e:p:iE*:In}; see~\cite[\S3.4]{AndreevSchweitzer2014})
\begin{equation*}
	\norm{ v - i_k v }{L_2(t_{n-1}, t_n)}^2
	\leq
	\tfrac{ 2 p^2 }{ 4 p - 1 / p }
	\left(
	\norm{ i_k v }{L_2(t_{n-1}, t_n)}^2
	+
	\tfrac{ k_n }{ k_{n+1} }
	\norm{ i_k v }{L_2(t_n, t_{n+1})}^2
	\right)
\end{equation*}
to derive
$ 
\norm{v}{F} \leq C \sqrt{p (1 + \sigma)} \normiii{v}{F}
$ 
with the backward successive temporal element ratio~$\sigma$
from~\eqref{e:sigma} and
a universal constant $C > 0$.
Therefore,
the discrete inf-sup condition~\eqref{e:ga-EF}
holds for the $\GR_p$ family
with
\[
	\label{e:ga-p}
	\gamma_k \geq \gamma_0 / \sqrt{p (1 + \sigma)}
	,
\]
where $\gamma_0 > 0$ is a constant independent of all parameters.
The Gauss--Legendre family $\GL_p$
suffers from the same
potential instability
as the $\text{CN}^\star$ scheme,
see \S\ref{s:dis:1:CN*}.

Consider now the solution $u^k$
to \eqref{e:uk}.
From the ODE \eqref{e:uode},
the reconstruction
\begin{equation*}
	%
	\widehat{u}^k := g + \int_0^t \{ f(s) - \lambda u^k(s) \} \rd s
\end{equation*}
can be expected to provide a better approximation
of the exact solution.
With~\eqref{e:uk} we find
the orthogonality property
$\scalar{\widehat{u}^k - u^k, v'}{E} = 0$
for all $v \in F^k$.
Let
\[
	\label{e:Q}
	q_k \from E \to \partial_t F^k
\]
be the orthogonal projection
(in $E$ or in $L_2(J)$).
The orthogonality property gives
$q_k \widehat{u}^k = q_k u^k$.
Hence,
the postprocessed solution
$\bar{u}^k := q_k u^k$
is an approximation of the reconstruction $\widehat{u}^k$.
In the case of Gauss--Legendre collocation nodes,
$i_k^\star$ is the identity,
so that $E^k = i_k F^k$,
and therefore $q_k u^k = u^k$ has no effect.
In the Gauss--Radau case, however,
the projection is useful to improve
the convergence rate upon mesh refinement,
as will be seen in \S\ref{s:convergence}.

Note that $q_k$ is injective on $E^k$ in both cases.
In the Gauss--Radau case,
$q_k^{-1}$ sends
the shape function
$P_{p-1}$
to
$P_{p-1} - \tfrac{P_p(1)}{P_{p-1}(1)} P_p$.
Since $P_d(1) = \sqrt{2 d + 1}$, this gives
\[
	\label{e:normqk}
	\norm{ q_k^{-1} }{}^2
	=
	1 + \tfrac{ 2 p + 1 }{ 2 (p-1) + 1 }
	.
\]

\subsection{Petrov--Galerkin approximations}
\label{s:dis:PG}

In this subsection we
comment on
Petrov--Galerkin
discretizations of
the generic linear variational problem
\begin{equation*}
	\text{Find}
	\quad
	u \in X
	\TEXT{s.t.}
	\duality{Bu, v}{} = \duality{\ell, v}{}
	\quad
	\forall v \in Y
	,
\end{equation*}
where
$X$ and $Y$ are
\emph{normed vector spaces}.
This generalization away from Hilbert spaces
(that can also be found e.g.~in \cite{Stern2015})
will
allow us
to address
the variational problem
\eqref{e:U-mul}.

We assume that
$X_h \times Y_h \subset X  \times Y$
are finite-dimensional subspaces
with nonzero $\dim X_h = \dim Y_h$.
Here, $h$ refers to the ``discrete'' nature
of these subspaces,
and
the pair $X_h \times Y_h$ is fixed.
We write
$\norm{ \CDOT }{Y_h'} := \sup_{v \in S(Y_h)} |\duality{ \CDOT, v }{}|$.

In order
to admit variational crimes
we suppose
that we have access
to an operator
$\bar{B} \from X \to Y'$
that approximates $B$
(although $\bar{B} \from X \to Y_h'$ suffices).
For this approximation
we assume the discrete inf-sup condition
in the form of a constant
$\bar{\gamma}_h > 0$
such that
$\norm{ \bar{B} w_h }{Y_h'} \geq \bar{\gamma}_h \norm{ w_h }{X}$
for all $w_h \in X_h$.
The proof
of the following Proposition
is obtained by standard arguments
(for the discussion of the constant ``$1 + $''
see \cite{Andreev-QUASI,Stern2015,XuZikatanov2003}).

\begin{proposition}
	\label{p:well:crime}
	Fix $u \in X$.
	Under the above assumptions
	there exists a unique $u_h \in X_h$
	such that
	\begin{equation*}
		\duality{ \bar{B} u_h, v_h }{} = \duality{ B u, v_h }{}
		\quad
		\forall v_h \in Y_h
		.
	\end{equation*}
	The mapping
	$u \mapsto u_h$ is linear with
	$\norm{u_h}{X} \leq \bar{\gamma}_h^{-1} \norm{ B u }{Y_h'}$,
	and satisfies
	the quasi-optimality estimate
	\begin{equation*}
		%
		\norm{ u - u_h }{X}
		\leq
		(1 + \bar{\gamma}_h^{-1} \norm{ \bar{B} }{})
		\inf_{w_h \in X_h}
		\norm{ u - w_h }{X}
		+
		\bar{\gamma}_h^{-1}
		\norm{(B - \bar{B}) u }{Y_h'}
		.
	\end{equation*}
\end{proposition}

\subsection{Tensorized discretizations}
\label{s:dis:1:o}

Recall
the definition of the tensor product spaces
$E_{2/\pi}$ and $F_{2/\epsilon}$
from~\eqref{e:E2F2} and~\eqref{e:EpiFeps}.
Recall also that
we can extend $B := (b \otimes b)$
to an isometric isomorphism
$B\from E_2 \to F_2'$
or
$B\from E_\pi \to F_\epsilon'$.
We discuss here these two viewpoints in parallel.
Consider the variational formulation
\[
	\label{e:BUL}
	\text{Find}
	\quad
	U \in E_{2/\pi}
	\quad\text{s.t.}\quad
	B(U, v)
	=
	\ell(v)
	\quad
	\forall v \in F_{2/\epsilon}
	,
\]
where $\ell \in F_{2/\epsilon}'$.
If $E^k \times F^k$ is a discretization for~\eqref{e:uode-vf}
then
the pair of tensorized subspaces
\[
	\label{e:EkxFk}
	E_{2/\pi}^k \times F_{2/\epsilon}^k
	:=
	(E^k \otimes E^k) \times (F^k \otimes F^k)
	\subset
	E_{2/\pi} \times F_{2/\epsilon}
\]
is a natural choice
for the discretization for \eqref{e:BUL}.
The subscript $2$ or $\pi$ (and $2$ or $\epsilon$)
indicates which norm
the algebraic tensor product
$E^k \otimes E^k$ (and $F^k \otimes F^k$)
is equipped with;
since these spaces are finite-dimensional,
no norm-closure is necessary.

We now turn to the discrete variational formulation
\[
	\label{e:BUL-k}
	\text{Find}
	\quad
	U^k \in E_{2/\pi}^k
	\quad\text{s.t.}\quad
	B(U^k, v)
	=
	\ell(v)
	\quad
	\forall v \in F_{2/\epsilon}^k
	.
\]
The inf-sup constant
required in the analysis
is the square $\gamma_k^2$
of the discrete inf-sup constant
$\gamma_k$ from~\eqref{e:ga-EF}
in both cases:
\[
	\label{e:ga2}
	\inf_{w \in S(E_2^k)}
	\sup_{v \in S(F_2^k)}
	B(w, v)
	=
	\gamma_k^2
	=
	\inf_{w \in S(E_\pi^k)}
	\sup_{v \in S(F_\epsilon^k)}
	B(w, v)
	.
\]
Indeed, consider the $\pi/\epsilon$ situation.
For $w \in E^k$ let $b_k w$ denote
the restriction of $b w$ to $F^k$.
The discrete inf-sup condition~\eqref{e:ga-EF}
says that
$b_k \from E^k \to (F^k)'$ is an isomorphism with
$\norm{b_k^{-1}}{} = \gamma_k^{-1}$.
The mapping
$B_k := b_k \otimes b_k \from E^k \otimes_\pi E^k \to (F^k)' \otimes_\pi (F^k)'$
has the inverse $b_k^{-1} \otimes b_k^{-1}$.
It is therefore an isomorphism
with
$\norm{B_k^{-1}}{} = \gamma_k^{-2}$.
The identification
$(F^k)' \otimes_\pi (F^k)' \cong (F_\epsilon^k)'$
shows that for any $w \in E_\pi^k$,
the functional $B_k w$
is the restriction of $B w$ to $F_\epsilon^k$.
This gives~\eqref{e:ga2}.

Proposition~\ref{p:well:crime}
(with $\bar{B} := B$)
provides
a unique solution $U^k \in E^k \otimes E^k$
to the discrete variational problem~\eqref{e:BUL-k}
that approximates the solution $U$ of \eqref{e:BUL}
as soon as $\gamma_k > 0$ in \eqref{e:ga-EF}.
The solution is, moreover, quasi-optimal
(recall that $\norm{B}{} = 1$):
\[
	\label{e:Ukquasiopt}
	\norm{ U - U^k }{2/\pi}
	\leq
	(1 + \gamma_k^{-2})
	\inf_{w \in E^k \otimes E^k}
	\norm{ U - w }{2/\pi}
	.
\]

We will also be interested in the postprocessed solution
{$
	\bar{U}^k := (q_k \otimes q_k) U^k
	,
$}
where $q_k \from E \to \partial_t F^k$ is
the orthogonal projection in~\eqref{e:Q}.

Analogously to
Lemma \ref{l:psd}
one proves:

\begin{lemma}
	\label{l:discrete-psd}
	The discrete solution~$U^k$ to~\eqref{e:BUL-k}
	is SPSD
	if and only if $\ell$ is
	SPSD on $F^k \otimes F^k$.
	The same is true for the postprocessed solution.
\end{lemma}

\subsection{Second moment discretization: additive noise}
\label{s:dis:2-add}

In view of the previous section,
any discretization pair
$E^k \times F^k$ satisfying the discrete inf-sup condition~\eqref{e:ga-EF}
induces a valid discretization of the variational problem~\eqref{e:U-add}
for the second moment of the solution process
to
the stochastic ODE with additive noise~\eqref{e:sode-add}
if we choose
the trial space as $E^k \otimes E^k$
and
the test space as $F^k\otimes F^k$.
The functional on the right-hand side of \eqref{e:BUL}
is then
$
\ell
:=
\e{X_0^2} (\Dirac_0 \otimes \Dirac_0) + \mu^2 \diag
.
$
Moreover,
the discrete solution satisfies
the quasi-optimality estimates in~\eqref{e:Ukquasiopt}
simultaneously with respect to
$\norm{\CDOT}{2}$ and $\norm{\CDOT}{\pi}$,
because $\ell \in F_\epsilon' \subset F_2'$.

\subsection{Second moment discretization: multiplicative noise}
\label{s:dis:2-mul}

As in the continuous case
for sufficiently small values of the volatility $\rho$,
namely in the range
\[
	\label{e:range}
	0 \leq \rho^2 < 2 \lambda \gamma_k^2
	,
\]
we immediately obtain
a discrete inf-sup condition
for
the operator $B - \rho^2 \DIAG$.
The purpose of this section is to address
the whole range $\rho \geq 0$.

We will focus on the
$\text{CN}^\star$ and $\text{iE}^\star$ discretizations
discussed in \S\S\ref{s:dis:1:CN*}--\ref{s:dis:1:iE*},
although with some work,
our methods may be adapted
to higher-order schemes from \S\ref{s:dis:1:both}.
Throughout,
we
assume that the discretization pair $E^k \times F^k \subset E \times F$
satisfies
the discrete inf-sup condition~\eqref{e:ga-EF}.
The discrete trial and test spaces
{$
	E_\pi^k \times F_\epsilon^k
	\subset
	E_\pi \times F_\epsilon
$}
are defined as
in \eqref{e:EkxFk}.

We introduce some more notation.
In what follows,
the default range of the indices
(we use $m$ as an index, since the first moment does not appear anymore)
is
\begin{equation*}
	0 \leq i, j \leq N-1
	\TEXT{and}
	1 \leq m, n \leq N
	.
\end{equation*}
Recall that
the discrete test space $F^k \subset F$
consists of continuous piecewise affine functions with respect
to the temporal mesh $\cT$ in~\eqref{e:cT}
that vanish at the terminal time~$T$.
It is equipped with the hat function basis $\{v_i\}_i$,
determined by $v_i(t_j) = \krondelta_{i j}$.
The basis functions $\{ e_n \}_n$ of the discrete trial space $E^k \subset E$
are
supported on $\operatorname{supp}(e_n) = [t_{n-1}, t_n]$
in both schemes.
Specifically,
$e_n$
is a constant for $\text{CN}^\star$
and
is a dilated translate of
the shape function $\phi : s \mapsto (4 - 6 s)$ for $\text{iE}^\star$.
The following statements
do not depend on the scaling of the basis functions,
if not specified otherwise.

\subsubsection{The discrete problem}

In the multiplicative case,
the trace product~$\DIAG$ from~\eqref{e:def:DIAG}
appears in the variational problem~\eqref{e:U-mul}
for the second moment.
The basis functions $\{e_n\}_n \subset E^k$
for the $\text{iE}^\star$ discretization
lead to an inconsistency in the $\DIAG$ term,
see \S\ref{s:iE2}.
For this reason, we introduce
the approximate trace product
\[
	\label{e:DIAGk}
	\DIAG^k \from E_\pi \times F_\epsilon \to \IR
	,
\]
to be specified below.
We require
that $\DIAG^k$ reproduces the following properties of the exact
trace product $\DIAG$:
\begin{enumerate}[label={(\roman*)}]
	\item\label{enum:DeltaSymm}

	\emph{Symmetry and definiteness}:
	for every SPSD $w \in E^k_\pi$,
	the functional $\DIAG^k w$ is
	SPSD on $F^k \otimes F^k$, i.e.,
	\begin{equation*}
		%
		\DIAG^k(w, \psi \otimes \tilde\psi) = \DIAG^k(w, \tilde\psi \otimes \psi)
		\TEXT{and}
		\DIAG^k(w, \psi \otimes \tilde\psi) \geq 0
		\quad
		\forall \psi, \tilde\psi \in F^k.
	\end{equation*}
	\item\label{enum:DeltaSparse}
	\emph{Locality}:
	\begin{equation*}
		\DIAG^k(e_m \otimes e_n, v_i \otimes v_j) \neq 0
		\quad
		\text{only if}
		\quad
		m=n
		\text{\; and \;}
		i, j \in \{ n-1, n \}.
	\end{equation*}

	\item\label{enum:DeltaBdd}

	\emph{Bilinearity and continuity on $E_\pi \times F_\epsilon$}.
\end{enumerate}
The corresponding approximation of the operator $\cB$
is defined as
$\cB^k := B - \rho^2 \DIAG^k$.
We are now interested in the solution
of the discrete variational problem
\[
	\label{e:vfk}
	\text{Find}
	\quad
	U^k\in E^k_\pi
	\quad
	\text{s.t.}
	\quad
	\cB^k(U^k,v) = \ell(v)
	\quad
	\forall v \in F^k_\epsilon
\]
which approximates
\eqref{e:U-mul-L}.

\subsubsection{Well-posedness of the discrete problem}

The solution~$U^k$ to~\eqref{e:vfk} can be expanded
in terms of the basis
$\{e_m \otimes e_n\}_{mn}$ of $E_\pi^k$ as
\[
	\label{e:Ukexpansion}
	U^k = \sum_{mn} U_{mn} (e_m \otimes e_n)
	\quad
	\text{with}
	\quad
	U_{mn}
	=
	\tfrac{\scalar{U^k, e_m \otimes e_n}{2}}{\norm{e_m}{E}^2 \norm{e_n}{E}^2}
	.
\]
We combine its coefficients in the $N\times N$ matrix $\mathbf{U} := (U_{m n})_{mn}$.
Furthermore,
we define the values
\begin{equation*}
	%
	b_{in} := b(e_n, v_i)
	\TEXT{and}
	\ell_{ij} := \ell(v_i \otimes v_j)
	.
\end{equation*}
If the discrete inf-sup condition~\eqref{e:ga-EF} is satisfied
then
$b_{n-1,n} \neq 0$ follows.

The sparsity assumption on $\DIAG^k$
together with the fact that the discretization pair
$E^k_\pi \times F^k_\epsilon$
is a tensor product discretization
allow for an explicit formula
for the diagonal entries of $\mathbf{U}$.
This is presented in the lemma below.

For future purpose, we note that $w \in E^k \otimes E^k$
is SPSD
if and only if
the matrix of coefficients
$\mathbf{w} := (w_{m n})_{mn}$
with respect to $\{e_m \otimes e_n\}_{mn}$ is.
Indeed, if $\varphi \in L_2(J)$
and
$\boldsymbol{\varphi} = (\scalar{e_n,\varphi}{L_2(J)})_n \in \IR^N$
then
{$
	\boldsymbol{\varphi}^\T \mathbf{w} \boldsymbol{\varphi}
	=
	\sum_{mn} w_{mn} \scalar{e_m,\varphi}{L_2(J)} \scalar{e_n,\varphi}{L_2(J)}
	=
	\scalar{w, \varphi\otimes\varphi}{L_2(J\times J)}
	.
$}

According to the locality assumption~\ref{enum:DeltaSparse},
the nonzero
values of $\DIAG^k$
(as acting on the basis functions)
can be combined in the $2 \times 2$ matrices
\[
	\label{e:DIAGn}
	\boldsymbol{\DIAG}^n
	:=
	\begin{pmatrix}
	\DIAG^k(e_n \otimes e_n, v_{n-1} \otimes v_{n-1}) & \DIAG^k(e_n \otimes e_n, v_{n-1} \otimes v_n)
	\\
	\DIAG^k(e_n \otimes e_n, v_n     \otimes v_{n-1}) & \DIAG^k(e_n \otimes e_n, v_n     \otimes v_n)
	\end{pmatrix},
	\quad
	1 \leq n \leq N-1
	,
\]
and in $\boldsymbol{\DIAG}^N := \DIAG^k(e_N\otimes e_N, v_{N-1}\otimes v_{N-1})$.
The foregoing remark and
Assumption~\ref{enum:DeltaSymm} on $\DIAG^k$
imply
that
each $\boldsymbol{\DIAG}^n$
is SPSD.

We define
\[
	\label{e:betan}
	\beta_n
	:=
	(
		1 - \rho^2 b_{n-1,n}^{-2} \DIAG^n_{11}
	)^{-1},
	\quad
	n = 1, \ldots, N
	,
\]
where
$\DIAG^n_{pq}$ denotes the $(p,q)$-th entry in the matrix $\boldsymbol{\DIAG}^n$,
and for $n \geq 2$:
\[
	\label{e:thetan}
	\theta_n & :=   b_{n-1,n}^{-1} b_{n-1,n-1},  \qquad
	\\
	\label{e:alphan}
	\alpha_n & := \beta_n
	\left[
		\theta_n^2 +
		\rho^2  b_{n-1,n}^{-2}
		\left(
			\DIAG^{n-1}_{22} - 2 b_{n-2,n-1}^{-1} b_{n-1,n-1} \DIAG^{n-1}_{12}
		\right)
	\right].
\]
We note that
\[
	\label{e:abt-scale}
	\tfrac{\norm{e_n}{E}^2}{\norm{e_{n-1}}{E}^2} \alpha_n
	,
	\quad
	\tfrac{\norm{e_n}{E}}{\norm{e_{n-1}}{E}} \theta_n
	,
	\TEXT{and}
	\beta_n
\]
do note depend on the scaling of the basis $\{ e_n \}_n$.

For technical reasons we
also introduce
the function $G^k \in E_\pi^k$ as the
solution
(which is well-defined under the inf-sup condition \eqref{e:ga-EF}/\eqref{e:ga2})
to
\[
	\label{e:vfgk}
	\text{Find}
	\quad
	G^k \in E^k_{\pi}
	\quad\text{s.t.}\quad
	B(G^k, v) = \ell(v)
	\quad
	\forall v \in F^k_\epsilon.
\]
Let $G_{m n}$ denote its coefficients
with respect to $\{e_m \otimes e_n\}_{mn}$.

\begin{lemma}
	\label{l:Uknn}
	Let $\ell \in F_\epsilon'$.
	Assume that $\beta_n$ is finite for all $n$.
	Then there exists a unique solution $U^k \in E_\pi^k$
	to the discrete variational problem~\eqref{e:vfk}.
	Its diagonal coefficients
	in~\eqref{e:Ukexpansion}
	are
	\[
		\label{e:Uknn:kn}
		U_{nn} =
		\beta_n
		G_{nn} +
		\sum_{m=1}^{n-1}
		G_{mm}
		(\beta_m \alpha_{m+1} - \beta_{m+1} \theta_{m+1}^2 )
		\prod_{\nu = m+2}^{n} \alpha_\nu
		.
	\]
\end{lemma}

\begin{proof}
	By locality of the support of $e_n$ and $v_i$,
	the values
	$b_{in} = b(e_n, v_i)$
	are non-zero at most for $i \in \{n-1, n\}$.
	Therefore,
	the coefficients $\{w_n\}_n$
	of the solution $w\in E^k$ to
	the problem
	``$b(w, v) = f(v)$ for all $v \in F^k$''
	are obtained by recursion,
	\begin{equation*}
		b_{n-1,n} w_n
		= f(v_{n-1}) - b_{n-1,n-1} w_{n-1}
		= \sum_{j = 0}^{n-1} \Pi_j^{n-1} f(v_j)
		,
		\quad \text{where} \quad
		\Pi_j^n := \prod_{i = j+1}^{n} \frac{ -b_{ii} }{ b_{i-1,i} }.
	\end{equation*}
	Hence,
	the coefficients of the solution $G^k$
	to the tensorized problem~\eqref{e:vfgk}
	satisfy
	\[
		\label{e:Gmn}
		b_{m-1,m} b_{n-1,n} G_{mn} & = \sum_{i=0}^{m-1} \sum_{j=0}^{n-1} \Pi_{i}^{m-1} \Pi_{j}^{n-1} \ell_{ij}
		.
	\]
	Applying this formula to
	$B U = \ell + \rho^2 \DIAG^k U$
	instead of
	$B G = \ell$
	gives
	\[
		\label{e:Umn}
		b_{m-1,m} b_{n-1,n} U_{mn} & =
		b_{m-1,m} b_{n-1,n} G_{mn}
		+
		\rho^2 \sum_{i=0}^{m-1} \sum_{j=0}^{n-1} \Pi_{i}^{m-1} \Pi_{j}^{n-1} [\DIAG^k U^k]_{ij}
		.
	\]
	Due to the locality~\ref{enum:DeltaSparse}
	of $\DIAG^k$,
	the double sum
	contains only
	the diagonal
	coefficients $U_{r r}$ with $r \leq \min\{ m, n \}$
	and no off-diagonal ones;
	specifically,
	only the entries
	\begin{subequations}
		\label{e:[DIAG U]}
		\[
			[\DIAG^k U^k]_{r-1,r-1} & = U_{r-1,r-1} \DIAG^{r-1}_{22} + U_{rr} \DIAG^{r}_{11},
			\\
			[\DIAG^k U^k]_{r-2,r-1} & = U_{r-1,r-1} \DIAG^{r-1}_{12},
			\\
			[\DIAG^k U^k]_{r-1,r-2} & = U_{r-1,r-1} \DIAG^{r-1}_{21},
		\]
	\end{subequations}
	occur.
	In particular, if $m = n$ then
	the formula gives a recursion
	for $U_{n n}$
	with
	$\rho^2 \DIAG^n_{11} U_{n n}$
	on the right-hand side.
	Therefore,
	we can solve for $U_{n n}$
	if
	$b_{n-1,n}^2 \neq \rho^2 \DIAG^n_{11}$
	(which is equivalent to $\beta_n$ being finite).
	The formula then provides
	the remaining off-diagonal coefficients $U_{m n}$.
	With this,
	the existence of the discrete solution
	is established.

	To obtain the representation \eqref{e:Uknn:kn},
	we subtract
	from formula \eqref{e:Umn} for $U_{n n}$
	that for $U_{n-1, n-1}$.
	After some manipulation,
	this leads to the iteration
	\begin{equation*}
		U_{11} = \beta_1 G_{11},
		\qquad
		U_{nn} = \beta_n G_{nn} - \beta_n \theta_n^2 G_{n-1,n-1} + \alpha_n U_{n-1,n-1},
		\quad 2 \leq n \leq N
		,
	\end{equation*}
	and by induction to the claim \eqref{e:Uknn:kn}.
\end{proof}

Equation~\eqref{e:Uknn:kn}
is
the discrete version of the identity in~\eqref{e:ugrel},
which
was used to prove (see Theorem~\ref{t:stab-mul})
that
an SPSD right-hand side $\ell$
entails
the same property for the solution $U$.
The following lemma
characterizes
the conditions on the discretization parameters
for
which this is true in the discrete.

\begin{lemma}
	\label{l:Ukspsd}
	The following are equivalent:
	\begin{enumerate}[label={(\roman*)}]
		\item\label{l:Ukspsd:beta}
		$\beta_n > 0$ in \eqref{e:betan} for all $n$;

		\item\label{l:Ukspsd:spsdsol}
		For every SPSD $\ell \in F_\epsilon'$
		the discrete variational problem~\eqref{e:vfk} has
		a unique
		solution $U^k \in E^k_\pi$,
		and it is SPSD.
	\end{enumerate}
\end{lemma}

\begin{proof}
	Assume~\ref{l:Ukspsd:beta}.
	Let $\ell \in F_\epsilon'$ be SPSD.
	Then $G^k \in E^k_{\pi}$
	defined in~\eqref{e:vfgk}
	is also SPSD
	by Lemma~\ref{l:discrete-psd}.
	As remarked above, its matrix of coefficients
	is therefore also SPSD,
	in particular $G_{nn} \geq 0$.
	From this and \eqref{e:Uknn:kn},
	it follows that also $U_{nn} \geq 0$.
	Indeed,
	with \ref{l:Ukspsd:beta} $\beta_n > 0$, we obtain the equivalence
	\[
		\label{e:alpha-theta-rel}
		\beta_{n+1}^{-1} \alpha_{n+1} \geq \beta_{n}^{-1} \theta_{n+1}^2
		\quad
		\Leftrightarrow
		\quad
		(-b_{n-1,n}^{-1} b_{nn}, 1) \boldsymbol{\DIAG}^n (-b_{n-1,n}^{-1} b_{nn}, 1)^\T
		\geq 0
		.
	\]
	Since the matrices $\boldsymbol{\DIAG}^n$ are positive semi-definite,
	$\beta_{n} \alpha_{n+1} \geq  \beta_{n+1} \theta_{n+1}^2$ holds
	and, thus, $\alpha_{n+1} \geq 0$
	and $U_{nn} \geq 0$
	for all $n$.
	Set now $\widehat{U}^k := \sum_{n=1}^N U_{nn} (e_n\otimes e_n)$.
	Since the discrete inf-sup condition~\eqref{e:ga-EF} is assumed,
	there exists a unique $U^k \in E^k_\pi$ satisfying
	$B(U^k,v) = \widehat{\ell}(v)$
	for all $v \in F^k_\epsilon$,
	where
	$\widehat{\ell} := \rho^2 \DIAG^k \widehat{U}^k + \ell$.
	By Assumption~\ref{enum:DeltaSymm} on $\DIAG^k$,
	the functional~$\widehat{\ell}$
	is SPSD
	on $F^k\otimes F^k$.
	%
	%
	By Lemma~\ref{l:discrete-psd},
	$U^k$ is also SPSD.
	Moreover, the identity \eqref{e:Gmn}
	applied
	to the right-hand side $\widehat{\ell}$
	yields
	$ 
	b_{n-1,n}^2 U_{nn}
	= \sum_{i,j < n} \Pi_{i}^{n-1} \Pi_{j}^{n-1} [\rho^2 \DIAG^k \widehat{U}^k + \ell]_{ij}
	=
	b_{n-1,n}^2 \widehat{U}_{nn},
	$ 
	where the last equality follows from the definition of the
	coefficients $\widehat{U}_{nn} = U_{nn}$ and
	the locality properties \eqref{e:[DIAG U]}.
	Consequently, $\DIAG^k \widehat{U}^k = \DIAG^k U^k$ on $F^k_\epsilon$,
	and $U^k$ is the desired solution.

	Conversely, assume~\ref{l:Ukspsd:spsdsol}.
	For any $g_1, \ldots, g_N \geq 0$,
	the function $G^k := \sum_n g_n (e_n\otimes e_n) \in E_\pi^k \subset E_\pi$
	is SPSD.
	By Lemma~\ref{l:psd}, the functional
	{$
		\ell := B G^k
		\in
		F_\epsilon'
	$}
	inherits this property
	and, moreover, by assumption also the solution
	$U^k$ to~\eqref{e:vfk} is
	positive semi-definite.
	In particular, $U_{nn} \geq 0$.
	Fix $n\in\{1,\dots,N\}$ and choose $g_n = 1$ and $g_m = 0$ for all $m\neq n$.
	With this choice, the nonnegativity of $U_{nn}$
	along with its representation in~\eqref{e:Uknn:kn}
	imply that $\beta_n \geq 0$.
	Since $\beta_n$ is a fraction~\eqref{e:betan}, we
	conclude that \ref{l:Ukspsd:beta} $\beta_1,\ldots,\beta_N$ are positive.
\end{proof}

\subsubsection{Discrete stability and inf-sup}

The representation of $U_{nn}$ in~\eqref{e:Uknn:kn}
in combination with the Lemmas~\ref{l:2andpinorm} and~\ref{l:Ukspsd}
allow for an explicit representation of the
$E_\pi$-norm of the discrete solution:

\begin{corollary}
	\label{c:Ukpinorm}
	Suppose $\beta_n > 0$ in \eqref{e:betan} for all $n$.
	Let $\ell \in F_\epsilon'$ be SPSD.
	Then
	the discrete variational problem~\eqref{e:vfk}
	admits a unique solution $U^k \in E^k_\pi$.
	It is SPSD
	with norm
	\[
		\label{e:Ukpinorm:kn}
		\norm{U^k}{\pi} =
		\sum_{n=1}^{N} \left( \beta_n G_{nn}
		+ \sum_{m=1}^{n-1} G_{mm} (\beta_m \alpha_{m+1} - \beta_{m+1} \theta_{m+1}^2 )
		\prod_{\nu=m+2}^{n} \alpha_\nu \right)
		\norm{e_n}{E}^2.
	\]
\end{corollary}

\begin{proof}
	Lemmas~\ref{l:2andpinorm},~\ref{l:Uknn}
	and~\ref{l:Ukspsd}
	give
	$
	\norm{U^k}{\pi} = \lambda \diag(U^k)
	= \sum_{n=1}^{N} U_{nn} \norm{e_n}{E}^2.
	$
	Inserting the expression~\eqref{e:Uknn:kn} for $U_{nn}$
	yields~\eqref{e:Ukpinorm:kn}.
\end{proof}

From Corollary~\ref{c:Ukpinorm},
the norm of the discrete solution $U^k$ can be
estimated in terms of the norm of
the right-hand side $\ell$.
We shall do this under the additional assumption
of a uniform temporal mesh.
For convenience of notation,
we rescale the basis $\{e_n\}_n$ to $\norm{e_n}{E} = 1$,
so that in view of \eqref{e:abt-scale},
the numbers
$(\alpha, \beta, \theta) := (\alpha_n, \beta_n, \theta_n)$
do not depend on $n$
(cf.~Table \ref{t:stab1}).

\begin{theorem}
	\label{t:stab:dis}
	In addition
	to the conditions posed in Corollary~\ref{c:Ukpinorm},
	assume that the temporal mesh is uniform.
	Then the discrete solution $U^k$ to~\eqref{e:vfk}
	satisfies the stability bound
	\[
		\label{e:stab>1}
		\norm{U^k}{\pi}
		\leq
		C_k
		\norm{\ell}{-\epsilon}
		\quad
		\text{with}
		\quad
		C_k
		:=
		\gamma_k^{-2}
		\beta
		(
		1 + (\alpha -  \theta^2 ) \tfrac{\alpha^{N-1} - 1}{\alpha - 1}
		)
		,
	\]
	where
	$\gamma_k$ is the discrete inf-sup constant from~\eqref{e:ga-EF}.
	If $\alpha = 1$ then
	$C_k = \gamma_k^{-2} \beta ( \theta^2 + N (1 - \theta^2 ) )$.
\end{theorem}
\begin{proof}
	Corollary~\ref{c:Ukpinorm} yields
	\[
		\label{e:stabproof1}
		\norm{U^k}{\pi} & =
		\beta \sum_{n=1}^N G_{nn}
		+
		\beta(\alpha- \theta^2 ) \sum_{m=1}^{N-1}
		G_{mm}
		\sum_{n = 0}^{N-m-1} \alpha^n
		,
	\]
	where we have changed the order of summation.

	It follows from the
	observations in~\eqref{e:alpha-theta-rel} that
	$\alpha \geq \theta^2 \geq 0$.
	Thus, if $\alpha \neq 1$ we have
	$
	\tfrac{1-\alpha^{N-n}}{1 - \alpha}
	\leq
	\tfrac{1-\alpha^{N-1}}{1 - \alpha}
	$
	and
	evaluating the geometric sum in \eqref{e:stabproof1}
	yields
	\begin{equation*}
		\norm{U^k}{\pi}
		=
		\beta \sum_{n=1}^{N} (1 + (\alpha - \theta^2 ) \tfrac{1-\alpha^{N-n}}{1 - \alpha}) G_{nn}
		\leq
		\beta( 1 + (\alpha - \theta^2 ) \tfrac{1-\alpha^{N-1}}{1 - \alpha} )
		\norm{G^k}{\pi}
		\leq
		C_k \norm{\ell}{-\epsilon}.
	\end{equation*}
	For $\alpha = 1$, the claim follows directly from \eqref{e:stabproof1}.
\end{proof}

As a consequence of the
the stability bound in the previous theorem
we obtain
an inf-sup condition
for $\cB^k = B - \rho^2 \DIAG^k$.
It is convenient
to formulate it
on
the subspaces
$\widehat{E}^k_\pi \subset E^k_\pi$ and
$\widehat{F}^k_\epsilon \subset F^k_\epsilon$
of symmetric functions.

\begin{corollary}
	\label{cor:cBinfsup}
	Suppose the temporal mesh is uniform with $\beta > 0$.
	Then $\cB^k$ in~\eqref{e:vfk}
	satisfies the discrete inf-sup condition
	(note the symmetrization)
	\[
		\label{e:cBinfsup}
		\inf_{w\in S(\widehat{E}^k_\pi)}
		\sup_{v\in S(\widehat{F}^k_\epsilon)}
		\cB^k(w,v)
		\geq
		C_k^{-1}
		,
	\]
	where $C_k$ is the discrete stability constant in~\eqref{e:stab>1}.
\end{corollary}

\begin{proof}
	Fix a symmetric $w \in \widehat{E}^k_\pi$.
	On $\widehat{F}_\epsilon^k$
	define the functional
	$\ell := \cB^k w$,
	extending
	it
	via Hahn--Banach
	with equal norm to $F_\epsilon$.
	Decompose it as
	$\ell =: \ell^+ - \ell^- + \ell^a$
	as in \eqref{e:lpm}.
	Then $\ell^a = 0$ by symmetry of $w$.
	Let $w^\pm \in \widehat{E}^k_\pi$ be the solution
	to~\eqref{e:vfk} with the right-hand side $\ell^\pm$.
	Clearly, $w = w^+ - w^-$.
	Therefore,
	\begin{equation*}
		\norm{w}{\pi}
		\leq
		\norm{w^+}{\pi} + \norm{w^-}{\pi}
		\stackrel{ \eqref{e:stab>1} }{ \leq }
		C_k ( \norm{\ell^+}{-\epsilon} + \norm{\ell^-}{-\epsilon} )
		\stackrel{ \eqref{e:lpm} }{ = }
		C_k
		\norm{\ell}{-\epsilon}
		.
	\end{equation*}
	Since $w \in \widehat{E}^k_\pi$ was arbitrary
	and
	$\norm{\ell}{-\epsilon} = \sup_{v \in S(\widehat{F}_\epsilon^k)} \cB^k(w, v)$,
	the conclusion \eqref{e:cBinfsup} follows.
\end{proof}

Now we introduce
some approximations $\DIAG^k$ of
the trace product $\DIAG$.
This is of interest
primarily
for
the $\text{iE}^{\star}$ discretization.
%
%
%
%
The schemes we consider are
\begin{enumerate}
	\item[{$\text{CN}^{\star}_2$}:]
	
		the $\text{CN}^\star$ discretization discussed in \S\ref{s:dis:1:CN*}
		with the exact trace product $\DIAG^k := \DIAG$.
		
	\item[{$\text{iE}^{\star}_2$}:]
	
		The $\text{iE}^\star$ discretization introduced in \S\ref{s:dis:1:iE*}
		with the exact trace product
		$\DIAG^k := \DIAG$.

	\item[{$\text{iE}^{\star}_2/Q$}:]

		$\text{iE}^\star$
		with preprocessing:
		$\DIAG^k := \DIAG \circ (q_k \otimes q_k)$
		with $q_k$ from~\eqref{e:Q}.

	\item[{$\text{iE}^{\star}_2/\boxrule$}:]

		$\text{iE}^\star$
		with the ``box rule''
		\[
			\label{e:DIAGk:box}
			\DIAG^k(w,v) := \sum_{n=1}^{N} k_n^{-1} \int_{J_n \times J_n} w(s, t) v(s, t) \rd s \rd t,
			\quad
			(w, v) \in E^k_\pi \times F^k_\epsilon
			.
		\]
		This definition is motivated
		by observing
		that $\DIAG(w, v)$ is the double integral of
		$\Dirac(s - t) w(s, t) v(s, t)$
		over all ``boxes'' $J_n \times J_n$
		and
		approximating
		$\Dirac(s - t)$ by $k_n^{-1}$
		on $J_n \times J_n$.
\end{enumerate}

All these candidates for
the approximate trace product $\DIAG^k$
satisfy the assumptions
\ref{enum:DeltaSymm}--\ref{enum:DeltaBdd}
made above.
In particular, they are bilinear and continuous,
as quantified in the following lemma.

\begin{lemma}
	\label{l:|DIAGk|}
	Each of the above $\DIAG^k$ is
	bounded on $E_\pi \times F_\epsilon$
	with
	\begin{equation*}
		%
		\DIAG^k(w,v) \leq \tfrac{1}{2\lambda} \norm{w}{\pi} \norm{v}{\epsilon}
		\quad
		\forall
		(w, v) \in E_\pi \times F_\epsilon
		.
	\end{equation*}
\end{lemma}

\begin{proof}
	Boundedness of the exact trace product
	is
	the subject of
	Lemma~\ref{l:Delta}.
	For the approximation with preprocessing
	$\DIAG^k := \DIAG \circ (q_k \otimes q_k)$
	we
	have the same bound,
	because
	$\norm{ q_k \from E \to E^k }{} = 1$
	and
	therefore
	$\norm{ (q_k \otimes q_k) \from E_\pi \to E^k_\pi }{} = 1$.

	Now consider
	the ``box rule''
	$\DIAG^k$ as in
	$\eqref{e:DIAGk:box}$.
	Let $(w, v) \in E_\pi \times F_\epsilon$.
	By \cite[Thm.~2.4]{Schatten1950}
	we may assume that $w = w^1 \otimes w^2$.
	Employing
	$|v(s, t)| \leq \tfrac{1}{2} \norm{v}{\epsilon}$
	from
	\eqref{e:vtt}
	in
	\eqref{e:DIAGk:box}
	results in the estimate
	$ 
	\DIAG^k(w,v)
	\leq
	\tfrac{1}{2}
	\norm{v}{\epsilon}
	\sum_n
	\norm{w_1}{L_2(J_n)}
	\norm{w_2}{L_2(J_n)}
	\leq
	\tfrac{1}{2 \lambda}
	\norm{v}{\epsilon}
	\norm{w_1}{E}
	\norm{w_2}{E}
	.
	$ 
\end{proof}

The values of~$\DIAG^n$, $\alpha$, $\beta$ and $\theta$
for each scheme
are given in Table~\ref{t:stab1} below
in terms of the time-step size
$k > 0$
(assumed uniform)
and
the dimensionless numbers
$z := \lambda k$ and $q := \rho^2 / (2\lambda)$.
Recall that
the basis $\{ e_n \}_n \subset E^k$
is normalized to $\norm{e_n}{E} = 1$
to define these values.
The denominator of
$\beta_n = \lambda k_n b_{n-1,n}^2 / D_n$
is $D_n = \lambda k_n (b_{n-1,n}^2 - \rho^2 \DIAG^n_{11})$.
Thus, $D_n > 0$
necessary and sufficient
for
$\beta_n > 0$
in Lemma \ref{l:Ukspsd}.
On a uniform mesh we write $D := D_n$.
We remark that $D > 0$
holds for all our schemes
if
the temporal mesh width $k$ is sufficiently small,
namely
when $k \rho^2 \lesssim 1$.

\begin{table}[htb]
	\centering
	\begin{tabular}{c||c|c|c|c|c|}
		Scheme & $\lambda \DIAG^n$ & $D$ & $\alpha - 1$ & $\beta$ & $\theta$
		\\
		\hline
		$\text{CN}^{\star}_2$
		& $\tfrac{1}{6} \left(\begin{smallmatrix}
		2 & 1 \\
		1 & 2
		\end{smallmatrix}\right)$
		& $\left(1 + z/2\right)^2 - \tfrac{2}{3} q z$
		& $\tfrac{(2/3)(2+\sqrt{\theta})q z - 2z}{D}$
		& $\tfrac{\left(1+ z/2 \right)^2}{D}$
		& $\tfrac{z/2-1}{z/2+1}$
		\\
		\hline
		$\text{iE}^{\star}_2$
		& $\tfrac{1}{60} \left(\begin{smallmatrix}
		38 & 7 \\
		7 & 8
		\end{smallmatrix}\right)$
		& $\tfrac{1}{4} (1+z)^2 - \tfrac{19}{15} qz$
		& $\tfrac{(4/15)(23 + 7\sqrt{\theta}) qz - z(2+z)}{4D}$
		& \multirow{3}{*}{$\tfrac{(1+z)^2}{4D}$}
		& \multirow{3}{*}{$\tfrac{-1}{1+z}$}
		\\
		$\text{iE}^{\star}_2/Q$
		& $\tfrac{1}{24} \left(\begin{smallmatrix}
		2 & 1 \\
		1 & 2
		\end{smallmatrix}\right)$
		& $\tfrac{1}{4}(1 + z)^2 - \tfrac{1}{6} q z$
		& $\tfrac{(2/3)(2+\sqrt{\theta})q z - z(2+z)}{4D}$
		&
		&
		\\
		$\text{iE}^{\star}_2/\boxrule$
		& $\tfrac{1}{4} \left(\begin{smallmatrix}
		1   & 0   \\
		0   & 0
		\end{smallmatrix}\right)$
		& $\tfrac{1}{4}(1 + z)^2 - \tfrac{1}{2} q z$
		& $\tfrac{1-4D}{4D}$
		&
		&
	\end{tabular}
	\caption{%
		Discretization parameters for
		the schemes from \S\ref{s:dis:2-mul}
		expressed in terms of
		$z := \lambda k$ and $q := \rho^2 / (2\lambda)$.
	}
	\label{t:stab1}
\end{table}

With Theorem~\ref{t:stab:dis}
we find that
$\lim_{k \to 0} C_k = C$
for the schemes
$\text{CN}^{\star}_2$,
$\text{iE}^{\star}_2/Q$,
and
$\text{iE}^{\star}_2/\boxrule$
(but \emph{not} for $\text{iE}^{\star}_2$),
where $C$ is the stability constant in~\eqref{e:stab-mul}
of the continuous problem~\eqref{e:U-mul-L}.

\subsubsection{Error analysis and convergence}
\label{s:convergence}

In this subsection we estimate the difference between
the exact solution $U$ to~\eqref{e:U-mul-L}
and
the discrete solution $U^k$ to~\eqref{e:vfk}.
We first remark that by Lemma~\ref{l:|DIAGk|},
the
norm of
$\cB^k = B - \rho^2 \DIAG^k$
is
bounded by
\begin{equation*}
	%
	\norm{ \cB^k }{}
	\leq
	1 + \tfrac{\rho^2}{2\lambda}
\end{equation*}
for each
{$
	\DIAG^k
	\in
	\{
	\DIAG,
	\,
	\DIAG \circ (q_k \otimes q_k),
	\,
	\eqref{e:DIAGk:box}
	\}
$}.
Moreover, $\cB^k$ satisfies
the inf-sup condition~\eqref{e:cBinfsup}
on
{$
	\widehat{E}^k_\pi \times \widehat{F}^k_\epsilon
	\subset
	E_\pi \times F_\epsilon
$},
and
the dimensions of these subspaces coincide.
Hence, Proposition~\ref{p:well:crime}
on quasi-optimality of the discrete solution
applies.
This quasi-optimality is formulated
in terms of the symmetric subspace $\widehat{E}^k_\pi$,
but
we can improve this to $E^k_\pi$ for symmetric solutions $U$.
Indeed,
if $U \in E_\pi$ is symmetric
then
{$
	\norm{ U - \frac12(w + w^*) }{\pi}
	\leq
	\frac{1}{2} ( \norm{U - w}{\pi} + \norm{(U - w)^*}{\pi} )
	=
	\norm{U - w}{\pi}
$}
for
any $w \in E_\pi$,
where $(\CDOT)^*(s, t) := (\CDOT)(t, s)$.
Furthermore,
the appearing residual
$
(\cB - \cB^k) U = (\DIAG - \DIAG^k) U
$
is a symmetric functional,
whether $U$ is symmetric or not,
and
therefore vanishes on
anti-symmetric elements of $F_\epsilon$.
This leads to the estimate
\begin{equation*}
	%
	\norm{U - U^k}{\pi}
	\leq
	(
	1 + C_k \norm{ \cB^k }{}
	)
	\inf_{ w \in E^k_\pi } \norm{U - w}{\pi}
	+
	C_k
	\norm{
	(\DIAG - \DIAG^k) U
	}{ (F^k_\epsilon)' }
	%
\end{equation*}
for symmetric $\ell$.
Replacing $C_k$ by $(\gamma_k^{-2} + C_k)$,
the assumption of symmetry
may be dropped.

This result
shows
convergence for the {$\text{CN}^{\star}_2$} scheme,
where $\DIAG^k = \DIAG$.
Unfortunately,
it
is not useful
for
the {$\text{iE}^{\star}_2$} scheme
and
its variants,
because
the best approximation from
the discrete space $E_\pi^k$
does not converge
to $U$ as we refine the temporal mesh,
see the discussion at the end of \S\ref{s:dis:1:iE*}.
This motivates
looking at
the postprocessed solution
\[
	\label{e:Ukpost}
	\bar{U}^k := Q_k U^k
	\quad\text{with}\quad
	Q_k := (q_k \otimes q_k)
\]
for these schemes, where $q_k$ is the projection from \eqref{e:Q}.
Recall that $q_k$ is injective on $E^k$.
By $Q_k^{-1}$ we will mean the inverse of $Q_k \from E_\pi^k \to Q_k E_\pi^k$.
In the case of
the {$\text{iE}^{\star}_2$} discretization,
\eqref{e:normqk}
implies
\[
	\label{e:|Q|}
	\norm{ Q_k w }{\pi}
	=
	\tfrac14
	\norm{ w }{\pi}
	\quad
	\forall w \in E_\pi^k
	.
\]

The convergence of the postprocessed solution
will again be
obtained
via
Proposition~\ref{p:well:crime}.
To this end,
we define
$\bar{\cB}^k := \cB^k \circ Q_k^{-1} Q_k \from E_\pi \to F_\epsilon'$
with
the motivation that
the postprocessed solution
solves the modified discrete problem
\[
	\label{e:vfk-post}
	\text{Find}
	\quad
	\bar{U}^k \in Q_k E_\pi^k
	\quad\text{s.t.}\quad
	\bar{\cB}^k( \bar{U}^k, v )
	=
	\ell(v)
	\quad
	\forall v \in F_\epsilon^k
	.
\]
The operator $\bar{\cB}^k$ is bounded with
$\norm{ \bar{\cB}^k }{} \leq 4 \norm{ \cB^k }{}$.
Moreover,
it follows from \eqref{e:|Q|}
that
if
$\cB^k$
satisfies the discrete inf-sup condition \eqref{e:cBinfsup}
on $\widehat{E}_\pi^k \times \widehat{F}_\epsilon^k$
with the constant $C_k^{-1}$
then
so does
$\bar{\cB}^k$
on $Q_k \widehat{E}_\pi^k \times \widehat{F}_\epsilon^k$
with
the constant $4 C_k^{-1}$.
The following is our main result of this section.

\begin{proposition}
	\label{p:post}
	Let $\ell \in F_\epsilon'$ be symmetric.
	Assume the discrete inf-sup condition~\eqref{e:cBinfsup}.
	Then
	the exact solution $U \in E_\pi$ to~\eqref{e:U-mul-L}
	and
	the postprocessed discrete solution
	$\bar{U}^k \in Q_k E^k_\pi$ to~\eqref{e:vfk-post}
	differ by
	\begin{equation*}
		%
		\norm{U - \bar{U}^k}{\pi}
		\leq
		(
		1 + C_k \norm{ \cB^k }{}
		)
		\inf_{ w \in Q_k E^k_\pi } \norm{U - w}{\pi}
		,
	\end{equation*}
	for the {$\text{CN}^{\star}_2$} scheme,
	and
	by
	\[
		\label{e:iE-post}
		\norm{U - \bar{U}^k}{\pi}
		\leq
		(
		1 + C_k \norm{ \cB^k }{}
		)
		\inf_{ w \in Q_k E^k_\pi } \norm{U - w}{\pi}
		+
		\tfrac14 C_k
		\norm{
		(\cB - \bar{\cB}^k) U
		}{ (F^k_\epsilon)' }
	\]
	for any of
	the {$\text{iE}^{\star}_2$} schemes.
\end{proposition}

To complete the analysis
we need to estimate the residual term in
\eqref{e:iE-post}.
Hence,
from now on
we focus entirely on
the {$\text{iE}^{\star}_2$} schemes.
Recalling that $\cB = B - \rho^2 \DIAG$
and
$\bar{\cB}^k = (B - \rho^2 \DIAG^k) Q_k^{-1} Q_k$
we split the residual according to
\[
	\label{e:res}
	\cB - \bar{\cB}^k
	=
	\cB (\mathrm{Id} - Q_k)
	-
	B (\mathrm{Id} - Q_k) Q_k^{-1} Q_k
	-
	\rho^2
	(\DIAG Q_k - \DIAG^k) Q_k^{-1} Q_k
\]
and address
it term by term.
\begin{itemize}
	\item
		The first term
		$T_1 := \norm{ \cB (\mathrm{Id} - Q_k) U }{ (F^k_\epsilon)' }$
		in \eqref{e:iE-post}/\eqref{e:res}
		goes to zero upon mesh refinement
		by density of the subspaces $Q_k E_\pi^k \subset E_\pi$.

	\item
		To bound
		the second term
		$T_2 := \norm{ B (\mathrm{Id} - Q_k) Q_k^{-1} Q_k U }{ (F^k_\epsilon)' }$
		in \eqref{e:iE-post}/\eqref{e:res}
		we proceed in two steps.
		First, we observe
		that
		$
		b( (\mathrm{Id} - q_k) w, v )
		=
		( (\mathrm{Id} - q_k) w, v )_E
		=
		( (\mathrm{Id} - q_k) w, (\mathrm{Id} - q_k) v )_E
		\leq
		\norm{w}{E} \norm{ (\mathrm{Id} - q_k) v }{E}
		$
		for any $(w, v) \in E \times F$.
		The Poincar\'e--Wirtinger inequality
		on each temporal element
		yields
		$
		\norm{ (\mathrm{Id} - q_k) v }{E}
		\leq
		\frac{1}{\sqrt{12}} \lambda \max_n k_n
		\norm{ v }{F}
		$
		for all $v \in F^k$.
		Second, we write
		\[
			\label{e:Id-Qk}
			\mathrm{Id} - Q_k
			=
			\tfrac12
			[
			(\mathrm{Id} - q_k) \otimes (\mathrm{Id} + q_k)
			+
			(\mathrm{Id} + q_k) \otimes (\mathrm{Id} - q_k)
			]
			,
		\]
		and use this identity in $B (\mathrm{Id} - Q_k)$.
		Recalling
		$\norm{ Q_k^{-1} Q_k U }{\pi} = 4 \norm{Q_k U}{\pi} \leq 4 \norm{U}{\pi}$
		from \eqref{e:|Q|},
		this gives
		$
		T_2
		\leq
		\frac{4}{\sqrt{3}} \lambda \max_n k_n
		\norm{ U }{\pi}
		.
		$

	\item
		Consider now the third term
		$
		T_3
		:=
		\rho^2
		\norm{ (\DIAG Q_k - \DIAG^k) Q_k^{-1} Q_k U }{ (\widehat{F}^k_\epsilon)' }
		$
		in \eqref{e:iE-post}/\eqref{e:res}.
		For the {$\text{iE}^{\star}_2$} scheme
		where $\DIAG^k = \DIAG$,
		this term does not converge to zero
		upon mesh refinement,
		see \S\ref{s:iE2}.
		For the {$\text{iE}^{\star}_2/Q$} scheme
		where $\DIAG^k = \DIAG Q_k$,
		this term vanishes identically.

		It remains to discuss the ``box rule''
		where
		$\DIAG^k = \eqref{e:DIAGk:box}$.
		To this end, we first note that for
		$v\in F^k_\epsilon$
		\[
			\label{e:box-est0}
			(\DIAG Q_k - \DIAG^k)(Q_k^{-1} Q_k U, v)
			= \DIAG( Q_k U, (\mathrm{Id}  - I_k) v),
		\]
		where $I_k := i_k \otimes i_k$
		with the interpolation operator $i_k$
		onto the space of piecewise constants
		from~\eqref{e:iii-F1}.
		To estimate
		the last expression,
		we expand $\mathrm{Id} - I_k$ as in~\eqref{e:Id-Qk}.
		Recalling from \cite[\S3.2]{Ryan2002}
		that
		$
		C^0(\bar{J} \times \bar{J})
		=
		C^0(\bar{J}) \otimes_\epsilon C^0(\bar{J})
		$,
		the estimates
		$
		\norm{\psi - i_k \psi}{C^0(\bar{J})}
		\leq
		\sqrt{ \lambda \max_n k_n }
		\norm{\psi}{F}
		$
		and
		$
		\norm{\psi + i_k \psi}{C^0(\bar{J})}
		\leq
		\sqrt{2}
		\norm{\psi}{F}
		$
		for $\psi \in F^k$,
		then imply
		\begin{equation*}
			%
			| \eqref{e:box-est0} |
			\leq
			\diag(|Q_k U|) \norm{(\mathrm{Id} - I_k) v}{C^0(\bar{J} \times \bar{J})}
			\leq
			\sqrt{\textstyle 2 \max_n k_n / \lambda}
			\norm{U}{\pi} \norm{v}{\epsilon}.
		\end{equation*}
\end{itemize}

\subsubsection{Non-convergence of {$\text{iE}^{\star}_2$} with postprocessing}
\label{s:iE2}

We introduced
the approximate trace product \eqref{e:DIAGk}
because
even with postprocessing,
the {$\text{iE}^{\star}_2$} scheme
with the exact trace product
does not converge
upon temporal mesh refinement.
In fact,
it is consistent with the value $2 \rho$
for the volatility instead of $\rho$,
as we will indicate here.
First,
as in \eqref{e:normqk},
we have
$\DIAG(w, Q_k v) = \DIAG(4 Q_k w, Q_k v)$
for all $(w, v) \in E^k_\pi \times F^k_\epsilon$.
Therefore,
invoking $\diag(|w - 4 Q_k w|) \leq \tfrac{1}{\lambda} \norm{w - 4 Q_k w}{\pi}$
and the identity~\eqref{e:|Q|},
\[
	\label{e:iE2-4}
	|\DIAG(w, v) - 4 \DIAG(Q_k w, v)|
	=
	|\DIAG(w - 4 Q_k w, v - Q_k v)|
	\leq
	\tfrac{2}{\lambda} \norm{w}{\pi} \norm{v - Q_k v}{C^0(\bar{J} \times \bar{J})}
	.
\]
To bound
the last term,
we use
the estimates
$
\norm{\psi - q_k \psi}{C^0(\bar{J})}
\leq
\tfrac12 \sqrt{ \lambda \max_n k_n }
\norm{\psi}{F}
$
and
$
\norm{\psi + q_k \psi}{C^0(\bar{J})}
\leq
\sqrt{2}
\norm{\psi}{F}
$
for $\psi \in F^k$,
which yield
$
\norm{ v - Q_k v }{ C^0(\bar{J} \times \bar{J}) }
\lesssim
\sqrt{ \lambda \max_n k_n }
\norm{ v }{\epsilon}
$.
By the preceding subsection,
the {$\text{iE}^{\star}_2/Q$} scheme
with $\DIAG Q_k$
does provide
a consistent approximation,
so \eqref{e:iE2-4} shows that
{$\text{iE}^{\star}_2$}
does not.

\subsection{Numerical example}
\label{s:numex}

In the following numerical experiment
we implement
the schemes
{$\text{CN}^{\star}_2$},
{$\text{iE}^{\star}_2$},
{$\text{iE}^{\star}_2/Q$},
and
{$\text{iE}^{\star}_2/\boxrule$}
proposed in \S\ref{s:dis:2-mul}
to solve
the discrete variational problem \eqref{e:vfk}.
In addition,
we apply
the discretizations
of polynomial degree $p = 2$
from \S\ref{s:dis:1:both}
with
the exact trace product $\DIAG$,
denoted by
{$\text{CN}^{\star}_2(2)$}
and
{$\text{iE}^{\star}_2(2)$}.
We choose
$T = 2$, $\lambda = 3$, $\rho^2 = \lambda/2$,
and for
the right-hand side
$\ell(v) := v(0)$,
motivated by \eqref{e:U-mul}.
The error
against the exact solution from \eqref{e:X-table-EXsXt}
is measured
as
the $L_1$ error on the diagonal,
$\mathrm{E}(U^{\mathrm{num}}) := \diag(|U - U^{\mathrm{num}}|)$
for
$U^{\mathrm{num}} = U^k$ (without postprocessing)
and
$U^{\mathrm{num}} = \bar{U}^k$ (with postprocessing).
Note that
only the inequality
$\diag(|w|) \leq \frac1\lambda \norm{w}{\pi}$
holds
(with equality when $w$ is SPSD).
Nevertheless, we use this measure
for simplicity
and for easier comparison with Monte Carlo below.
The results are shown in Figure~\ref{f:nm1}.
Convergence of the schemes
is summarized in the following table
(along with the number of conjugate gradients iterations as discussed below).
The convergence, where present, is
of first order in
the temporal mesh width.
%
%
%
\[
	\label{t:nm1}
	%
	\text{
	\newcommand{\yes}{$\checkmark$}
	\newcommand{\no}{$\times$}
	\begin{tabular}{l|cccccc}
	& {$\text{CN}^{\star}_2$} & {$\text{CN}^{\star}_2(2)$}
	& {$\text{iE}^{\star}_2$} & {$\text{iE}^{\star}_2(2)$}
	& {$\text{iE}^{\star}_2/Q$} & {$\text{iE}^{\star}_2/\boxrule$}
	\\
	\hline
	$\bar{U}^k$       & \yes & \yes & \no & \yes & \yes & \yes
	\\
	$U^k$             & \yes & \yes & \no & \yes & \no  & \no
	\\
	\hline
	$n_{\mathrm{CG}}$ & 50   & 71   & 294 & 113  & 50   & 49
	\end{tabular}
	}
\]
%
%
%
These results are in line
with
the convergence results
established in
\S\ref{s:dis:2-mul}.
The schemes
of polynomial degree $p = 2$
exhibit only first order convergence,
presumably
due to
the limited smoothness of
the second moment across the diagonal.
However, they do not require
pre- or postprocessing
for convergence.
The stability of
the {$\text{iE}^{\star}_2(2)$} scheme,
in particular,
does not depend on the temporal mesh width
as long as it is equidistant, see \eqref{e:ga-p},
but our analysis does not cover
this statement beyond the trivial range \eqref{e:range}.

\begin{figure}[t]
	\includegraphics[width=0.49\textwidth]{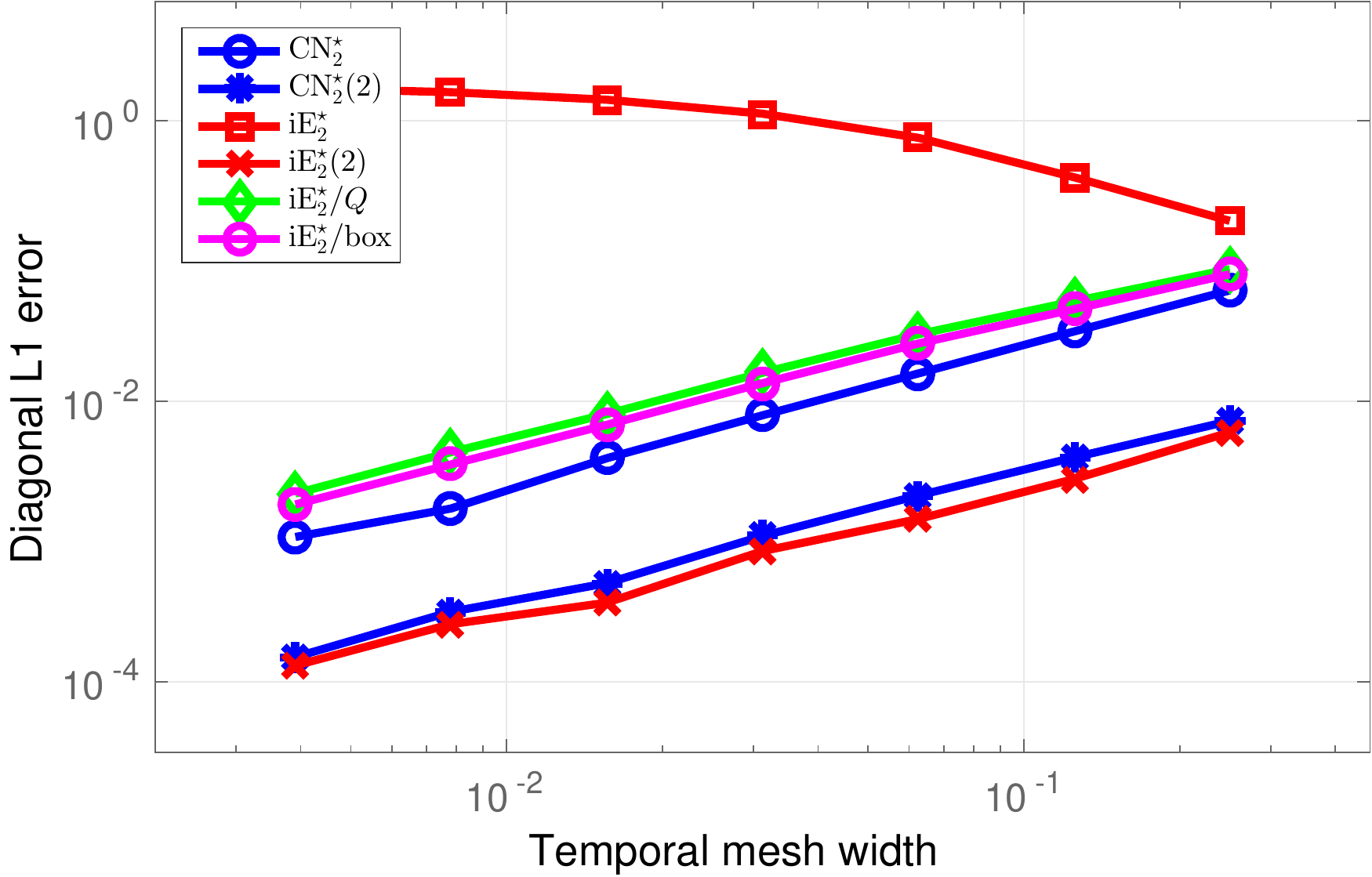}
	\hfill
	\includegraphics[width=0.49\textwidth]{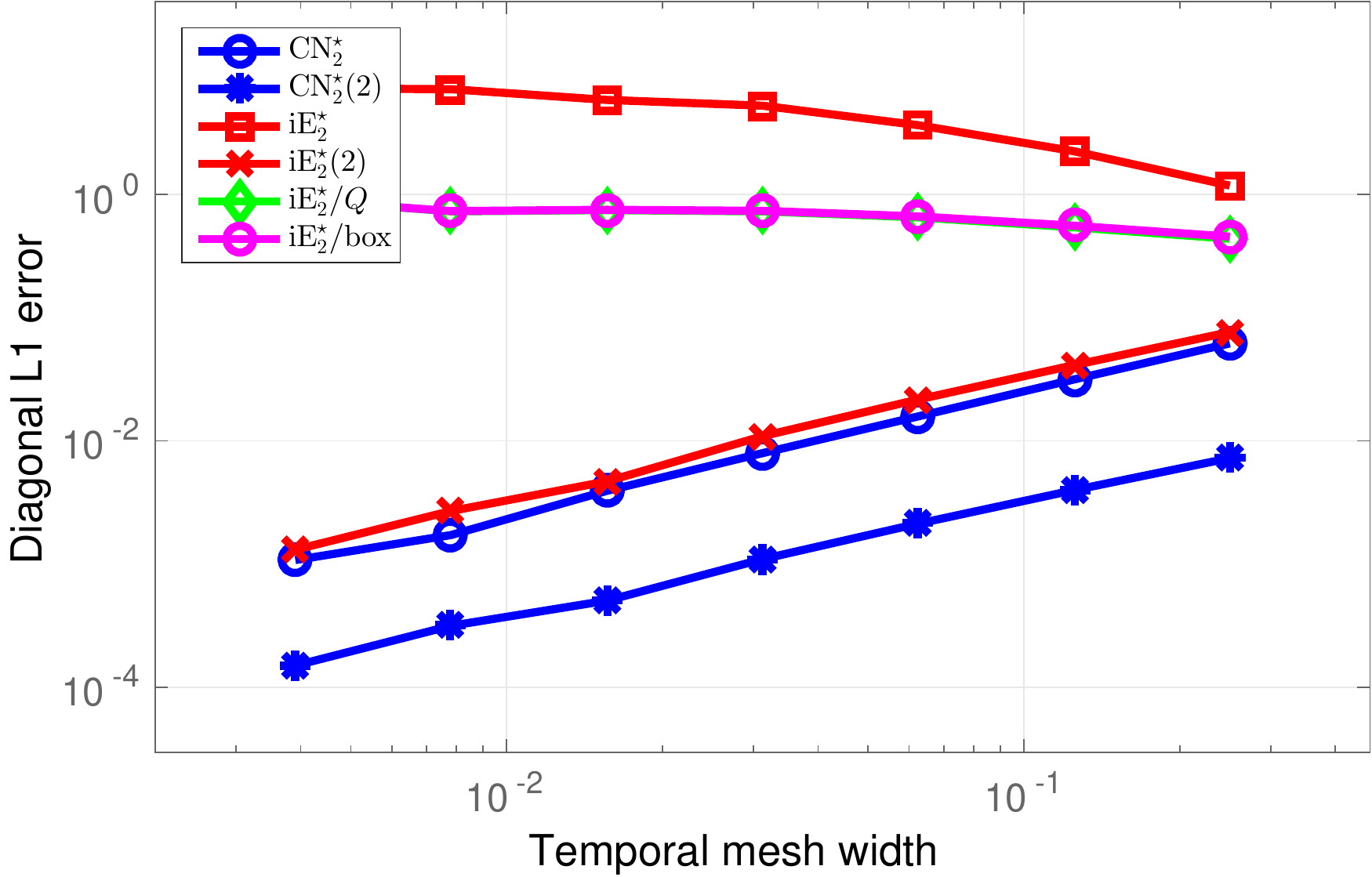}

	\caption{%
		The error
		$\mathrm{E}(U^{\mathrm{num}}) = \diag(|U - U^{\mathrm{num}}|)$
		as a function of the temporal mesh width
		for the example from
		\S\ref{s:numex}.
		\textbf{Left}:
		with postprocessing,
		$U^{\mathrm{num}} = \bar{U}^k$.
		\textbf{Right}:
		without postprocessing,
		$U^{\mathrm{num}} = U^k$.
		\label{f:nm1}
	}
\end{figure}

\newcommand{\VEC}{\operatorname{Vec}}
The discrete variational problem \eqref{e:vfk},
with the choice of bases described at the beginning of \S\ref{s:dis:2-mul},
leads to
the linear algebraic problem
$\mathbf{B} \VEC(\mathbf{U}) = \mathbf{F}$.
Here, $\VEC$ stacks the columns of
the matrix $\mathbf{U}$ into one long vector.
Let $\mathbf{M} := \mathbf{m} \otimes \mathbf{m}$
and $\mathbf{N} := \mathbf{n} \otimes \mathbf{n}$,
where $\mathbf{m} / \mathbf{n}$ are the mass matrices for $E / F$.
We symmetrize the problem
as
$\mathbf{B}^\T \mathbf{N}^{-1} \mathbf{B} \VEC(\mathbf{U}) = \mathbf{B}^\T \mathbf{N}^{-1} \mathbf{F}$
and
solve this with the conjugate gradients method
preconditioned with $\mathbf{M}$.
The matrix-vector products are implemented in a matrix-free fashion
with linear complexity in the size of $\mathbf{U}$,
which is of order $k^{-2}$.
The symmetrization is motivated by
operator preconditioning
that was shown
to be effective
for space-time discretizations of parabolic evolution equations \cite{AndreevTobler2015},
but
the correct adaptation to
the present setting
of Banach spaces that are not strictly convex
is an open issue.
We use the MATLAB \texttt{pcg} solver with tolerance $10^{-10}$,
resulting in a number of iterations $n_{\mathrm{CG}}$
that increases with increasing temporal resolution.
Thus the computational effort is of order $n_{\mathrm{CG}} k^{-2}$.
The number of iterations $n_{\mathrm{CG}}$ for $k = 2^{-9} T$
is shown in Table \eqref{t:nm1}.
%


Another possibility to solve the discrete problem
is indicated by Lemma \ref{l:Uknn},
where first only the diagonal of
the discrete second moment is determined
from the data.
More fundamentally,
one could directly target numerically
the ordinary differential / integral equation satisfied
by the diagonal
of the continuous second moment,
see the proof of Theorem \ref{t:stab-mul}.

We comment briefly on the error
of the Monte Carlo empirical estimate
of the second moment.
Let $X_1, \ldots, X_R$ be
i.i.d.~copies of the solution process $X$.
The empirical estimate of the second moment $M$
in $s, t \in \bar{J}$ with $R$ samples
is the random variable
$M_R(s, t) := \frac1R \sum_{r = 1}^R X_r(s) X_r(t)$.
Then we have
$ 
\e{ |M(s, t) - M_R(s, t)|^2 }
=
\Var(M_R(s, t))
=
\tfrac1R \Var(X(s) X(t))
.
$ 
Setting $s = t$ and integrating over $J$
leads to
the strong error estimate
$
\e{ \diag(|M - M_R|)^2 }
\leq
\frac{T}{R}
\int_J \Var{( X(t)^2 )} \rd t
.
$
We expect a similar estimate to hold for the $\norm{\CDOT}{\pi}$ norm.
Balancing
the Monte Carlo error $1/\sqrt{R}$
with
the temporal discretization error $k$
requires
$R \sim k^{-2}$
samples;
since adding one summand to $M_R$ is on the order of $k^{-2}$ operations,
this
leads to an overall effort of $\mathcal{O}(k^{-4})$.
The effort could be reduced with parallelization
and other techniques mentioned
in the introduction.
%
%

\section{Stochastic PDEs with affine multiplicative noise}
\label{s:pde}

In this section we generalize the preceding discussion
of scalar stochastic ODEs
to vector-valued stochastic PDEs
\[
	\label{e:spde-mul}
	\rd X(t) + A X(t) \rd t
	=
	G[ X(t) ] \rd L(t)
	,
	\quad
	t \in \bar{J},
	\quad\text{with}\quad
	X(0) = X_0
	.
\]
Here, $A \from \mathcal{D}(A) \subset H \to H$ is
a self-adjoint, positive definite operator,
densely defined on a real Hilbert space $H$,
with a compact inverse $A^{-1} \from H \to H$.
Furthermore, $L := (L(t),t\geq 0)$ is a
square-integrable zero-mean L\'{e}vy process
taking values in a Hilbert space $\cU$
with a self-adjoint
positive semi-definite
trace-class covariance operator $Q$, i.e.,
$
\bbE[
	\scalar{L(s), x}{\cU}
	\scalar{L(t), y}{\cU}
]
=
(s\wedge t) \scalar{Q x, y}{\cU}
$
for all $s, t \geq 0$ and $x, y \in \cU$.
For each $\varphi \in H$,
$G[\varphi] \from \cU \to H$
is a bounded linear operator
and
$G$ is affine:
$G[\varphi] = G_1[\varphi] + G_2$,
for certain
$G_1 \in \cL(H; \cL(\cU; H))$
and
$G_2 \in \cL(\cU; H)$.
Further technical assumptions
on $G$, $L$ and $X_0$
are those of \cite[\S2]{KirchnerLangLarsson2016}.

We define the space $V \subset H$ with the norm
$\norm{\CDOT}{V} := \norm{A^{1/2} \nativecdot}{H}$.
Identifying $H$ with its dual $H'$,
we obtain the Gelfand triple
\begin{equation*}
	\label{e:gelfand1}
	V \hookrightarrow H \cong H' \hookrightarrow V'
\end{equation*}
with continuous and dense embeddings,
and the $H$ inner product
has a unique extension by continuity
to the duality pairing on $V'\times V$,
denoted by $\duality{\CDOT,\CDOT}{}$.
Moreover, akin to~\eqref{e:order}, we find
\[
	\label{e:gelfand2}
	V_\pi
    \hookrightarrow
    H_\pi
    \hookrightarrow
    H_2 \cong H'_2
    \hookrightarrow
    V'_2
    \hookrightarrow (V')_\epsilon 
	,
\]
and
the $H_2$ inner product
extends continuously to the duality pairing
$\duality{\CDOT,\CDOT}{\pi,\epsilon}$
on $V_\pi \times (V')_\epsilon$.
The functional framework for the deterministic PDE
of the second moment
is based on the Bochner spaces
\[
	\label{e:XY}
	\cX := L_2(J; V)
	\TEXT{and}
	\cY := \{ v \in H^1(J;V') \cap L_2(J;V) : v(T) = 0 \}
\]
which are
equipped with the norms
$\norm{w}{\cX}^2 := \int_J \norm{w(t)}{V}^2 \rd t$
and
$ 
\norm{v}{\cY}^2
:=
\norm{-\partial_t v + Av}{L_2(J;V')}^2
=
\int_J \norm{\partial_t v(t)}{V'}^2 \rd t
+
\int_J \norm{v(t)}{V}^2 \rd t
+
\norm{v(0)}{H}^2
$ 
and the obvious corresponding inner products.
The norm on $\cY$ is equivalent to
the one used in~\cite{KirchnerLangLarsson2016}.
Analogously to~\eqref{e:Fnorm-eq},
these norms render
the operator $b \from \cX \to \cY'$,
stemming from the bilinear form
\[
	\label{e:b-pde}
	b \from \cX \times \cY \to \IR
	,
	\quad
	b(w, v) := \int_J \duality{w(t), -\partial_t v(t) + Av(t)}{} \rd t
	,
\]
an isometric isomorphism.
Consequently,
on the tensor product spaces
$\cX_\pi := \cX \otimes_\pi \cX$
and
$\cY_\epsilon' := (\cY\otimes_\epsilon\cY)' \cong \cY' \otimes_\pi \cY'$,
the (properly extended) operator
\[
	\label{e:B-pde}
	B := (b \otimes b)
	\from \cX_\pi \to \cY_\epsilon'
	\quad
	\text{is an isometric isomorphism}
	.
\]

As in~\eqref{e:cB}, the multiplicative noise in~\eqref{e:spde-mul}
causes an additional term acting on the temporal diagonals
of elements in $\cX_\pi$ and $\cY_\epsilon$ in
the bilinear form $\cB\from\cX_\pi\times\cY_\epsilon\to\IR$
for the variational formulation of the second moment equation.
The continuity of this diagonal term is a consequence of
the following two properties
of the tensor spaces $\cX_\pi$ and $\cY_\epsilon$:
\textbf{a)} the boundedness
of point evaluation functionals
on $\cY$ and $\cY_{\epsilon}$
addressed in Lemma~\ref{l:YC0}, and
\textbf{b)} the role
of the diagonal for elements in $\cX_\pi$
emphasized in Lemma~\ref{l:pinorm-pde}.
Being simple extensions of
Lemmas~\ref{lem:FC0} and~\ref{l:2andpinorm},
respectively,
the proofs are omitted.

\begin{lemma}
    \label{l:YC0}
    Let $\tilde{v}\in\cY$, $v\in\cY_\epsilon$. Then
    \begin{equation*}
        \norm{\tilde{v}(t)}{H} \leq \tfrac{1}{\sqrt{2}} \norm{\tilde{v}}{\cY}
        \AND
        \norm{v(s,t)}{H_\epsilon} \leq \tfrac{1}{2} \norm{v}{\cY_\epsilon}
        \quad
        \forall s,t\in\bar{J}.
	\end{equation*}
    In particular,
    the temporal diagonal
    $\widehat{v} : t \mapsto v(t, t)$
    is in $C^0(J; (V')_\epsilon)$.
\end{lemma}

\begin{lemma}
	\label{l:pinorm-pde}
	If $w \in \cX_\pi$,
	then
	its temporal diagonal $\widehat{w} \from t \mapsto w(t, t)$
	belongs to
	$L_1(J; V_\pi)$
	and satisfies
	$\norm{ \widehat{w} }{ L_1(J; V_\pi) } \leq \norm{w}{\cX_\pi}$.
	If $w \in\cX_\pi$ is $\cX$-SPSD \eqref{e:HSPSD} then equality holds.
\end{lemma}

Together with \eqref{e:gelfand2}, the two lemmas imply that
the vector analogue of the trace product \eqref{e:def:DIAG},
\[
	\label{e:DIAG}
	\DIAG(w, v)
	:=
	\int_J \duality{ w(t,t), v(t,t) }{\pi,\epsilon} \rd t.
\]
is well-defined on $\cX_\pi \times \cY_\epsilon$.

The covariance of the L\'evy process
will enter the deterministic PDE
for the second moment through the linear operator
$\cG_1 \from V_\pi \to V_\pi$
defined
by
\[
	\label{e:defcurlyG1}
	\scalar{\cG_1 [\psi \otimes \tilde{\psi}], \varphi \otimes \tilde\varphi}{H_2}
	=
	\scalar{
	Q^{1/2} G_1[       \psi ]' \varphi,
	Q^{1/2} G_1[\tilde{\psi}]' \tilde\varphi
	}{\cU}
	\quad
	\forall \varphi, \tilde\varphi \in H
	,
\]
where $G_1[\psi]' \from H \to \cU$ is the adjoint of $G_1[\psi]$.
In the scalar case, $\cG_1 = \rho^2$.
This operator is well-defined under suitable boundedness
assumptions on $G_1$.
For example,
if $\psi \mapsto G_1[\psi] Q^{1/2}$
is a bounded map from
$V$ into the Hilbert--Schmidt operator space $\cL_2(\cU; V)$,
then
\[
	\label{e:curlyG1bound}
	C_G
	:=
	\norm{\cG_1}{\cL(V_\pi)}
	\leq
	\norm{G_1[\nativecdot] Q^{1/2}}{\cL(V;\cL_2(\cU;V))}^2
	.
\]
Henceforth, we assume that $C_G$ is indeed finite.

We write $\DIAG \cG_1 \from \cX_\pi \to \cY_\epsilon'$ for the operator corresponding to $\DIAG(\cG_1 \CDOT, \CDOT)$.
This composition is also well-defined
because
the temporal diagonal
of $\cG_1 w$
belongs to $L_1(J; V_\pi)$
if $w \in \cX_\pi$,
cf.\ \cite[\S3]{KirchnerLangLarsson2016}.

%
Finally, we define
the operator for the second moment equation
in the vector-valued case,
\begin{equation*}
	\cB := B - \DIAG \cG_1
	.
\end{equation*}

Given a functional $\ell \in \cY_\epsilon'$,
we are now interested in
the variational problem
\[
	\label{e:vf-pde}
	\text{Find}
	\quad
	U \in \cX_\pi
	\quad\text{s.t.}\quad
	\cB(U, v) = \ell(v)
	\quad
	\forall v \in \cY_\epsilon.
\]

The second moment $M$ and the covariance $C$ of
the solution process $X$ to the SPDE~\eqref{e:spde-mul}
satisfy the deterministic variational problem~\eqref{e:vf-pde}
with suitable right-hand sides
$\ell_M$ and $\ell_C$,
see \cite[Thms.~4.2 \& 6.1]{KirchnerLangLarsson2016}.
These are given by
$
\ell_M(v)
:=
\duality{\bbE[X_0 \otimes X_0], v(0)}{\pi,\epsilon}
+
\DIAG((\cG - \cG_1)[\bbE X \otimes \bbE X], v)
$
and
$
\ell_C(v)
:=
\duality{\operatorname{Cov}(X_0), v(0)}{\pi,\epsilon}
+
\DIAG(\cG [\bbE X \otimes \bbE X], v)
$.
Here, $\cG[\nativecdot]$ is defined
as in~\eqref{e:defcurlyG1} with $G_1$ replaced by $G$.
Note that
$\ell_M, \ell_C \in \cY_\epsilon'$
are both SPSD.

%


\subsection{Well-posedness of the second moment PDE}
\label{s:pde:wellposed}

Well-posedness of~\eqref{e:vf-pde}
was deduced in~\cite[Thm.~5.5]{KirchnerLangLarsson2016}
under the smallness condition
$
\norm{G_1[\nativecdot] Q^{1/2}}{\cL(V;\cL_2(\cU;H))} < 1.
$
The following theorem
disposes of this assumption
%
by exploiting
semigroup theory on the Banach space $V_\pi$.
It is the vector analogue of Theorem~\ref{t:stab-mul}.
As in the scalar case,
the solution $U$ of \eqref{e:vf-pde} inherits
symmetry and definiteness
from an SPSD right-hand side $\ell \in \cY_\epsilon'$.
This crucial structural property
allows us to derive a stability
bound in the natural tensor norm.

\begin{theorem}
	\label{t:stab-mul-pde}
	Suppose $\cG_1 \in \cL(V_\pi)$
	with norm $C_G = \eqref{e:curlyG1bound}$.
	Then, for every functional $\ell \in \cY_\epsilon'$
	there exists a unique
	solution $U \in \cX_\pi$ to the variational problem~\eqref{e:vf-pde}.
	If $\ell$ is SPSD
	then
	$U$ is $\cX$-SPSD
	and satisfies the stability bound
	\[
		\label{e:UCL}
		\norm{U}{\cX_\pi} \leq C \, \norm{\ell}{\cY_\epsilon'}
		\quad
		\text{with}
		\quad
		C := \tfrac{C_G e^{(C_G - 2 \lambda_1)T} - 2 \lambda_1}{C_G - 2 \lambda_1}
		,
	\]
	where
	$\lambda_1 > 0$ is the smallest eigenvalue of $A$.
	In the limit $C_G = 2 \lambda_1$ we have $C = C_G T + 1$.
\end{theorem}

\begin{proof}
	Recall that $B \from \cX_\pi \to \cY_\epsilon'$ is an isometric isomorphism.
	Thus,
	the variational problem~\eqref{e:vf-pde}
	is equivalent to the following equality in $\cX_\pi$,
	\[
		\label{e:Ueq-pde}
		U = B^{-1} \ell + B^{-1} \DIAG \cG_1 U
		.
	\]
	Let $\widehat{U}$, $g$ and $f$ denote the diagonals of
	$U$, $B^{-1}\ell$ and $B^{-1} \DIAG \cG_1 U$.
	These are functions $J \to V_\pi$.
	By the assumptions at the beginning \S\ref{s:pde},
	a $C_0$-semigroup of contractions
	$(S(t))_{t\geq 0}$
	is generated by $-A$
	on $H$ and also on $V$.
	Owing to
    $\DIAG(w,v\otimes \tilde{v}) = \int_J \int_J \Dirac(s-s') \scalar{w(s,s'), v(s) \otimes \tilde{v}(s')}{H_2} \rd s \rd s'$,
	for $w\in\cX_\pi$ and $v,\tilde{v}\in\cY$,
	we can represent
	$B^{-1} \DIAG w \in \cX_\pi$
	explicitly in terms of the semigroup by
	\[
		\label{e:BinvDIAGeq-pde}
		(B^{-1}\DIAG w)(t,t')
		= \int_0^{t\wedge t'} (S(t-s) \otimes S(t'-s)) w(s,s) \rd s
		,
		\quad
		t, t' \in J
		.
	\]
	Set 
	$\cS(r) := S(r)\otimes S(r)$.
	Then
	$(\cS(t))_{t\geq 0}$
	forms a $C_0$-semigroup on $V_\pi$
	generated by
	$\cA := -\mathrm{Id} \otimes A - A \otimes \mathrm{Id}$.
	If $(-A)$ is the Laplacian in $d$ dimensions,
	then $\cA$ is the $2d$-Laplacian.
	By the perturbation theorem~\cite[Thm.~1.3]{Engel2000}, also $\tilde{\cA} := \cA + \cG_1$
	is a generator of a $C_0$-semigroup $(\tilde{\cS}(t))_{t\geq 0}$ on $V_\pi$,
	and
	$\norm{\tilde{\cS}(t)}{\cL(V_\pi)} \leq e^{(C_G - 2\lambda_1)t}$.
	With these definitions, we find that
	$f(s) = \int_0^s \cS(s-r) \cG_1 \widehat{U}(r) \rd r$.
	Thus, the derivative of $f$ satisfies
	$f' = \cA f + \cG_1 \widehat{U} = \tilde{\cA} f + \cG_1 g$
	and $f\in L_1(J;V_\pi)$ can be identified uniquely with
	$f(s) = \int_0^s \tilde{\cS}(s-r) \cG_1 g(r)\rd r$.
	It follows that
	$\widehat{U} = g + f$
	is well-defined in $L_1(J;V_\pi)$.
	By~\eqref{e:Ueq-pde}--\eqref{e:BinvDIAGeq-pde} then,
	$U \in \cX_\pi$ is uniquely determined via
	\[
		\label{e:sol-U-pde}
		U(t,t') = (B^{-1} \ell)(t, t')
		+ \int_0^{t\wedge t'} (S(t-s) \otimes S(t'-s)) \cG_1 [ g(s) + f(s) ] \rd s.
	\]

	Assume now that $\ell\in\cY_\epsilon'$ is SPSD.
	Then, as in Lemma~\ref{l:psd},
	one can show that
	$B^{-1}\ell$ is $\cX$-SPSD
	and symmetry
	of $U$ is evident from the representation~\eqref{e:sol-U-pde}.
	The operators $\cG_1$ and $\cA$
	both preserve $V$-SPSD-ness;
	therefore
	the semigroup $\tilde{\cS}$
	generated by $\tilde{\cA} = \cA + \cG_1$ does, too:
	$\tilde{\cS}(t) w$ is $V$-SPSD, $t \geq 0$,
	if $w \in V_2$ is $V$-SPSD.
 	Therefore, for (a.e.) $s\in J$,
    we have semi-definiteness on $V$
    for quantities appearing under the integral in \eqref{e:sol-U-pde}:
	\begin{equation*}
		\scalar{\cG_1 g(s), \vartheta\otimes\vartheta}{V_2} \geq 0,
		\quad
		\scalar{\cG_1 f(s), \vartheta\otimes\vartheta}{V_2}
		= \int_0^s \scalar{\cG_1\tilde{\cS}(s-r)\cG_1 g(r), \vartheta\otimes\vartheta}{V_2} \rd r \geq 0
        \quad \forall \vartheta\in V.
	\end{equation*}
	Setting $z_\varphi(s) := \int_s^T S(t-s)' \varphi(t) \rd t$
	with the $V$-adjoint $S(r)'$ of $S(r)$,
	we find for all $\varphi \in \cX$
	\begin{equation*}
		\scalar{U, \varphi\otimes \varphi}{\cX_2}
		  = \scalar{B^{-1}\ell, \varphi\otimes \varphi}{\cX_2}
		+ \int_J \scalar{\cG_1 [g(s) + f(s)], z_\varphi(s) \otimes z_\varphi(s)}{V_2} \rd s \geq 0.
	\end{equation*}
	This proves that $U$ is $\cX$-SPSD.
	By Lemma~\ref{l:pinorm-pde} above, we have
	$\norm{U}{\cX_\pi}=\norm{\widehat{U}}{L_1(J;V_\pi)}$ and
	with $\widehat{U} = g + f$ we conclude that
	\begin{equation*}
		\norm{U}{\cX_\pi}
		\leq \norm{\ell}{\cY_\epsilon'} + \int_J \int_0^t \norm{\tilde{\cS}(t-s) \cG_1 g(s)}{V_\pi} \rd s \rd t
		\leq \norm{\ell}{\cY_\epsilon'} + C_G  \int_J \int_s^T e^{(C_G - 2 \lambda_1)(t-s)} \rd t \, \norm{g(s)}{V_\pi} \rd s,
	\end{equation*}
	where we have used~\eqref{e:curlyG1bound} and
	the bound
	$\norm{\tilde{\cS}(t)}{\cL(V_\pi)} \leq e^{(C_G - 2\lambda_1)t}$.
	In this way, the stability estimate~\eqref{e:UCL}
	follows as in the scalar case~\eqref{e:dBDu}.
\end{proof}

\subsection{Second moment discretization}
\label{s:pde:discretization}

In order to introduce conforming discretizations of the
second moment equation in the vector case~\eqref{e:vf-pde},
let $(V^h)_{h>0}$ be a family of finite-dimensional
subspaces of $V$,
whose members carry the same norm as on $V$.
In addition, let $E^k \times F^k \subset E \times F$ be
a discretization pair as considered in \S\ref{s:dis}
with basis functions $\{e_n\}_n \subset E^k$
and $\{v_i\}_i \subset F^k$.
The family $\{e_n\}_n$ is normalized in $L_2(J)$.
As before,
$N := \dim E^k = \dim F^k$
is the dimension of the temporal discretization.
If not specified otherwise, the range of the indices is
\[
    \label{e:ijmnpqrs}
	0 \leq i,j \leq N-1,
	\qquad
	1 \leq m,n \leq N,
	\qquad
	1 \leq p,q,r,s \leq \dim V^h
	.
\]
By choosing
$\cX^{k,h} := E^k \otimes V^h$ and
$\cY^{k,h} := F^k \otimes V^h$
we obtain finite-dimensional subspaces
of the Bochner spaces \eqref{e:XY}.
%
The discrete spaces
$\cX^{k,h}_\pi      := \cX^{k,h} \otimes \cX^{k,h}$ and
$\cY^{k,h}_\epsilon := \cY^{k,h} \otimes \cY^{k,h}$
then form a conforming discretization pair
of the trial and test spaces in~\eqref{e:vf-pde}.
As in \S\S\ref{s:dis:1:o}--\ref{s:dis:2-mul},
the subscript indicates the norm.

The discrete operator $A^h$ on $V^h$ is defined by
$\scalar{A^h \varphi^h, \psi^h}{H} = \duality{A \varphi^h, \psi^h}{}$
for $\varphi^h, \psi^h\in V^h$.
Its eigenvalues and the
the corresponding $H$-orthonormal eigenvectors
are denoted by $\{ \lambda_p^h \}_p$ and
$\{ \varphi_p^h \}_p$, respectively.
We define the bilinear form $b_p$ as in \eqref{e:b},
replacing $\lambda$ by $\lambda_p^h$.

Let the discretization pair $E^k \times F^k$ satisfy
\begin{equation*}
    \label{e:gammakp}
    \gamma_{k,p} := \inf_{w\in S(E^k)} \sup_{v\in S(F^k)} b_p(w,v) > 0,
    \quad
    1 \leq p \leq \dim V^h.
\end{equation*}
Then the inf-sup constant of
\begin{itemize}
		\item
			$b$ from \eqref{e:b-pde}
			on $\cX^{k,h} \times \cY^{k,h}$
			equals $\min_p \gamma_{k,p} > 0$;
			\hfill
	        \eqlab{e:b-pde:g}
		\item
			$B$ from \eqref{e:B-pde}
			on $\cX^{k,h}_\pi \times \cY^{k,h}_\epsilon$
			equals $\min_p \gamma_{k,p}^2 > 0$.
			\hfill
			\eqlab{e:B-pde:g}
\end{itemize}

For the approximation of the vector trace product \eqref{e:DIAG}
we have to take its interaction with the operator $\cG_1$
in the variational problem~\eqref{e:vf-pde}
into account.
Even if $w\in \cX^{k,h}_\pi$ is an element
of the discrete space, this is not necessarily
the case for $\cG_1[w]$.
This necessitates the definition of an approximate
vector trace product on $\cX_\pi \times \cY_\epsilon$.
To this end,
we first note that, for $w\in\cX_\pi$, $ v\in\cY_\epsilon$,
\[
    \label{e:wpq-vpq}
    w_{pq} := \scalar{w,\varphi^h_p\otimes\varphi^h_q}{H_2}
    \AND
    v_{pq} := \scalar{v,\varphi^h_p\otimes\varphi^h_q}{H_2},
    \quad
    1\leq p,q \leq \dim V^h,
\]
can be identified with elements in $E_\pi$ and $F_\epsilon$, respectively.
Furthermore, if we let $P_h$ denote the $H$-orthonormal projection onto $V^h$,
we obtain
$(P_h\otimes P_h) w = \sum_{pq} w_{pq} (\varphi^h_p\otimes\varphi^h_q) \in \cX_\pi$
and, similarly, for $(P_h\otimes P_h) v \in \cY_\epsilon$.
We can then approximate the vector trace product as follows,
\begin{equation*}
    \DIAG(w,v) \approx \DIAG((P_h\otimes P_h) w, (P_h\otimes P_h) v) = \sum_{pq} \DIAG( w_{pq}, v_{pq} )
        \approx \sum_{pq} \DIAG^k( w_{pq}, v_{pq} ),
\end{equation*}
where $\DIAG^k\from E_\pi\times F_\epsilon \to \IR$
is the scalar approximate trace product from~\eqref{e:DIAGk}.
This motivates the following definition
of the approximate trace product in the vector-valued case,
\[
    \label{e:DIAGkh}
    \DIAG^{k,h}\from \cX_\pi\times \cY_\epsilon \to \IR,
    \quad
    \DIAG^{k,h} ( w,v ) := \sum_{pq} \DIAG^k(w_{pq}, v_{pq}),
\]
with $w_{pq}\in E_\pi$ and $v_{pq}\in F_\epsilon$
from~\eqref{e:wpq-vpq}.
We note that the identities
$\DIAG^{k,h} = \DIAG$
and $\DIAG^{k,h}\cG_1 = \DIAG \cG_1$
hold on the discrete subspaces
$\cX^{k,h}_\pi \times \cY^{k,h}_\epsilon$
if $\DIAG^k :=\DIAG$ is the exact scalar trace product.
We furthermore point out that the definition of $\DIAG^{k,h}$ in \eqref{e:DIAGkh}
depends on the subspace $V^h\subset V$,
but it is independent of the choice
of the $H$-orthonormal basis $\{\varphi^h_p\}_p \subset V^h$.

%
Setting
\[
    \label{e:cBkh}
    \cB^{k,h} := B - \DIAG^{k,h} \cG_1,
\]
we introduce the discrete variational problem
\[
	\label{e:vf-dis-pde}
	\text{Find}
	\quad
	U^{k,h} \in \cX^{k,h}_\pi
	\quad\text{s.t.}\quad
	\cB^{k,h}(U^{k,h}, v) = \ell(v)
	\quad
	\forall v \in \cY^{k,h}_\epsilon.
\]

In the following, we suppose that the temporal mesh is uniform.
Then $\boldsymbol{\DIAG}^n$ in~\eqref{e:DIAGn} does not depend on $n$
and,
for all $n$,
\begin{equation*}
	b_p(e_n,v_{n-1}) = b_{p0} := b_p(e_1,v_0)
	\TEXT{and}
	b_p(e_n,v_n) = b_{p1} := b_p(e_1,v_1)
	.
\end{equation*}
Furthermore, $b_{p0}\neq 0$ for all $p$
by~\eqref{e:b-pde:g}.
Under these assumptions,
we derive existence, uniqueness and stability
of a solution $U^{k,h}$ to the discrete problem~\eqref{e:vf-dis-pde}
for any SPSD functional $\ell$.
As in \S\ref{s:dis:2-mul},
the result is formulated using the following constants:
\[
	\beta
	&
	:= (1-\tilde{\beta})^{-1},
	\qquad
	\tilde{\beta}
	:=
	\norm{
		P_h G_1[\nativecdot] Q^{1/2}
	}{
		\cL(V^h; \cL_2(\cU; V^h))
	}^2
	\DIAG_{11} \max_p b_{p0}^{-2}
	,
	\label{e:def:beta-pde}
	\\
	\theta_{+} & := \max_p |\theta_p|,
	\qquad
	\theta_{-} := \min_p |\theta_p|,
	\qquad
	\theta_p := b_{p0}^{-1} b_{p1},
	%
    \notag
	\\
	\alpha
	& :=
	\beta \max_{p}
	\left\{
		\theta_+^2 + b_{p0}^{-2}
		|\DIAG_{22} - 2\DIAG_{12}\theta_p| \,
		\norm{
			P_h G_1[\nativecdot] Q^{1/2}
		}{
			\cL(V^h; \cL_2(\cU; V^h))
		}^2
	\right\}.
	%
    \notag
\]
%
These quantities should be compared to
\eqref{e:betan}--\eqref{e:alphan}
from the scalar case.
The following result is the analogue of Theorem~\ref{t:stab:dis}.

\begin{theorem}
	\label{t:stab:dis-pde}
	Suppose the temporal mesh is unform and that $\beta > 0$.
	Then, if $\ell\in\cY_\epsilon'$ is SPSD,
	the discrete variational problem~\eqref{e:vf-dis-pde}
	has a unique solution $U^{k,h}\in\cX^{k,h}_\pi$,
	it is $\cX$-SPSD,
	and
	\[
		\label{e:dis-stab-pde}
		\norm{U^{k,h}}{\cX_\pi} \leq C_{k,h}  \norm{\ell}{\cY_\epsilon'},
		\quad
		\text{where}
		\quad
		C_{k,h} := \max_p \gamma_{k,p}^{-2} \, \beta ( 1 + (\alpha - \theta_-^2 ) \tfrac{\alpha^{N-1}-1}{\alpha-1} ).
	\]
\end{theorem}

Let
$\bbM^h_{+}$
denote the set of SPSD matrices of size $\dim V^h \times \dim V^h$.
Define the matrix-valued operator
$\cT$ on $\bbM^h_{+}$ componentwise by
\[
	\label{e:def:T}
	(\cT \Wmat)_{pq}
	:= \DIAG_{11} b_{p0}^{-1} b_{q0}^{-1} \sum_{rs} W_{rs} \scalar{\cG_1[\varphi_r^h\otimes\varphi_s^h], \varphi_p^h\otimes\varphi_q^h}{H_2}
	= \DIAG_{11} b_{p0}^{-1} b_{q0}^{-1} \scalar{\cG_1 W, \varphi_p^h\otimes\varphi_q^h}{H_2}
	,
\]
where the second equality holds whenever
\[
	\label{e:rel-W-Wh}
	\Wmat = (W_{rs})_{rs} \in \IR^{\dim V^h \times \dim V^h}
	\TEXT{and}
	W := \sum_{rs} W_{rs} \varphi_r^h\otimes \varphi_s^h\in V^h \otimes V^h
	.
\]
The following lemma is the key ingredient
for the proof of Theorem~\ref{t:stab:dis-pde}.
For
$\Lambdamat := \operatorname{diag}(\lambda_p^h)_p$,
we introduce the weighted trace
$
\tr_{\Lambdamat}(\Wmat)
:=
\tr( \Lambdamat^{1/2} \Wmat \Lambdamat^{1/2})
$.
Note that
$\tr_{\Lambdamat}(\Wmat) = \norm{W}{V_\pi}$
for  $\Wmat\in\bbM^h_{+}$.
%
%

\begin{lemma}
	\label{l:T}
	Recall $\tilde{\beta}$
	and $\beta$ from~\eqref{e:def:beta-pde}.	
	For
	$\Wmat \in \bbM^h_{+}$ the following hold:
	\begin{enumerate}[label={(\roman*)}]
		\item
			$\bbM^h_{+}$ is invariant under $\cT$, i.e., $\cT \Wmat \in \bbM^h_{+}$.
		\item
			$\tr_{\Lambdamat}(\cT \Wmat) \leq \tilde{\beta} \tr_{\Lambdamat}(\Wmat)$.
		\item
			If $\tilde{\beta} < 1$
			then $(\mathrm{Id} - \cT)^{-1} \Wmat$ exists in $\bbM^h_{+}$
            and
            $\tr_{\Lambdamat}((\mathrm{Id} - \cT)^{-1} \Wmat) \leq \beta \tr_{\Lambdamat}(\Wmat)$.
	\end{enumerate}
\end{lemma}

\begin{proof}
	With \eqref{e:rel-W-Wh},
	we have
	$\Wmat \in \bbM^h_{+}$
	if and only if
	$W \in  V^h\otimes V^h$
	is $V$-SPSD.
	Since $\cG_1$ preserves this property,
	see \eqref{e:defcurlyG1},
	and $\DIAG_{11} > 0$,
	the claim (i) follows.
%
	%
	For (ii), let
	$W = \sum_q s_q \psi_q \otimes \psi_q$
	be an expansion
	with $s_q \geq 0$ and
    $V$-orthonormal $\psi_q \in V^h$
	for which
    $\tr_{\Lambdamat}(\Wmat) = \norm{W}{V_\pi} = \sum_q s_q$.
	The assertion (ii) follows:
	\begin{equation*}
		\tr_{\Lambdamat}(\cT \Wmat)
		=
		\sum_{p} (\cT \Wmat)_{pp} \lambda_p^h
		\leq \DIAG_{11} \max_p b_{p0}^{-2}
		\max_{q} \norm{P_h G_1[\psi_q] Q^{1/2}}{\cL_2(\cU;V^h)}^2
		\tr_{\Lambdamat}(\Wmat)
		\leq \tilde{\beta} \tr_{\Lambdamat}(\Wmat),
	\end{equation*}
	since, letting $P_h'$ denote the $H$-adjoint of the projection $P_h$, we find
	\begin{equation*}
		\sum_p
		\scalar{\cG_1[\psi_q \otimes \psi_q], \varphi_p^h \otimes \varphi_p^h}{H_2}
		\lambda_p^h
		=
		\sum_p
		\norm{ Q^{1/2} G_1[\psi_q]' P_h' \varphi_p^h }{\cU}^2
		\lambda_p^h
		=
		\norm{P_h G_1[\psi_q] Q^{1/2}}{\cL_2(\cU;V^h)}^2
	\end{equation*}
	and $\psi_q$ has unit $V$-norm.
	Finally,
	if $\tilde{\beta} < 1$
	then the Neumann series
	$\sum_{n \geq 0} \cT^n \Wmat$
	consists of terms in $\bbM^h_{+}$
	and converges,
	which gives (iii).
\end{proof}

\begin{proof}[Proof of Theorem~\ref{t:stab:dis-pde}]
	Consider the expansion of $U^{k,h}$ in terms of the
	basis $\{e_m\otimes\varphi_p^h \}_{m,p}\subset \cX^{k,h}$
	\begin{equation*}
		U^{k,h} = \sum_{mn,pq} U_{mn,pq} (e_m \otimes \varphi_p^h) \otimes (e_n \otimes \varphi_q^h) \in \cX^{k,h} \otimes \cX^{k,h}.
	\end{equation*}
	Similarly, let $G_{mn,pq}$ denote the corresponding coefficients of the solution $G^{k,h}$ to the problem
	\begin{equation*}
		\text{Find}
		\quad
		G^{k,h} \in \cX^{k,h}_\pi
		\quad\text{s.t.}\quad
		B(G^{k,h}, v) = \ell(v)
		\quad
		\forall v \in \cY^{k,h}_\epsilon.
	\end{equation*}
    Define the matrices
	$\Umat_n := (U_{nn,pq})_{pq}$
	and
	$\Gmat_n := (G_{nn,pq})_{pq}$.
	Since $\ell$ is SPSD,
	we have $\Gmat_n \in \bbM^h_{+}$.
    By testing \eqref{e:vf-dis-pde} with
    $v_{ij,pq} := (v_i \otimes \varphi_p^h) \otimes (v_j \otimes \varphi_q^h)$
    for fixed $p,q$, we find as in \eqref{e:Umn} that
    \begin{equation*}
        b_{p0} b_{q0} U_{nn,pq}
            = b_{p0} b_{q0} G_{nn,pq}
                + \sum_{i=0}^{n-1} \sum_{j=0}^{n-1} [\Pi_p]_{i}^{n-1} [\Pi_q]_{j}^{n-1} \DIAG^{k,h}\cG_1 (U^{k,h}, v_{ij,pq})
                ,
    \end{equation*}
    where
    $
		[\Pi_p]_{i}^{n} := ( -b_{p1}  / b_{p0} )^{n-i}
	$.
    After rearranging, this gives
    $( \Umat_1 - \cT\Umat_1)_{pq} = (\Gmat_1)_{pq}$
    when $n = 1$,
    and, for $n\geq 2$,
    \[
        \label{e:Id-T,pq}
        (\Umat_n - \cT\Umat_n)_{pq}
             = \theta_p\theta_q ( \Umat_{n-1} - \Gmat_{n-1})_{pq}
               + (\Gmat_n)_{pq}
               + \DIAG_{11} ^{-1} (\DIAG_{22} - \DIAG_{12}\theta_p - \DIAG_{21}\theta_q )(\cT \Umat_{n-1})_{pq},
    \]
    where $\cT$ is the operator from~\eqref{e:def:T}.
    In terms of $\boldsymbol{\theta}_p := (-\theta_p, 1)^\T$ and
    $\boldsymbol{\DIAG}$ we thus have
    \[
        \label{e:Id-T,pq2}
        (\Umat_n - \cT\Umat_n)_{pq}
            &= \theta_p\theta_q ( \Umat_{n-1} - \cT\Umat_{n-1} - \Gmat_{n-1})_{pq}
               + (\Gmat_n)_{pq}
               + \DIAG_{11} ^{-1} ( \boldsymbol{\theta}_p^\T \boldsymbol{\DIAG} \boldsymbol{\theta}_q )(\cT \Umat_{n-1})_{pq}.
    \]
    We define the diagonal matrices
    $\Thetamat = \operatorname{diag}(\theta_p)_p$
	and
	$\Dmat_\eta := \DIAG_{11}^{-1/2} \operatorname{diag}(L_{1\eta} \theta_p - L_{2\eta})_p$,
	for $\eta=1,2$ and a Cholesky factor $\Lmat$ of
	$\boldsymbol{\DIAG} = \Lmat\Lmat^\T$,
    and can then express \eqref{e:Id-T,pq2}
    in matrix form as
	\[
		\label{e:Id-T}
		(\mathrm{Id} - \cT) \Umat_n
		& =
		\Thetamat (\mathrm{Id} - \cT) \Umat_{n-1} \Thetamat
		+
		\Gmat_n - \Thetamat \Gmat_{n-1} \Thetamat
		+
		\sum_{\eta=1}^2 \Dmat_\eta \cT \Umat_{n-1}  \Dmat_\eta,
        \quad
        n \geq 2.
	\]
    For $n=1$, we have $(\mathrm{Id} - \cT) \Umat_1 = \Gmat_1$.
    By induction, it follows from \eqref{e:Id-T} that
	\[
		\label{e:Id-T,2}
		(\mathrm{Id} - \cT) \Umat_n
		& =
		\Gmat_n
		+
		\sum_{\eta=1}^2 \sum_{\nu=1}^{n-1}
		\Thetamat^{n-1-\nu} \Dmat_\eta \cT \Umat_{\nu} \Dmat_\eta \Thetamat^{n-1-\nu}.
	\]
	Since $\beta > 0$ by assumption,
	$(\mathrm{Id} - \cT)$ is invertible
	on $\bbM^h_{+}$
	by Lemma~\ref{l:T},
	so that
	\eqref{e:Id-T,2}
	defines
	$\Umat_1, \ldots, \Umat_N$
	in $\bbM^h_{+}$.
	Let $\widehat{U}^{k,h}:=\sum_{n,pq} U_{nn,pq} (e_n\otimes\varphi_p^h)\otimes(e_n\otimes\varphi_q^h)$
	and $\widehat{\ell} := \DIAG^{k,h} \cG_1 \widehat{U}^{k,h} + \ell$.
	The fact that $\cG_1$ preserves $H$-SPSD-ness
	and $\Umat_1,\ldots,\Umat_N\in\bbM^h_{+}$
	imply that $\widehat{\ell}$ is SPSD on $\cY^{k,h}$.
	Owing to~\eqref{e:B-pde:g}, there exists
	a unique $\cX$-SPSD $U^{k,h}\in\cX^{k,h}_{\pi}$ with
	$B(U^{k,h},v) = \widehat{\ell}(v)$ for all $v\in \cY_\epsilon^{k,h}$.
	As in the scalar case, we conclude from the construction
	of $\widehat{U}^{k,h}$ and $U^{k,h}$ that
	$U_{nn,pq} = \widehat{U}_{nn,pq}$
	so that $U^{k,h}$ is the unique solution to~\eqref{e:vf-dis-pde}.
	The $\cX$-SPSD-ness of $U^{k,h}$ and
	the $L_2(J;H)$-orthonormality of
	$\{e_m \otimes \varphi_p^h \}_{m,p} \subset \cX^{k,h}$
	yield
	\[
		\label{e:Ukh-stab1}
		\norm{U^{k,h}}{\cX_\pi}
		= \sum_{n,p} \lambda_p^h U_{nn,pp}
		= \sum_n \tr_{\Lambdamat} ( \Umat_n )
		.
	\]
	A similar
	equality holds also for $G^{k,h}$.
    For $n\geq 2$, we estimate with Lemma~\ref{l:T}~(i)--(ii)
    \[
    \label{e:trUn}
        0 \leq \beta^{-1} \tr_{\Lambdamat}( \Umat_n )
            = (1 - \tilde\beta) \tr_{\Lambdamat}( \Umat_n )
            \leq \tr_{\Lambdamat}( \Umat_n ) - \tr_{\Lambdamat}( \cT\Umat_n )
            = \tr_{\Lambdamat}( (\mathrm{Id} - \cT) \Umat_n ),
    \]
	and use the identity \eqref{e:Id-T,pq}
    as well as
    $\DIAG_{12} \tr_\Lambda( \Thetamat \cT\Umat_{n-1} ) =\DIAG_{21} \tr_\Lambda( \cT\Umat_{n-1} \Thetamat )$
	to derive the bound
	\[
    \begin{split}
    \label{e:trUn-TUn}
		\tr_\Lambda ( (\mathrm{Id} - \cT) \Umat_n )
		&=
		\tr_\Lambda( \Thetamat \Umat_{n-1} \Thetamat
		+ \DIAG_{11}^{-1} ( \DIAG_{22} - 2\DIAG_{12}\Thetamat) \cT\Umat_{n-1}
		+ \Gmat_n
		- \Thetamat \Gmat_{n-1} \Thetamat )
		\\
		& \leq
		\beta^{-1} \alpha \tr_\Lambda( \Umat_{n-1} )
		+ \tr_\Lambda( \Gmat_n )
		- \theta_-^2 \tr_\Lambda( \Gmat_{n-1} ).
    \end{split}
	\]
    Furthermore, by Lemma~\ref{l:T}~(iii) above
	we find
    $\tr_{\Lambdamat}( \Umat_1 )
    = \tr_{\Lambdamat}( (\mathrm{Id} - \cT)^{-1} \Gmat_1 )
    \leq \beta \tr_{\Lambdamat}( \Gmat_1 )$.
	By combining this with \eqref{e:trUn}--\eqref{e:trUn-TUn} we obtain by induction, for all $n$,
	\begin{equation*}
		\tr_{\Lambdamat}( \Umat_n )
            \leq \beta \tr_{\Lambdamat}( \Gmat_n )
                + \beta( \alpha - \theta_-^2 ) \sum_{\nu=1}^{n-1} \alpha^{n-1-\nu} \tr_{\Lambdamat}( \Gmat_\nu ).
	\end{equation*}
	Inserting this estimate in~\eqref{e:Ukh-stab1}
	and changing the order of summation in the second term gives
	\begin{equation*}
		\norm{U^{k,h}}{\cX_\pi}
		\leq
		\beta \norm{G^{k,h}}{\cX_\pi}
		+ \beta( \alpha - \theta_-^2 ) \sum_{\nu=1}^{N-1} \tr_{\Lambdamat}( \Gmat_\nu ) \sum_{n=0}^{N-1-\nu} \alpha^{n}
		\leq
		\beta( 1 + ( \alpha - \theta_-^2 ) \tfrac{\alpha^{N-1} - 1}{\alpha - 1} )  \norm{G^{k,h}}{\cX_\pi}
		.
	\end{equation*}
    The application of the discrete stability
    estimate
	$
    \norm{G^{k,h}}{\cX_\pi}
	\leq
	(\min_p \gamma_{k,p}^2)^{-1}
	\norm{\ell}{\cY_\epsilon'}
	$
	from \eqref{e:B-pde:g}
    completes the proof of the stability bound~\eqref{e:dis-stab-pde}.
\end{proof}

As for the scalar case, the discrete stability estimate~\eqref{e:dis-stab-pde} implies
an inf-sup condition for $\cB^{k,h}$ in \eqref{e:cBkh} on the
subspaces $\widehat{\cX}^{k,h}_\pi \subset \cX^{k,h}_\pi$ and
$\widehat{\cY}^{k,h}_\epsilon \subset \cY^{k,h}_\epsilon$
of symmetric elements~\eqref{e:HSPSD}.
Subsequently, Proposition~\ref{p:well:crime} is applicable, which
gives a quasi-optimality estimate for the $\text{CN}^{\star}_2$ scheme.
These observations are summarized in the following theorem.

\begin{theorem}\label{thm:conv-pde}
Suppose that
the temporal mesh is uniform with $\beta > 0$, and
let $\cG_1\in\cL(V_\pi)$
with norm $C_G = \eqref{e:curlyG1bound}$.
Then $\cB^{k,h}$ in \eqref{e:cBkh} satisfies the
discrete inf-sup condition
\[
    \label{e:inf-sup-cBkh}
    \inf_{w\in S(\widehat{\cX}^{k,h}_\pi)} \sup_{v\in S(\widehat{\cY}^{k,h}_\epsilon)} \cB^{k,h}(w,v) \geq C_{k,h}^{-1},
\]
where $C_{k,h}$ is the discrete stability constant in \eqref{e:dis-stab-pde}.
If $\ell\in\cY_\epsilon'$ is symmetric,
the error between the exact solution $U\in\cX_\pi$ to \eqref{e:vf-pde}
and the discrete solution $U^{k,h} \in \cX^{k,h}_\pi$ to \eqref{e:vf-dis-pde}
for the $\text{CN}^{\star}_2$ scheme
admits the bound
\[
    \label{e:quasiopti-pde}
    \norm{U - U^{k,h}}{\cX_\pi}
        \leq (1 + C_{k,h} \norm{\cB^{k,h}}{} ) \inf_{w\in\cX^{k,h}} \norm{U - w}{\cX_\pi},
\]
where $\norm{\cB^{k,h}}{} \leq 1 + \tfrac{C_G}{2}$ is the operator norm
of $\cB^{k,h}\from \cX_\pi \to \cY_\epsilon'$ induced by \eqref{e:cBkh}.
\end{theorem}

\begin{proof}
Since the inf-sup estimate \eqref{e:inf-sup-cBkh}
follows by exactly the same arguments as in
the scalar case, Corollary~\ref{cor:cBinfsup},
we focus on the derivation of the
quasi-optimality estimate \eqref{e:quasiopti-pde}.
By Proposition~\ref{p:well:crime} we have
\begin{equation*}
    \norm{U - U^{k,h}}{\cX_\pi}
        \leq (1 + C_{k,h} \norm{\cB^{k,h}}{}) \inf_{w\in\widehat{\cX}^{k,h}} \norm{U - w}{\cX_\pi}
             + C_{k,h} \norm{ (\DIAG - \DIAG^{k,h}) \cG_1 U}{(\widehat{\cY}^{k,h}_\epsilon)'}.
\end{equation*}
For the exact scalar trace product $\DIAG^k := \DIAG$,
the definition of $\DIAG^{k,h}$ in \eqref{e:DIAGkh}
gives
\begin{equation*}
    \DIAG^{k,h}(w,v) = \DIAG(w,(P_h \otimes P_h)v)
    \quad
    \forall (w,v) \in \cX_\pi \times \cY_\epsilon.
\end{equation*}
This shows that the residual term
$\norm{ (\DIAG - \DIAG^{k,h}) \cG_1 U }{(\widehat{\cY}^{k,h}_\epsilon)'}$ vanishes
for the $\text{CN}^{\star}_2$ scheme.
Furthermore, $\cB^{k,h}$  is continuous on
$\cX_\pi \times \cY_\epsilon$ with $\norm{\cB^{k,h}}{} \leq 1 + \tfrac{C_G}{2}$,
since
Lemmas~\ref{l:YC0}--\ref{l:pinorm-pde} and $\norm{P_h}{\cL(H)} = 1$
yield the bound
$\DIAG(\cG_1 w, (P_h\otimes P_h)v)
    \leq \tfrac{C_G}{2} \, \norm{w}{\cX_\pi} \norm{v}{\cY_\epsilon}$.
The quasi-optimality estimate \eqref{e:quasiopti-pde}
formulated with respect to $\widehat{\cX}^{k,h}_\pi$
instead of $\cX^{k,h}_\pi$ follows.
Finally, if $U\in\cX_\pi$ is symmetric,
taking the symmetrization $\tfrac{1}{2}( w + w^*)$ of $w\in\cX^{k,h}_\pi$,
where
$w^*(s,t) := \sum_{pq} \scalar{w(t,s), \varphi_q^h \otimes \varphi^h_p}{H_2} \varphi^h_p \otimes \varphi_q^h$,
gives
$\norm{U-\tfrac{1}{2}(w+w^*)}{\cX_\pi} \leq \norm{U-w}{\cX_\pi}$,
since $U^* = U$.
This proves \eqref{e:quasiopti-pde} on $\cX^{k,h}_\pi$.
\end{proof}

\subsection{Numerical example}
\label{s:pde:numexp}

For the following numerical experiment
we set $T:=1$,
$H := L_2(0, 1)$ equipped with the usual inner product,
and $A := -(\CDOT)''$ with homogeneous Dirichlet boundary conditions.
Then $V := \cD(A^{1/2}) = H_0^1(0, 1)$, and
the norm on $V$ is the $H^1$ semi-norm.
The eigenvalues and $H$-orthonormal eigenfunctions of $A$ are
\begin{equation*}
	\lambda_\nu = \nu^2 \pi^2
	\TEXT{and}
	\varphi_\nu(x) = \sqrt{2} \sin(\nu \pi x),
    \qquad
    \nu=1,2,\ldots
\end{equation*}
Furthermore, the sequence $\{\psi_\nu\}_\nu$,
defined by $\psi_\nu := \lambda_\nu^{-1/2}\varphi_\nu$,
forms an orthonormal basis of~$V$.
For the noise $L$ of the stochastic PDE~\eqref{e:spde-mul}, we choose
a $Q$-Wiener process  $L := (W^Q(t),t\geq0)$
taking values in the Sobolev space $\cU := V$.
We assume that the trace-class covariance operator $Q\in\cL(V)$
diagonalizes with respect to the orthonormal
basis $\{\psi_\nu\}_\nu$ of $V$, i.e., there exists
a summable sequence $\{\mu_\nu\}_\nu$
of nonnegative real numbers such that
$Q\phi = \sum_\nu \mu_\nu \scalar{\phi, \psi_\nu}{V} \psi_\nu$
for all $\phi \in V$.
Finally, we specify the affine operator
$G[\CDOT] = G_1[\CDOT] + G_2$
in \eqref{e:spde-mul}.
We set $G_2:=0$ and let~$G_1$
be a Nemytskii operator,
$G_1[\phi]\psi : x \mapsto \phi(x) \psi(x)$
for $\phi\in H$, $\psi\in V$.

In this case, well-posedness of the
deterministic variational problem~\eqref{e:vf-pde}
is ensured by
Theorem~\ref{t:stab-mul-pde}
for any $\ell\in\cY_\epsilon'$, since
\[
    \label{e:nemytskii-bounded}
    \norm{G_1[\phi]Q^{1/2}}{\cL_2(\cU;V)}^2
        = \sum_\nu \mu_\nu \norm{G_1[\phi] \psi_\nu}{V}^2
        \leq 8 \lambda_1^{-1} \tr(Q) \norm{\phi}{V}^2
        \quad
        \forall\phi\in V,
\]
and, therefore,
$C_G\leq 8 \lambda_1^{-1} \tr(Q)$
in~\eqref{e:curlyG1bound}.
In the last step of \eqref{e:nemytskii-bounded}
we have used the bounds
$|\psi_\nu (x)|^2 \leq 2 \lambda_\nu^{-1}$ and
$|\psi_\nu'(x)|^2 \leq 2$
for $x\in[0,1]$, as well as
$\norm{\phi}{H}^2 \leq \lambda_1^{-1} \norm{\phi}{V}^2$
to derive, for all $\nu$,
\begin{equation*}
    \norm{G_1[\phi] \psi_\nu}{V}^2 \leq 2 \int_0^1 ( |\phi'(x) \psi_\nu(x)|^2 + |\phi(x) \psi_\nu'(x)|^2 ) \, \rd x
        \leq 4  (\lambda_\nu^{-1} + \lambda_1^{-1}) \norm{\phi}{V}^2
    \quad \forall \phi\in V.
\end{equation*}

In order to discretize the problem,
we let $V^h := \operatorname{span}\{\varphi_1, \ldots, \varphi_{\dim V^h}\} \subset V$
be the subspace generated by the first $\dim V^h$
eigenfunctions of $A$.
We recall the range~\eqref{e:ijmnpqrs} of the indices $p,q,r,s$
and define the functional $\ell_{pq}\in F_\epsilon'$
as the unique continuous linear extension of
\begin{equation*}
    \ell_{pq}(v\otimes\tilde{v}) := \ell((v\otimes\varphi_p) \otimes (\tilde{v} \otimes \varphi_q))
    \quad
    \forall v, \tilde{v} \in F.
\end{equation*}
Besides, we introduce the notation
$B_{pq} := b_p \otimes b_q$ and
$\rho_{pq,rs} := \scalar{\cG_1[\varphi_r \otimes \varphi_s], \varphi_p\otimes\varphi_q}{H_2}$.
%
Then the coefficients $U_{pq} := \scalar{U,\varphi_p\otimes\varphi_q}{H_2}$,
cf.~\eqref{e:wpq-vpq}, of the semi-discrete solution
$\sum_{p,q=1}^{\dim V^h} U_{pq} (\varphi_p\otimes \varphi_q)$
satisfy the following system of variational problems
on $E_\pi \times F_\epsilon$:
\[
\label{e:system-pde-ex}
    \text{Find}
    \quad
    U_{pq} \in E_\pi
    \TEXT{s.t.}
    B_{pq}(U_{pq} , v) - \sum_{r,s=1}^{\dim V^h} \rho_{pq,rs} \DIAG(U_{rs}, v)
        = \ell_{pq}(v)
        \quad
        \forall v\in F_\epsilon
        .
\]

For the simulation, we choose the functional $\ell\in\cY_\epsilon'$ as
the right-hand side of the second moment equation
$\ell(v) := \ell_M(v) = \duality{\bbE[X_0 \otimes X_0], v(0)}{\pi,\epsilon}$,
and
we take the deterministic initial value
\linebreak
$X_0(x):= \sqrt{30} (x-x^2)$
from \cite[\S4]{LangPetersson2018}, which is normalized to $\norm{X_0}{H}=1$.
We furthermore let $\mu_\nu := 32 \nu^{-5}$
be the eigenvalues of the covariance operator $Q$, i.e.,
$\mu_\nu = C \lambda_\nu^{-r}$
for $C = 32\pi^5$, $r = 5/2$, and, thus,
$Q$ is a constant multiple of the inverse fractional $1d$-Laplacian.
Stochastic processes with covariance operators of this type
are sometimes called Riesz fields, see
\cite[\S4]{vanWykGunzburgerBurkhardtStoyanov2015}.

As a reference solution $U^{\mathrm{ref}}$
for the second moment $U=\bbE[X\otimes X]\in\cX_\pi$
of the solution~$X$ to~\eqref{e:spde-mul},
we take
the Monte Carlo estimator from $2^{12}=h_{\mathrm{ref}}^{-2}$ sample paths
generated with the 
Euler--Maruyama
method with a constant time step size $k_{\mathrm{ref}}=2^{-15}$
and continuous piecewise
affine basis functions on spatial
grid with uniform mesh width $h_{\mathrm{ref}}=2^{-6}$.
The sample paths of the $Q$-Wiener process are simulated
from a truncation of the representation
$W^Q(t) = \sum_{\nu} \mu_\nu W_\nu(t) \psi_\nu$,
where $\{W_\nu\}_\nu$ is a sequence of
mutually independent real-valued Wiener processes.
Since the decay of the eigenvalues of $Q$ is given by
$\mu_\nu \lesssim \nu^{-\eta}$
for $\eta = 5$, 
we truncate this series after $\kappa := (h_{\mathrm{ref}})^{-\frac{2}{\eta-1}} = 8$ terms,
see \cite[Thm.~3.2]{LangPetersson2018}.

We let $\dim V^h = 5$ and discretize the system~\eqref{e:system-pde-ex}
with the $\text{CN}_2^*$ scheme proposed in \S\ref{s:dis:2-mul}. 
To this end, we use the representation
$\rho_{pq,rs} = \sum_\nu \mu_\nu \sigma^\nu_{p,r} \sigma^\nu_{q,s}$,
where
$\sigma^\nu_{p,r}
:=
\scalar{ G_1[\varphi_r] \psi_\nu, \varphi_p }{H}$,
and use the same truncation for this series
as for the $Q$-Wiener process
in the simulation of the reference solution, i.e.,
we truncate after $\kappa = 8$ terms.
By evaluating the integrals
we find
\begin{equation*}
    \sigma^\nu_{p,r}
	=
    \int_0^1 \varphi_r(x) \psi_\nu(x) \varphi_p(x) \rd x
    =
	\begin{cases}
    	\lambda_\nu^{-1/2}
	\frac{
	- 8 \sqrt{2}
	\nu p r
	}{
	\pi (\nu + p + r) (\nu - p + r) (\nu + p - r) (\nu - p - r)
	}
	&
	\text{if $(\nu + p + r)$ is odd}, \\
	0
	&
	\text{otherwise}.
	\end{cases}
\end{equation*}
In this way, we obtain approximations $U^k_{pq} \in E^k_\pi$
of the coefficients $U_{pq}\in E_\pi$,
and an overall approximation
$U^{k,h} = \sum_{pq} U_{pq}^{k} (\varphi_p\otimes \varphi_q) \in \cX^{k,h}_\pi$
of the solution $U\in\cX_\pi$ to~\eqref{e:vf-pde}.

We use the symmetrization and preconditioning
from \S\ref{s:numex}
and solve the discretized system with
the conjugate gradients method
by applying the MATLAB \texttt{pcg} solver
with tolerance $10^{-10}$.
For $1\leq p\leq \dim V^h=5$,
we measure the error of $U^{\mathrm{num}}_{pp}:=U^{k}_{pp}$
against the coefficient~$U_{pp}^{\mathrm{ref}}$ of the reference solution
as the $L_1$ error on the diagonal
$\mathrm{E}_p(U^{\mathrm{num}}) := \diag(|U_{pp}^{\mathrm{ref}} - U_{pp}^{\mathrm{num}}|)$,
as in \S\ref{s:numex}.
Finally, we approximate the total error
$\norm{U - U^{k,h}}{\cX_\pi}$ by the weighted sum
$\mathrm{E}(U^{\mathrm{num}}) := \sum_{p} \lambda_p \mathrm{E}_p(U^{\mathrm{num}})$,
motivated by Lemma~\ref{l:pinorm-pde}
which gives
$\norm{w}{\cX_\pi} = \norm{\widehat{w}}{L_1(J;V_\pi)} = \sum_{\nu\in\IN} \lambda_\nu \diag(w_{\nu\nu})$
for every $\cX$-SPSD $w\in\cX_\pi$.
The results are presented in Figure~\ref{f:num-pde},
showing first order convergence with respect to the
temporal discretization parameter $k$
for every coefficient $U^{\mathrm{num}}_{pp}$ as well as for
the measure of the total error,
which is in accordance with Theorem~\ref{thm:conv-pde}.

\begin{figure}[t]
	\includegraphics[width=0.55\textwidth]{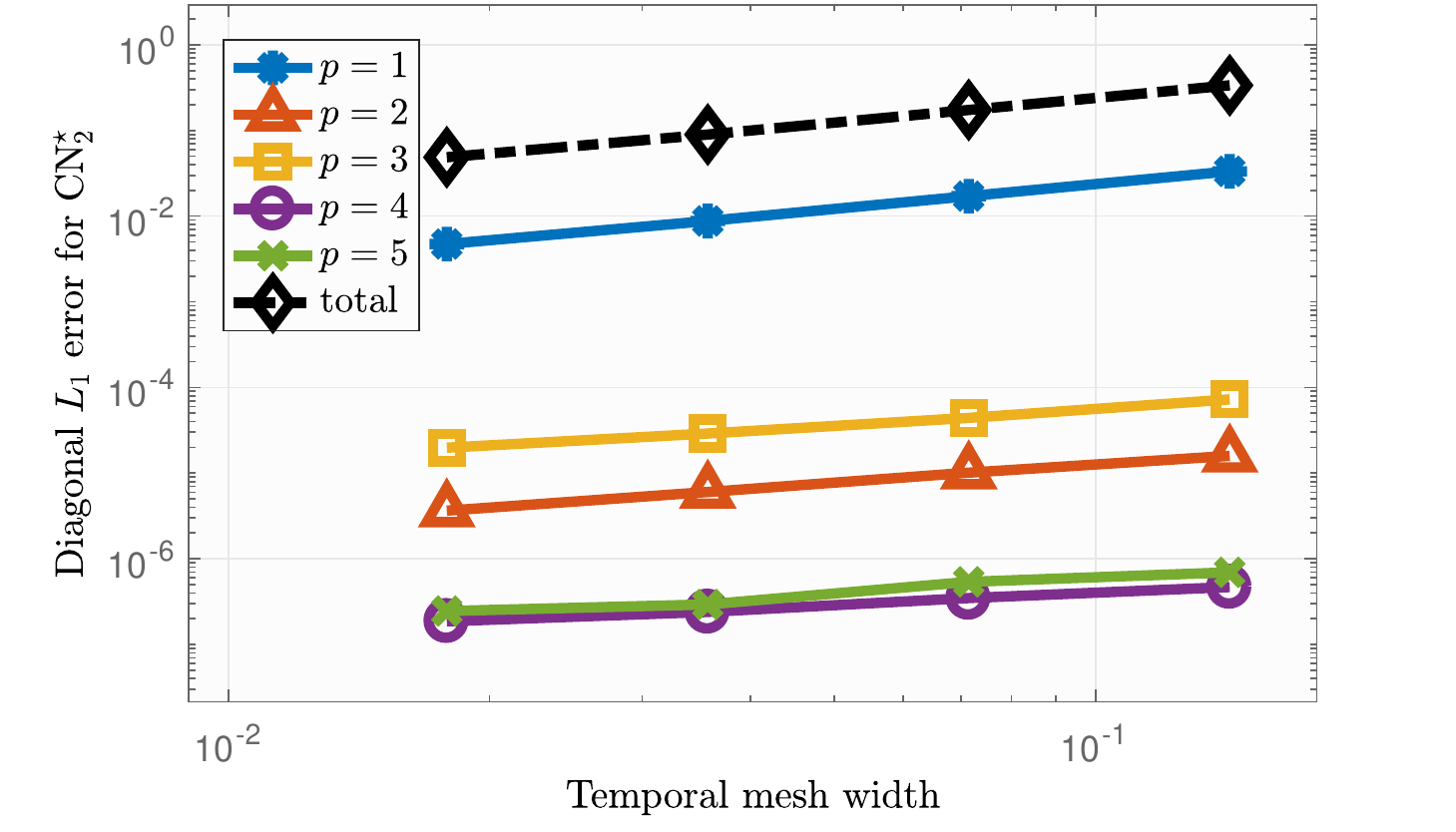}
    \caption{%
		The error of the coefficients
		$\mathrm{E}_p(U^{\mathrm{num}}) = \diag(|U_{pp}^{\mathrm{ref}} - U_{pp}^{\mathrm{num}}|)$
        as well as
        the total error $\mathrm{E}(U^{\mathrm{num}}) = \sum_p \lambda_p \mathrm{E}_p(U^{\mathrm{num}})$
		as a function of the temporal mesh width
		for the example from
		\S\ref{s:pde:numexp}.
		\label{f:num-pde}
	}
\end{figure}

\section{Conclusions}
\label{s:end}

We have considered the
model stochastic ODEs
\eqref{e:sode-add}, \eqref{e:sode-mul}
with additive and multiplicative Wiener noise
and
have derived
the deterministic equations
in variational form
satisfied
by the first \eqref{e:meaneqn} and second moment \eqref{e:U-add}, \eqref{e:U-mul}
of the solution.
The equations for the second moment
are
posed on tensor products
of function spaces,
which
can be taken
as
Hilbert tensor products
\eqref{e:E2F2}
in the additive case,
whereas
projective--injective
tensor product spaces
\eqref{e:EpiFeps}
as trial--test spaces
are required
in the multiplicative case.
The well-posedness
of these equations
is evident
in the additive case \eqref{e:U-add}
by the isometry property of the operator \eqref{e:B},
but
the multiplicative case,
analyzed in Theorem \ref{t:stab-mul},
requires more work due to
the presence of the trace product \eqref{e:D}
in the operator.

We have discussed
Petrov--Galerkin discretizations
of two basic kinds
for
the first moment:
$\text{CN}^\star$ (\S\ref{s:dis:1:CN*})
and
$\text{iE}^\star$ (\S\ref{s:dis:1:iE*}).
The main difference is in
the stability behavior documented in Figure~\ref{f:CN-iE},
wherein
$\text{CN}^\star$ requires the CFL number to be small,
as opposed to
$\text{iE}^\star$
which
can be made stable \eqref{e:p:iE*} under mild restrictions
on the temporal mesh.
Higher order generalizations
followed in \S\ref{s:dis:1:both}.
From these,
tensor product Petrov--Galerkin discretizations
are constructed
in \S\ref{s:dis:1:o}.
We have addressed the additive case briefly in \S\ref{s:dis:2-add}
in order to focus
the multiplicative case in \S\ref{s:dis:2-mul}.

Trying to harness
the favorable stability properties
of the $\text{iE}^\star$ discretization,
two problems arise in the multiplicative case:
lack of density of the trial spaces
(see \S\ref{s:dis:1:iE*})
and
inconsistent interaction
of the basis functions with
the trace product
(see \S\ref{s:iE2}).
The first issue is addressed by postprocessing \eqref{e:Ukpost}
and
the second
by a modification of the trace product
(we have suggested the two variants
{$\text{iE}^{\star}_2/Q$}
and
{$\text{iE}^{\star}_2/\boxrule$}).
Unfortunately,
postprocessing,
as analyzed in the framework of variational crimes in
\eqref{e:iE-post},
again entails a CFL restriction.
Postprocessing
is
not required
for the higher order
discretizations
(see Figure \ref{f:nm1} and Table \eqref{t:nm1}),
but
their stability beyond
the trivial range \eqref{e:range}
remains to be verified.

Finally, we have generalized these results
to stochastic PDES driven by affine multiplicative L\'evy noise as
considered in \cite{KirchnerLangLarsson2016}.
By means of $C_0$-semigroups on projective tensor product spaces,
we have
found the condition $C_G = \eqref{e:curlyG1bound} < \infty$
for well-posedness of the deterministic second moment equation~\eqref{e:vf-pde}
in the vector-valued case (see Theorem~\ref{t:stab-mul-pde}),
which is less restrictive than
the smallness assumption
on the multiplicative noise term
made in \cite[Eq.~(5.5)]{KirchnerLangLarsson2016}.
Furthermore, we have discussed stability
of numerical approximations based on
the tensor product Petrov--Galerkin discretizations from \S\ref{s:dis:2-mul} in time,
and standard Galerkin discretizations in space, see Theorem~\ref{t:stab:dis-pde}.
From this, the quasi-optimality estimate \eqref{e:quasiopti-pde} for approximations
generated with the $\textrm{CN}^*_2$ scheme has followed.
Since no postprocessing is necessary
for the $\textrm{CN}^*_2$ discretization (see \S\ref{s:convergence}),
for the sake of brevity, we have focussed
on this method for the quasi-optimality analysis in \S\ref{s:pde:discretization},
and for the numerical experiment in \S\ref{s:pde:numexp}, see Figure~\ref{f:num-pde}.
However, we point out that the definition~\eqref{e:DIAGkh}
of the vector approximate trace product
decouples the disctretizations in space and in time.
Thus, the convergence results of the (postprocessed)
scalar $\textrm{iE}^*_2$ schemes
from \S\ref{s:convergence}
should also readily transfer to the vector-valued situation.

\bibliographystyle{abbrv}
\bibliography{refs}

\section*{Acknowledgment}

The author thanks
Roman Andreev for his contributions and support
during the writing of this article
as well as Stig Larsson and Sonja Cox
for their valuable input.

%

\vfill
\addresseshere



\end{document}